\documentclass[11pt]{article}
\usepackage[utf8]{inputenc}
\usepackage[a4paper, margin=0.8 in]{geometry}
\usepackage{amsmath}
\usepackage{array}
\newcolumntype{P}[1]{>{\centering\arraybackslash}p{#1}}
\newcolumntype{M}[1]{>{\centering\arraybackslash}m{#1}}
\usepackage{xcolor}
\usepackage{amsfonts, amssymb}
\usepackage{amsthm}
\usepackage{esvect}
\usepackage[english]{babel}
\usepackage[ruled,linesnumbered]{algorithm2e}
\usepackage[T1]{fontenc}
\usepackage{multirow}
\usepackage{graphicx}
\usepackage{cite}
\usepackage[numbers,sort]{natbib}
 \usepackage{titlesec}
\titleformat{\section}
  {\normalfont\fontfamily{ptm}\fontsize{11}{11}\bfseries}{\thesection}{1em}{}
\titleformat{\subsection}
  {\normalfont\fontfamily{ptm}\fontsize{10}{11}\bfseries}{\thesubsection}{1em}{}
\titleformat{\subsubsection}
  {\normalfont\fontfamily{ptm}\fontsize{10}{11}\selectfont}{\thesubsubsection}{1em}{}
\newtheorem{theorem}{Theorem}[section]
\newtheorem{corollary}{Corollary}[theorem]
\newtheorem{assumption}[theorem]{Assumption}
\newtheorem{lemma}[theorem]{Lemma}
\newtheorem{proposition}[theorem]{Proposition}

\newtheorem{remark}{Remark}
\newtheorem{example}{Example}
\newtheorem{innercustomgeneric}{\customgenericname}
\providecommand{\customgenericname}{}
\newcommand{\newcustomtheorem}[2]{%
  \newenvironment{#1}[1]
  {%
   \renewcommand\customgenericname{#2}%
   \renewcommand\theinnercustomgeneric{##1}%
   \innercustomgeneric
  }
  {\endinnercustomgeneric}
}

\newcustomtheorem{customcoro}{Corollary}
\usepackage{hyperref}
\hypersetup{
    colorlinks=true,    
    urlcolor=cyan,
    linkcolor=blue,
    citecolor=blue
}
\usepackage{mathtools}

\providecommand{\keywords}[1]
{
  \small	
  \quad \quad \textbf{\textit{Keywords --}} #1
}
\newcommand{\conv}[1]{\text{conv(}{#1}\text{)}}

\newcommand{\exclude}[1]{}

\title{\large Strong valid inequalities for a class of concave submodular minimization problems under cardinality constraints}

\author{ \small Qimeng Yu \quad Simge K\"u\c{c}\"ukyavuz\vspace{0.2cm}  \\ \small Department of Industrial Engineering and Management Sciences \\ \small Northwestern University, Evanston, IL, USA \\ \small \{kim.yu@u.northwestern.edu, simge@northwestern.edu\}}
\date{\small \today} 
\begin{document}
\maketitle

\begin{abstract}
\noindent We study the polyhedral convex hull structure of a mixed-integer set which arises in  a class of  cardinality-constrained  concave submodular minimization problems. This class of problems has an objective function in the form of $f(a^\top x)$, where $f$ is a univariate concave function, $a$ is a non-negative vector, and $x$ is a binary vector of appropriate dimension. Such minimization problems frequently appear in  applications that involve risk-aversion or economies of scale. We propose three classes of strong valid linear inequalities for this convex hull and specify their facet conditions when $a$ has two distinct values. We show how to use these inequalities to obtain valid inequalities for general $a$ that contains multiple values. We further provide a complete linear convex hull description for this mixed-integer set when $a$ contains two distinct values and the cardinality constraint upper bound is two. Our computational experiments on the mean-risk optimization problem demonstrate the effectiveness of the proposed inequalities in a branch-and-cut framework. 
\end{abstract}

\section{Introduction}
\label{sect:intro}
Submodular set functions have received great interest in integer and combinatorial optimization. Many important combinatorial problems and structures, including the set covering problem, the min-cut problem, and matroids, are closely associated with submodular functions. These functions have also found immense utility in applications such as healthcare \citep{adhikari2019fast}, viral marketing \citep{kempe2015maximizing}, and sensor placement \citep{krause2008efficient}. Next, we formally state the definition of submodular functions. \\

Let $N = \{1,2,\dots,n\}$ be a non-empty finite set. We define the power set of $N$ to be $2^N = \{S:S\subseteq N\}$. A set function $g:2^N \rightarrow \mathbb{R}$ is \textit{submodular} if for any $X,Y\in 2^N$, 
\[g(X) + g(Y) \geq g(X\cap Y) + g(X\cup Y). \] 
For any $X\subseteq N$ and $i\in N\backslash X$, $\rho_i(X) := g(X\cup\{i\}) - g(X)$ represents the marginal return to the function value by adding item $i$ to the set $X$. This notion of marginal return provides an alternative definition of submodularity, namely $g$ is submodular if 
\[\rho_i(X) \geq \rho_i(Y)\] for any $X\subseteq Y\subseteq N$ and $i\in N\backslash Y$. Intuitively, this definition implies that all submodular functions possess a diminishing return property. Many studies have established that unconstrained submodular minimization is solvable in polynomial time \citep{lovasz1983submodular, grotschel1981ellipsoid, iwata2001combinatorial, lee2015faster, orlin2009faster}. However, constrained submodular minimization problems are NP-hard in general \citep{svitkina2011submodular}. There exist exceptions to this general observation---a class of submodular functions can be minimized in polynomial time in the presence of a cardinality constraint. We next describe this class of submodular functions in detail. \\

It is known that composing a non-negative modular function with a concave function yields a submodular function. To be more precise, let any $a\in\mathbb{R}^n_+$ and any concave function $f:\mathbb{R}\rightarrow\mathbb{R}$ be given. The function $F$ defined by $F(S) = f\left(\sum_{i\in S} a_i \right)$ for all $S\subseteq N$,
or equivalently $F(x) = f\left(\sum_{i=1}^n a_ix_i\right)$ for all $x\in\{0,1\}^n$, is submodular. The problem of minimizing such a submodular function with respect to a cardinality constraint assumes the form of \eqref{prob:cc_concave_sub_min}:
 
\begin{equation}
\label{prob:cc_concave_sub_min}
\min \left\{  f\left(\sum_{i=1}^n a_ix_i\right) : x\in\{0,1\}^n, \sum_{i=1}^n x_i \leq k \right\}.
\end{equation}
This problem has the following equivalent mixed-integer nonlinear programming formulation:
\[\min \left\{ w : (w, x)\in\mathcal{P}_k^m\right\},  \]
where
\begin{equation}
\label{eq:P}
\mathcal{P}_k^m = \left\{ (w, x)\in\mathbb{R}\times\{0,1\}^n : w \geq f\left(\sum_{i=1}^n a_ix_i\right), \sum_{i=1}^n x_i \leq k \right\}. 
\end{equation}
The superscript $m\in\{1,2,\dots, n\}$ denotes the number of distinct values in $a$, and the subscript $k$ is the cardinality upper bound. In what follows,  we refer to the values in $a$ as weights. This family of problems usually arises in applications that involve risk aversion or economies of scale, such as mean-risk optimization  \citep{atamturk2008polymatroids, atamturk2019lifted} and concave cost facility location \citep{feldman1966warehouse, hajiaghayi2003facility}. Studies \citep{hassin1989maximizing, onn2003convex} have shown that problem \eqref{prob:cc_concave_sub_min} is polynomial-time solvable. This complexity result suggests that a full characterization of $\conv{\mathcal{P}_k^m}$ may be tractable. \\

Inspired by this implication, we take a polyhedral approach to address problem \eqref{prob:cc_concave_sub_min} in this paper.  In seminal work, for unconstrained submodular minimization, \citet{edmonds1970submodular} proposes extended polymatroid inequalities and establishes an explicit linear convex hull description for the epigraph of any submodular function using these inequalities (see also \citep{edmonds2003submodular}).  Since this early work, polyhedral approaches have commonly been adopted in submodular optimization research. Such approaches have unique advantages especially in the presence of additional complicating constraints or when maximizing $f$ leading to NP-hard problems. In this regard, \citet{wolsey1999integer} take a polyhedral approach to tackle unconstrained submodular maximization problems. The authors introduce a class of valid linear inequalities, called submodular inequalities, for the hypograph of any submodular function. This enables the formulation of the problem as  a mixed-integer linear program. This formulation is later strengthened for constrained submodular maximization in \citep{ahmed2011maximizing,yu2017maximizing,shi2020sequence}. By exploiting hidden submodularity, studies including \citep{gomez2018submodularity, atamturk2020submodularity, atamturk2020supermodularity, kilincc2020conic, atamturk2008polymatroids, atamturk2019lifted} improve the formulations of mixed-binary convex quadratic and conic optimization problems. \citet{yu2022drsubmod} consider mixed-integer extensions of submodularity, known as diminishing returns (DR)-submodularity, and give the convex hull of the epigraph of a DR-submodular function under box and monotonicity constraints. 
\citet{atamturk2021submodular}  extend the polyhedral results to general set function minimization, in which the authors rewrite a set function as the difference between two submodular functions and form the outer approximation of the original set function based on the extended polymatroid inequalities and the submodular inequalities of the pair of submodular functions.  For  another generalization---namely, $k$-submodular optimization---where  the objective function is a set function with $k\geq 2$ arguments that maintain submodular properties, \citet{yu2020polyhedral, yu2021exact} provide polyhedral characterizations.  In another direction, recent works \citep{wu2018two, wu2019probabilistic, wu2020exact, kilincc2019joint, xie2019distributionally, zhang2018ambiguous} successfully adopt a polyhedral approach to submodular optimization in  stochastic settings. \\

The polyhedral study closely related to our work is  \cite{yu2017polyhedral}, in which the authors  consider problem \eqref{prob:cc_concave_sub_min}  and obtain a complete description of $\conv{\mathcal{P}_k^1}$ where the weights $a_i$ are identical across all the items $i\in N$. When $m\geq 2$, \citet{yu2017polyhedral} argue that one class of facet-defining inequalities for $\conv{\mathcal{P}_k^m}$ can be obtained using an $O(n^4)$ extreme point enumeration algorithm \citep{atamturk2009submodular}. This class of facets is not sufficient to fully describe $\conv{\mathcal{P}_k^m}$ when $m\geq 2$, and the explicit form of such inequalities is not provided. Instead, the authors approximate the inequality coefficients to give a weaker class of valid inequalities. Despite  the progress made by \cite{yu2017polyhedral} in understanding $\conv{\mathcal{P}_k^1}$, how to fully characterize $\conv{\mathcal{P}_k^m}$ when $m\geq 2$ remains an open problem. Our paper takes the first step to tackle this open problem by analyzing the structure of $\conv{\mathcal{P}_k^2}$, where the vector $a$ contains two distinct values. We further provide valid inequalities for $\conv{\mathcal{P}_k^m}$ where $m\geq 3$. Next we give a summary of our contributions. 

\subsection{Our contributions}
We propose three classes of strong valid linear inequalities for $\conv{\mathcal{P}^2_k}$. We present the explicit forms of these inequalities and specify the conditions under which they are facet-defining for $\conv{\mathcal{P}^2_k}$. We further show that these inequalities, along with the trivial bounds, the cardinality constraint, and a single additional inequality, fully describe $\conv{\mathcal{P}^2_2}$. Our computational experiments on the mean-risk optimization problem demonstrate the effectiveness of our proposed inequalities in a branch-and-cut framework. Moreover, we delineate how these inequalities can be extended to the cases with more than two distinct weights, and how they may be utilized in mixed-binary conic optimization. We also include remarks on the facets of $\conv{\mathcal{P}_k^m}$ when $k\geq 3$ or $m\geq 3$, which reflect the complexities in obtaining the complete linear description of $\conv{\mathcal{P}_k^m}$.

\subsection{Outline}
We structure this paper as follows. In Section \ref{sect:prelim}, we set forth our notation and review two classes of inequalities, namely the \emph{extended polymatroid inequalities} (EPIs) and the \emph{separation inequalities} (SIs). In Sections \ref{sect:EPI} and \ref{sect:sepa}, we exactly lift the aforementioned inequalities and obtain three classes of strong valid linear inequalities for $\conv{\mathcal{P}^2_k}$, which we call the lifted-EPIs, the lower-SIs, and the higher-SIs. Next, in Section \ref{sect:conv}, we provide a linear description of $\conv{\mathcal{P}^2_2}$ using the proposed inequalities and prove its completeness. We explain how to apply the proposed inequalities to the problem instances with three or more distinct weights, as well as how they can be extended for mixed-binary conic optimization problems in Section \ref{sect:ext}. In Section \ref{sect:comp}, we present a computational study on the mean-risk optimization problem with varying cardinality bounds to test the effectiveness of our proposed inequalities when used in a branch-and-cut algorithm. Lastly, in Section \ref{sect:conclude}, we include examples to illustrate the difficulty in constructing the complete linear description for general $\conv{\mathcal{P}^m_k}$.

\section{Preliminaries}
\label{sect:prelim}
\subsection{Notation}
Throughout this paper, $f:\mathbb{R}\rightarrow\mathbb{R}$ is a concave function, and we assume that $f(0)=0$ without loss of generality. To abbreviate set notations, we represent $\{1,2,\dots, j\}$ by $[j]$ for any integer $j\geq 1$, and we use the convention that $[0] = \emptyset$. In addition, we let $[i,j] = \{i,i+1, \dots, j\}$ for $1\leq i \leq j$; by convention, $[i,i] = \{i\}$. \\

Let $N=[n]$ be a non-empty finite ground set. We let $F:2^N\rightarrow\mathbb{R}$ be the function such that, given $a\in\mathbb{R}^n_+$, $F(S) = f\left(\sum_{i\in S} a_i \right)$ for any $S\subseteq N$. We note that for any $S\subseteq N$, there exists a unique characteristic vector $x^S\in\{0,1\}^n$ such that $x^S_i = 1$ for $i\in S$ and $x^S_i = 0$ otherwise. On the other hand, with any $x\in\{0,1\}^n$, we may recover exactly one $S^x = \{i\in N:x_i=1\}$. Thus $f$ and $F$ are used interchangeably in later discussions. The lemma below summarizes a crucial property of $f$. 
\begin{lemma}
\label{lemma:f_concave} For any $d\in\mathbb{R}_+$ and $y_1,y_2\in\mathbb{R}$ such that $y_1\leq y_2$, 
\[f(y_1+d) - f(y_1) \geq f(y_2+d) - f(y_2)\]
in any concave function $f:\mathbb{R}\rightarrow\mathbb{R}$. 
\end{lemma}
\begin{proof}
First we consider the case where $y_1\geq 0$. Let $g:\{0,1\}^3\rightarrow \mathbb{R}$ be a function defined by $g(x) = f(y_1x_1+(y_2-y_1)x_2 + dx_3)$. Since $f$ is concave and $y_1,y_2-y_1,d\geq 0$, we know that $g$, being the composition of a concave function and a non-negative modular function, is submodular. To simplify the notation, we use the alternative form of $g$, namely $G:2^{\{1,2,3\}}\rightarrow\mathbb{R}$. In particular, $G(\{1\}) = g([1,0,0]) = f(y_1)$, $G(\{1,2\}) = g([1,1,0]) = f(y_2)$, $G(\{1,3\})=g([1,0,1]) = f(y_1+d)$ and $G(\{1,2,3\})=g([1,1,1]) = f(y_2+d)$. Then 
\begingroup
\allowdisplaybreaks
\begin{align*}
f(y_1+d) - f(y_1) -[ f(y_2+d) - f(y_2) ] & = G(\{1,3\}) - G(\{1\}) - [G(\{1,2,3\}) - G(\{1,2\})] \\
& = \rho_3(\{1\}) - \rho_3(\{1,2\}) \geq 0.
\end{align*}
\endgroup

If $y_1 < 0$, then we define a function $\hat{f}:\mathbb{R}\rightarrow\mathbb{R}$ such that $\hat{f}(z) = f(z+y_1)$ for all $z\in\mathbb{R}$. This function $\hat{f}$ is $f$ shifted to the right by $|y_1|$, so it is also concave. We notice that $f(y_1)=\hat{f}(0)$, $f(y_2)=\hat{f}(y_2-y_1)$, $f(y_1+d) = \hat{f}(d)$ and $f(y_2+d) = \hat{f}(y_2-y_1+d)$. Thus our goal now is to show that $\hat{f}(d) - \hat{f}(0) \geq \hat{f}(y_2-y_1+d) - \hat{f}(y_2-y_1)$. This relation is true according to the analysis of the previous case, which completes the proof.
\end{proof}

For $\mathcal{P}_k^2$, we denote the two distinct weights in $a$ by $a_L$ and $a_H$, such that $0\leq a_L < a_H$. We let $\mathcal{I}_L = \{i\in N : a_i = a_L\}$ and $\mathcal{I}_H = \{i\in N : a_i = a_H\}$. Suppose the items in $N$ are permuted according to $\delta = (\delta_1,\delta_2,\dots, \delta_n)$. We define $L^t$, for $0\leq t\leq |\mathcal{I}_L|$, to be the set of the first $t$ lower-weighted items according to $\delta$. Similarly, we let $H^s$, for $0\leq s\leq |\mathcal{I}_H|$, be the set of the first $s$ higher-weighted items consistent with $\delta$. By convention, $L^0 = H^0 = \emptyset$. \\

Next, we review two useful classes of inequalities, namely EPIs \cite{edmonds2003submodular} and SIs \cite{yu2017polyhedral}. 

\subsection{Extended polymatroid inequalities (EPIs)}
Let $G:2^N\rightarrow\mathbb{R}$ be any submodular set function defined over the ground set $N=[n]$, with the equivalent form $g:\{0,1\}^n\rightarrow\mathbb{R}$. Without loss of generality, we assume that $G(\emptyset) = g(\mathbf{0}) = 0$. Given any permutation $\delta = (\delta_1, \delta_2,\dots,\delta_n)$ of $N$, the corresponding EPI is  
\begin{equation}
\label{eq:EPI}
w \geq \sum_{i=1}^n \rho_{\delta_i} x_{\delta_i}, 
\end{equation}
where $\rho_{\delta_1}= G(\{\delta_1\})$ and $\rho_{\delta_i} = G(\{\delta_1,\dots,\delta_i\}) - G(\{\delta_1,\dots,\delta_{i-1}\})$ for $i\in[2, n]$. In the \emph{unconstrained} set \[\mathcal{Q} = \left\{ (w, x)\in\mathbb{R}\times\{0,1\}^n : w \geq g(x) \right\},\] EPIs are known to be facet-defining for $\conv{\mathcal{Q}}$. In fact, $\conv{\mathcal{Q}}$ is fully described by the trivial inequalities $0\leq x_i\leq 1$, $i\in[n]$, and all the EPIs \cite{edmonds2003submodular}. \\

In our problem context, the EPIs are facet-defining for $\conv{\mathcal{P}^2_k(S)}$, where 
\begin{equation}
\label{eq:EPI_base_set}
\mathcal{P}^2_k(S) = \left\{ (w, x)\in\mathbb{R}\times\{0,1\}^S : w \geq f\left(\sum_{i\in S} a_ix_i\right), \sum_{i\in S} x_i \leq k \right\}
\end{equation} for any $S\subseteq N$ with $|S| \leq k$. This is because the cardinality constraint trivially holds for such $S$. In Section \ref{sect:EPI}, we lift the EPIs with respect to the variables $x_i$ for all $i\in N\backslash S$.

\subsection{Separation inequalities (SIs)}
SIs are strong valid linear inequalities for $\conv{\mathcal{P}^1_k}$ proposed in \cite{yu2017polyhedral}. In this case, we have $a_i = \alpha$ for all $i\in [n]$ given some $\alpha \in\mathbb{R}_+$, and 
\[\mathcal{P}^1_k = \left\{ (w, x)\in\mathbb{R}\times\{0,1\}^n : w \geq f\left(\alpha \sum_{i=1}^n x_i\right), \sum_{i=1}^n x_i \leq k \right\}. \]  Given any permutation of $N$, $\delta = (\delta_1,\delta_2,\dots, \delta_n)$, and a fixed parameter $i_0\in\{0,1,\dots, k-1\}$, an SI is defined by  
\begin{equation}
\label{eq:sepa_coeff}
w\geq \sum_{i=1}^{i_0} \rho_{\delta_i}x_{\delta_i} + \sum_{i=i_0+1}^{n} \psi x_{\delta_i}. 
\end{equation} Here  \[\psi = \frac{f(k \alpha) - f(i_0 \alpha)}{k-i_0},\]
and $\rho_{\delta_i}$ is the EPI coefficient $f(i \alpha) - f((i-1)\alpha)$. The authors further show that the SIs, together with $\sum_{i=1}^n x_i \leq k$ and $0\leq x_i\leq 1$ for $i\in N$, fully describe $\conv{\mathcal{P}^1_k}$. In our problem context, the same convex hull characterization holds for $\conv{\mathcal{P}^1_k(\mathcal{I}_L)}$ and $\conv{\mathcal{P}^1_k(\mathcal{I}_H)}$, where
\begin{equation}
\label{eq:L_P1k}
\mathcal{P}^1_k(\mathcal{I}_L) = \left\{ (w, x)\in\mathbb{R}\times\{0,1\}^{\mathcal{I}_L} : w \geq f\left(a_L \sum_{i\in\mathcal{I}_L} x_i\right), \sum_{i\in\mathcal{I}_L} x_i \leq k \right\},
\end{equation} and 
\begin{equation}
\label{eq:H_P1k}
\mathcal{P}^1_k(\mathcal{I}_H) = \left\{ (w, x)\in\mathbb{R}\times\{0,1\}^{\mathcal{I}_H} : w \geq f\left(a_H \sum_{i\in\mathcal{I}_H} x_i\right), \sum_{i\in\mathcal{I}_H} x_i \leq k \right\}.
\end{equation}
In Section \ref{sect:sepa}, we lift the SIs of $\conv{\mathcal{P}^1_k(\mathcal{I}_L)}$ and $\conv{\mathcal{P}^1_k(\mathcal{I}_H)}$, to obtain two classes of strong valid linear inequalities for $\conv{\mathcal{P}^2_k}$. \\

As mentioned earlier, \citet{yu2017polyhedral} give an $O(n^4)$ algorithm to exactly lift the EPIs for the multi-weighted case. However, the algorithmic approach does not yield explicit forms of the lifting coefficients, which hinders the effectiveness of this algorithmic approach in a branch-and-cut scheme. Due to this complexity, the authors give approximate coefficients of the lifted EPIs. In contrast, we directly describe the optimal solutions to the lifting problems given both EPIs and SIs as the base inequalities for the problems involving two weights. Such a closed-form description of sequence-dependent lifting coefficients is generally non-trivial. Furthermore, this description paves the path for the effective use of the resulting inequalities in a branch-and-cut framework as evidenced by our computational experiments.

\section{Exact Lifting of Extended Polymatroid Inequalities}
\label{sect:EPI}
The goal of this section is to lift the EPIs \eqref{eq:EPI} and derive a class of strong valid linear inequalities for $\conv{\mathcal{P}^2_k}$. We call this new class of inequalities the \textit{lifted-EPIs}.  \\

For any permutation $\delta$ of $N$, we can re-index $N$ such that $\delta$ is the natural order $(1,2,\dots,n)$. Let $S$ be any subset of $N$ such that $|S| = k$. Without loss of generality, we assume that $S=[k]$. This can also be achieved by re-indexing. Let $d_H = |\mathcal{I}_H\backslash [k-1]|$ and $d_L = |\mathcal{I}_L\cap [k-1]|$. We use $\mathcal{H} = (\mathcal{H}_1, \mathcal{H}_2, \dots, \mathcal{H}_{d_H})$ to denote the permutation of $\mathcal{I}_H\backslash [k-1]$ that is consistent with $\delta$. We let $\mathcal{L} = (\mathcal{L}_1, \mathcal{L}_2, \dots, \mathcal{L}_{d_L})$ be the permutation of $\mathcal{I}_L\cap [k-1]$ that is also consistent with $\delta$. For $q\in [d_H]$, $\mathcal{H}(q) = \{\mathcal{H}_1, \dots, \mathcal{H}_q\}$. If $q\leq 0$, then $\mathcal{H}(q) = \emptyset$. Similarly, we let $\mathcal{L}(q) = \{\mathcal{L}_1, \dots, \mathcal{L}_q\}$ for $q\in [d_L]$; $\mathcal{H}(q) = \emptyset$ when $q\leq 0$. The set $L^t$ is the same as $\mathcal{L}(t)$ for any $t\in [d_L]$. However, $L^t$ is defined for $t > d_L$ as well, while $\mathcal{L}(\cdot)\subseteq [k-1]$. The next example clarifies the new notation.

\begin{example}Suppose $k=3$, and $N = [5]$ such that $\mathcal{I}_L = \{1,3\}$ and $\mathcal{I}_H = \{2,4,5\}$. Given $\delta = (1,2,3,4,5)$, $\mathcal{H}=(4,5)$ and $d_H=2$. Meanwhile $\mathcal{L}=(1)$ and $d_L = 1$. In addition, $\mathcal{H}(2) = \{4,5\}$, $\mathcal{L}(1) = \{1\}$ and $\mathcal{L}(0) = \emptyset$.  
\end{example}

With the specified indexing,
\[\mathcal{P}^2_k(S) = \left\{ (w, x)\in\mathbb{R}\times\{0,1\}^k : w \geq f\left(\sum_{i=1}^k a_ix_i\right), \sum_{i=1}^k x_i \leq k \right\}. \]
This set is essentially $\mathcal{P}^2_k$ with $x_i$ fixed at 0 for all $i\in N\backslash S = [k+1, n]$. Let a base EPI $w \geq \sum_{i=1}^k \rho_i x_i$ associated with the natural ordering of $S$ be given. The coefficient $\rho_i$ is $\rho_i([i-1])$ for any $i\in[k]$ to be precise. Lifting this base inequality with the variables $x_i$, $i\in N\backslash S$, we can construct a valid inequality for $\conv{\mathcal{P}^2_k}$ in the form of 
\begin{equation}
\label{eq:EPI-lifted}
 w\geq \sum_{i=1}^k \rho_i x_i + \sum_{i=k+1}^n \xi_i x_i. 
\end{equation} which is what we call a lifted-EPI.  \\

In an intermediate step of lifting $x_j$ for $j\in[k+1, n]$, we derive a facet-defining inequality $w \geq \sum_{i=1}^k \rho_i x_i + \sum_{i = k+1}^j \xi_i x_i$ for the convex hull of the polyhedron 
\[\mathcal{P}^2_k([j]) = \left\{ (w, x)\in\mathbb{R}\times\{0,1\}^j : w \geq f\left(\sum_{i=1}^j a_ix_i\right), \sum_{i=1}^j x_i \leq k \right\}. \] The coefficient $\xi_j$ is the optimal objective value of the $j$-th lifting problem \eqref{eq:EPI_lifting_prob}.
\begingroup
\allowdisplaybreaks
\begin{equation}
\label{eq:EPI_lifting_prob}
\begin{aligned}
\xi_j := \min \hspace{0.2cm} & w - \sum_{i=1}^{k} \rho_i x_i - \sum_{i = k+1}^{j-1} \xi_i x_i&& \\
\textrm{s.t.} \quad & w\geq f\left(a_j + \sum_{i = 1}^{j-1} a_i x_i\right), &&\\
& \sum_{i=1}^{j-1} x_i \leq k-1, && \\
& x \in \{0,1\}^{j-1}. && 
\end{aligned}
\end{equation}
\endgroup

In fact, every lifted-EPI is identical with $w\geq \sum_{j=1}^n \zeta_jx_j$, in which $\zeta_j$ is the optimal objective of the $j$-th lifting problem 
 \eqref{eq:EPI_lifting_simplified} for $j\in[n]$. This observation is formalized in Lemma \ref{lemma:EPI_alternative}. 
\begingroup
\allowdisplaybreaks
\begin{subequations}
\label{eq:EPI_lifting_simplified}
\begin{alignat}{2}
\zeta_j := \min \hspace{0.2cm} & w - \sum_{i = 1}^{j-1} \zeta_i x_i&& \\
\textrm{s.t.} \quad & w\geq f\left(a_j + \sum_{i = 1}^{j-1} a_i x_i\right), && \label{constr:w}\\
& \sum_{i=1}^{j-1} x_i \leq k-1, && \label{constr:card} \\
& x \in \{0,1\}^{j-1}. && 
\end{alignat}
\end{subequations}
\endgroup

\begin{lemma}
\label{lemma:EPI_alternative}
In the base EPI $w \geq \sum_{i=1}^k \rho_i x_i$, $\rho_j = \zeta_j$ for all $j\in [k]$, where $\zeta_j$ is the optimal objective value of the $j$-th lifting problem \eqref{eq:EPI_lifting_simplified}.
\end{lemma}
\begin{proof}
We observe that for any $j\in[k]$, constraint \eqref{constr:card} naturally holds. When $j=1$, $w$ is the only decision variable in problem \eqref{eq:EPI_lifting_simplified}. To minimize $w$, constraint \eqref{constr:w} must be tight at the optimal solution. Thus $\zeta_1 = f(a_1) = \rho_1$. If $k=1$ then the proof is complete. Now suppose $k\geq 2$. When $j=2$, $\zeta_2 = \min \{f(a_1+a_2)-f(a_1), f(a_2)\}$, or equivalently, $\min\{F(\{1,2\})-F(\{1\}), F(\{2\})\}$. By submodularity of $F$, $F(\{1\}) + F(\{2\}) \geq F(\{1,2\})$. Therefore, $\zeta_2 = F(\{1,2\})-F(\{1\}) = \rho_2$. Now we have settled two base cases. For a strong induction, our induction hypothesis is that $\rho_j$ is the optimal objective of the $j$-th problem \eqref{eq:EPI_lifting_simplified} for all $j \in [J-1]$, where $J-1\in [k-1]$. Now we characterize the optimal solution to the $J$-th problem \eqref{eq:EPI_lifting_simplified}. The optimal objective value $\zeta_{J}$ is $\min_{Q\subseteq [J-1]} F(Q\cup\{J\}) - \sum_{i\in Q}\zeta_i$. Let $Q$ be an arbitrary subset of $[J-1]$.
\begingroup
\allowdisplaybreaks
\begin{align*}
F(Q\cup\{J\}) + \sum_{i\in [J-1] \backslash Q} \zeta_i 
& = F(Q\cup\{J\}) + \sum_{i\in [J-1] \backslash Q} \rho_i \quad\quad (\text{by induction hypothesis})&&  \\
& = F(Q\cup\{J\}) + \sum_{i\in [J-1] \backslash Q} \rho_i([i-1]) \quad\quad (\text{by definition of $\rho_i$}) &&  \\
& \geq F(Q\cup\{J\}) + \sum_{i\in [J-1] \backslash Q} \rho_i (Q\cup [i-1])  \quad\quad (\text{by submodularity of $F$}) && \\
& = F(Q\cup\{J\}) + \sum_{i\in [J-1] \backslash Q} [F(Q\cup [i]) - F(Q\cup [i-1]) ] && \\
& = F(Q\cup\{J\}) + F(Q\cup [J-1]) - F(Q\cup\emptyset) && \\
& = F(Q\cup\{J\}) + F([J-1]) - F(Q) && \\
& = F([J-1]) + \rho_{J}(Q) && \\
& \geq F([J-1]) + \rho_{J}([J-1])\quad\quad (\text{by submodularity of $F$}) && \\
& = F([J]). && 
\end{align*}
\endgroup

It follows that 
\begingroup
\allowdisplaybreaks
\begin{align*}
F(Q\cup\{J\}) - \sum_{i\in Q} \zeta_i 
& = F(Q\cup\{J\}) + \sum_{i\in [J-1] \backslash Q} \zeta_i - \sum_{i\in[J-1]} \zeta_i  \\
& \geq F([J]) - \sum_{i\in[J-1]} \zeta_i.
\end{align*}
\endgroup

Since the choice of $Q$ is arbitrary, $\zeta_{J} = F([J]) - \sum_{i\in [J-1]} \zeta_i = \rho_{J}$. By strong induction, we conclude that $\rho_j = \zeta_j$ for all $j\in[k]$. 
\end{proof}

Lemma \ref{lemma:EPI_alternative} shows that all the coefficients in a lifted-EPI are the optimal objective values of the corresponding lifting problems \eqref{eq:EPI_lifting_simplified}. This observation enables us to compare $\zeta_j$ across all $j\in [n]$. Lemma \ref{lemma:descending_zeta} captures a descending property of these coefficients. 

\begin{lemma}
\label{lemma:descending_zeta}
Let $w\geq \zeta ^\top x$ be a lifted-EPI associated with $S = [k]$. If $1\leq j_1 < j_2\leq n$ satisfy $a_{j_1} = a_{j_2}$, then $\zeta_{j_1} \geq \zeta_{j_2}$. 
\end{lemma}
\begin{proof}
By Lemma \ref{lemma:EPI_alternative}, we can view the variables $x_i$, $i\in S$, as the first $k$ variables to be lifted. Since $a_{j_1} = a_{j_2}$, this result follows from Proposition 1.3 on page 264 of \cite{wolsey1999integer}, which states that the lifting coefficients are non-increasing with respect to the order in which the variables are lifted. 
\end{proof}

Before stating the explicit form of any lifted-EPI, we introduce additional notation and make more observations about the lifting problem \eqref{eq:EPI_lifting_simplified}. In the $j$-th problem \eqref{eq:EPI_lifting_simplified}, suppose $x \in \{0,1\}^{j-1}$ satisfies $\sum_{i=1}^n x_i \leq k-1$. We denote the support of $x$ by $X = \{i\in [j-1] :  x_i = 1\}$. Since the objective is minimized, we attain the lowest objective value given $x$ when constraint \eqref{constr:w} is tight. We represent the corresponding objective value by $\zeta_j^X$. In other words, 
\[\zeta_j^X = f\left(a_j + \sum_{i\in X} a_i\right) - \sum_{i\in X}\zeta_i, \]
for any feasible $x \in \{0,1\}^{j-1}$. Then \[\zeta_j = \min_{X\subseteq[j-1], |X|\leq k-1} \zeta_j^X.\]

We observe that $N=[k-1]\cup (\mathcal{I}_L\backslash [k-1]) \cup (\mathcal{I}_H\backslash [k-1])$, where $[k-1]$, $\mathcal{I}_L\backslash [k-1]$ and $\mathcal{I}_H\backslash [k-1]$ are pairwise disjoint. Recall that $|\mathcal{I}_H\backslash [k-1]| = d_H$, and $\mathcal{I}_H\backslash [k-1] = \mathcal{H}(d_H) = \{\mathcal{H}_1,\dots, \mathcal{H}_{d_H}\}$. Thus every $j\in \mathcal{I}_H\backslash [k-1]$ is $\mathcal{H}_i$ for some unique $i\in[d_H]$. \\

Lemma \ref{lemma:no_k_in_sol} shows that, if we restrict the solutions to the $j$-th lifting problem \eqref{eq:EPI_lifting_simplified} by fixing $x_i=0$ for all $i\in[k,j-1]$, then $\zeta_j^{[k-1]}$ is the lowest attainable objective value. 
\begin{lemma}
\label{lemma:no_k_in_sol}
Let any $j\in[k+1, n]$ be given. {For all $Q\subseteq [k-1]$}, $\zeta_j^Q\geq \zeta_j^{[k-1]}$. 
\end{lemma}
\begin{proof}
We choose an arbitrary $Q\subseteq [k-1]$ such that $|Q|\leq k-1$. Then 
\begingroup
\allowdisplaybreaks
\begin{align*}
\zeta_{j}^Q & =  f\left(a_j + \sum_{i\in Q} a_i\right) - \sum_{i\in Q}\rho_i \\
& =  F(Q\cup\{j\}) - \sum_{i\in Q} \rho_i([i-1]) \\
& \geq F(Q\cup\{j\}) - \sum_{i\in Q} \rho_i(Q\cap [i-1]) \quad\quad \text{($F$ is submodular)}\\
& = F(Q\cup\{j\}) - F(Q) \\
& \geq F([k-1]\cup\{j\}) - F([k-1]) \quad\quad \text{($F$ is submodular)}\\ 
& = f\left(\sum_{i\in [k-1]}a_i + a_{j}\right) - \sum_{i\in [k-1]}\rho_i \\
& = \zeta_{j}^{[k-1]}. 
\end{align*}
\endgroup 
\end{proof}

The next lemma shows that, in any lifting problem \eqref{eq:EPI_lifting_simplified}, among all the feasible supports with exactly $t$ lower-weighted items and $s$ higher-weighted items, $L^t\cup H^s$ has the lowest objective value. 
\begin{lemma}
\label{lemma:general_Q_form}
Let any $j\in[k,n]$ and fixed integers $0\leq t\leq |\mathcal{I}_L|$, $0\leq s\leq |\mathcal{I}_H|$ with $t+s\leq k-1$ be given. For any $Q\subseteq [j-1]$, such that $|Q\cap\mathcal{I}_L| = t$ and $|Q\cap\mathcal{I}_H| = s$, $\zeta_j^Q \geq \zeta_j^{L^t\cup H^s}$. 
\end{lemma}
\begin{proof}
Consider any $Q$ with the stated properties. It satisfies $|Q| = s + t \leq k-1$, so $Q$ corresponds to a feasible solution to the $j$-th lifting problem \eqref{eq:EPI_lifting_simplified}.  Lemma \ref{lemma:descending_zeta} suggests that $\sum_{i\in L^t}\zeta_i$ is the sum of the highest $t$ lifting coefficients for the lower-weighted items. There are $t$ lower-weighted items in $Q\cap\mathcal{I}_L$ as well, so $\sum_{i\in L^t}\zeta_i \geq \sum_{i\in Q\cap\mathcal{I}_L}\zeta_i$. Similarly,  $\sum_{i\in H^s}\zeta_i \geq \sum_{i\in Q\cap\mathcal{I}_H}\zeta_i$. Thus 
\begingroup
\allowdisplaybreaks
\begin{align*}
\zeta_j^Q & = f\left(a_j + ta_L + sa_H\right) - \sum_{i\in Q\cap\mathcal{I}_L}\zeta_i -  \sum_{i\in Q\cap\mathcal{I}_H}\zeta_i \\
& \geq f\left(a_j + ta_L + sa_H\right) - \sum_{i\in L^t}\zeta_i -  \sum_{i\in H^s}\zeta_i \quad \quad \text{(by Lemma \ref{lemma:descending_zeta}, as discussed above)} \\
& = \zeta_j^{L^t\cup H^s}. 
\end{align*}
\endgroup 
\end{proof}
We may infer from this lemma that an optimal support for any lifting problem \eqref{eq:EPI_lifting_simplified} assumes the form $L^t\cup H^s$ for some $t$ and $s$. In Lemma \ref{lemma:same_a}, we provide the optimal solution to the $j$-th lifting problem, when all the items in $[j-1]$ have the same weight.

\begin{lemma}
\label{lemma:same_a}
Let any $j\in[k,n]$ be given. If $[j-1]\subseteq\mathcal{I}_L$, or $[j-1]\subseteq\mathcal{I}_H$, then $\zeta_j = \zeta_j^{[k-1]}$ in the $j$-th lifting problem \eqref{eq:EPI_lifting_simplified}. 
\end{lemma}
\begin{proof}
Without loss of generality, suppose $[j-1]\subseteq\mathcal{I}_L$. Lemma \ref{lemma:general_Q_form} implies that $\arg\min_{Q\subseteq [j-1], |Q|\leq k-1}\zeta_j^Q$ has the form of $[t]$ for some $0\leq t\leq k-1$. We notice that for any such set, 
\begingroup
\allowdisplaybreaks
\begin{align*}
\zeta_{j}^{[t]} & =  f\left(a_j + t a_L\right) - \sum_{i\in [t]}\zeta_i \\
& = F(\{j\}\cup [t]) - \sum_{i\in [t]}\rho_i \\
& = F(\{j\}\cup [t]) - F([t])\\
& \geq F(\{j\}\cup[k-1]) - F([k-1]) \quad \quad \text{(by submodularity of $F$)} \\
& = \zeta_j^{[k-1]}.
\end{align*}
\endgroup Therefore, for any $Q\subseteq [j-1]$ such that $|Q|\leq k-1$, $\zeta_j^Q \geq \zeta_j^{[k-1]}$. We conclude that $\zeta_j = \zeta_j^{[k-1]}$. The case when $[j-1]\subseteq\mathcal{I}_H$ follows similarly.
\end{proof}

In Proposition \ref{prop:EPI_lifted_coeff}, we present the explicit form of any lifted-EPI.

\begin{proposition}
\label{prop:EPI_lifted_coeff}
A lifted-EPI assumes the form $w\geq \sum_{i=1}^n \zeta_i x_i$, where 
\[
\zeta_j = 
\begin{cases}
\rho_j, & \text{ if } j\in[k-1], \\
 & \\
\zeta_j^{[k-1]}, & \text{ if } j\in\mathcal{I}_L\backslash [k-1], \\
 & \\
\min\left\{ \zeta_{\mathcal{H}_{i-1}},  \zeta_j^{\mathcal{H}(\min\{i-1, d_L\}) \cup \mathcal{L}(d_L-i+1)\cup (\mathcal{I}_H\cap [k-1])}\right\}, & \text{ if } j = \mathcal{H}_i, i\in [d_H],  
\end{cases}
\] and $\zeta_{\mathcal{H}_{0}} =   \zeta_j^{\mathcal{H}(\min\{0, d_L\}) \cup \mathcal{L}(d_L-0)\cup (\mathcal{I}_H\cap [k-1])} =  \zeta_j^{[k-1]}$.
\end{proposition}

Before we prove Proposition \ref{prop:EPI_lifted_coeff}, a remark is in order. 
\begin{remark}
\label{remark:understand_lifted_EPI}
Proposition \ref{prop:EPI_lifted_coeff} allows us to efficiently derive the lifting coefficients in a sequential fashion. Here we provide some intuition behind the proposed coefficients. Given the base EPI, $\zeta_j = \rho_j$ for $j\in [k]$. Thus the first case in Proposition \ref{prop:EPI_lifted_coeff} when $j\in [k-1]$ naturally follows. Next we verify $\zeta_k = \rho_k$ in the construction of Proposition \ref{prop:EPI_lifted_coeff}. When $k\in\mathcal{I}_L$, $\zeta_k = \zeta_k^{[k-1]} = \rho_k$ in the second case. When $k\in\mathcal{I}_H$, then it falls under the third case, where $\zeta_k = \min\{\zeta_k^{[k-1]}, \zeta_k^{[k-1]}\} = \rho_k$. Therefore, $\zeta_k = \rho_k$ is satisfied by the proposed construction. Such division of cases is designed for a conciser proof by strong induction. \\

Now suppose $j\geq k+1$. The second case in Proposition \ref{prop:EPI_lifted_coeff} states that, when $j$ is a lower-weighted item, the support of the optimal solution to the corresponding lifting problem \eqref{eq:EPI_lifting_simplified} is $[k-1]$. This implies that $\zeta_j$ is a constant for all such $j$. That is, lifting is sequence independent for $j\in\mathcal{I}_L$. On the other hand, if $j\in\mathcal{I}_H$, then $j = \mathcal{H}_i$ for some $i\in[d_H]$. This means that $j$ is the $i$-th higher-weighted item strictly after $k-1$ in the fixed permutation. In this case, $\zeta_j$ is the minimum of two candidates. The first candidate is $\zeta_{j'}$ where $j' = \mathcal{H}_{i-1}$ is the higher-weighted item right before $j$ in the given permutation. The coefficient $\zeta_{j'}$ has already been obtained before computing $\zeta_j$ because $j'$ comes before $j$ in the lifting sequence. The second candidate has a support set $\mathcal{H}(\min\{i-1, d_L\}) \cup \mathcal{L}(d_L-i+1\})\cup (\mathcal{I}_H\cap [k-1])$ which always has cardinality $k-1$. Intuitively, this set is constructed by replacing the last $i-1$ lower-weighted items in $[k-1]$ with the first $i-1$ higher-weighted items strictly after $k-1$. If $i-1\geq d_L$, then this support set is $\mathcal{H}(d_L)\cup (\mathcal{I}_H\cap [k-1])$ which is the set of the first $k-1$ higher-weighted items.  
\end{remark}

Next, we present a proof by strong induction to show that the proposed lifted-EPI coefficients are indeed the optimal objective values in the lifting problems \eqref{eq:EPI_lifting_simplified}. The correctness of case 1, when $j \in [k-1]$, in Proposition \ref{prop:EPI_lifted_coeff} is immediate from the base EPI. It suffices to show that when $j \geq k$, cases 2 and 3 in Proposition \ref{prop:EPI_lifted_coeff} are also correct. For a strong induction, we use $j=k$ and $j=k+1$ as our base cases. Remark \ref{remark:understand_lifted_EPI} has cleared the case of $j=k$. Thus it suffices to examine the case of $j=k+1$. Once we settle the base cases, we show the correctness of $\zeta_J$ for some $J \geq k+2$ given the induction hypothesis that $\zeta_j$'s are correct for all $j \in [k,J-1]$. After that, Proposition \ref{prop:EPI_lifted_coeff} is formally established. \\

Lemma \ref{lemma:l_base_case} examines the base case of $j=k+1$ when $k+1\in\mathcal{I}_L$. 
\begin{lemma}
\label{lemma:l_base_case}
When $k+1\in\mathcal{I}_L$, $\zeta_{k+1} = \zeta_{k+1}^{[k-1]}$. 
\end{lemma}
\begin{proof}
For $i\in [k]$, we know that $\zeta_i = \rho_i$. Consider any $Q\subseteq [k]$ with $|Q| \leq k-1$. Such a set $Q$ is the support of any feasible solution $x$ to the $(k+1)$-th lifting problem \eqref{eq:EPI_lifting_simplified}. If $k\notin Q$, then Lemma \ref{lemma:no_k_in_sol} applies. For all such $Q$, $\zeta_j^Q\geq \zeta_j^{[k-1]}$. On the other hand, suppose $k\in Q$. We denote $Q\backslash \{k\}$ by $Q'$ and note that $Q'\subseteq [k-1]$, $|Q'|\leq k-2$. Thus $\sum_{i\in Q'}a_i \leq \sum_{i\in [k-1]}a_i - a_L$. In this case, 
\begingroup
\allowdisplaybreaks
\begin{align*}
\zeta_{k+1}^Q & =  f\left(\sum_{i\in Q'} a_i + a_k + a_{k+1}\right) - \sum_{i\in Q'}\rho_i - \rho_k \\
& =  F(Q'\cup\{k,k+1\}) - \sum_{i\in Q'}\rho_i([i-1]) - \rho_k  \\
& \geq F(Q'\cup\{k,k+1\}) - \sum_{i\in Q'}\rho_i(Q'\cap [i-1]) - \rho_k  \quad\quad \text{($F$ is submodular)}\\
& = F(Q'\cup\{k,k+1\}) - F(Q') - \rho_k  \\
& = F(Q'\cup\{k,k+1\}) - F(Q'\cup\{k+1\}) - \rho_k +  F(Q'\cup\{k+1\})  - F(Q')\\
& = f\left(\sum_{i\in Q'} a_i + a_k + a_L\right) - f\left(\sum_{i\in Q'} a_i + a_L\right) - \left[ f\left(\sum_{i\in [k-1]} a_i + a_k \right) - f\left(\sum_{i\in [k-1]} a_i \right)\right] \\
& \quad +  F(Q'\cup\{k+1\})  - F(Q')\\
& \geq  F(Q'\cup\{k+1\})  - F(Q') \quad\quad \text{($f$ is concave and $\sum_{i\in Q'} a_i + a_L \leq \sum_{i\in [k-1]} a_i$)} \\
& \geq F([k-1]\cup\{k+1\}) - F([k-1]) \quad\quad \text{($F$ is submodular)}\\
& = f\left(\sum_{i\in [k-1]}a_i + a_{k+1}\right) - \sum_{i\in [k-1]}\rho_i \\
& = \zeta_{k+1}^{[k-1]}. 
\end{align*}
\endgroup 

Therefore, for any $Q\subseteq [k]$ such that $|Q| \leq k-1$, $\zeta_{k+1}^Q \geq \zeta_{k+1}^{[k-1]}$. It follows that $\zeta_{k+1} = \zeta_{k+1}^{[k-1]}$.
\end{proof}

We continue to explore the base case of $j=k+1$ when $k+1\in\mathcal{I}_H$. Three scenarios are possible in this case: \begin{enumerate}
\item[(1)] $k \in \mathcal{I}_L$; 
\item[(2)] $k \in \mathcal{I}_H$ and $d_L = |\mathcal{I}_L\cap [k-1]| \geq 1$;
\item[(3)] $k \in \mathcal{I}_H$ and $d_L = 0$; in other words, $[k]\subseteq \mathcal{I}_H$. 
\end{enumerate}
Lemmas \ref{lemma:h_base_case_special1}, \ref{lemma:h_base_case_special2} and \ref{lemma:h_base_case_special3} address these three scenarios respectively.

\begin{lemma}
\label{lemma:h_base_case_special1}
If $k+1\in\mathcal{I}_H$ and $k\in\mathcal{I}_L$, then $\zeta_{k+1} = \zeta_{k+1}^{[k-1]}$.
\end{lemma}
\begin{proof}
Consider any $Q\subseteq [k]$ such that $|Q|\leq k-1$. For every such $Q$ that does not contain $k$, $\zeta^Q_{k+1}\geq \zeta_{k+1}^{[k-1]}$ due to Lemma \ref{lemma:no_k_in_sol}. Now we consider any $Q\ni k$ with $|Q|\leq k-1$. Let $Q' = Q\backslash \{k\}$. We know that $Q'\subseteq [k-1]$ and $|Q'|\leq k-2$. Let $u$ be $\arg\min \{a_i : i\in [k-1]\}$, which means that if $\mathcal{I}_L\cap [k-1] = \emptyset$ then $u$ is any $i\in[k-1]$; otherwise, $u$ is any $i\in\mathcal{I}_L\cap [k-1]$. By this choice of $u$, $a_L\leq a_u \leq a_i$ for all $i\in [k-1]$. Let $T = [k-1]\backslash \{u\}$. It follows that $\sum_{i\in Q'}a_i \leq \sum_{i\in[k-1]}a_i - a_u = \sum_{i\in T}a_i$. In this case,
\begingroup
\allowdisplaybreaks
\begin{align*}
\zeta_{k+1}^Q 
& = f\left(a_{k+1} + a_k + \sum_{i\in Q'} a_i\right) - \sum_{i\in Q'} \rho_i -  \rho_k  \\
& = f\left(a_H + a_L + \sum_{i\in Q'} a_i\right) - \sum_{i\in Q'} \rho_i([i-1]) -  \rho_k([k-1])  \\
& \geq f\left(a_H + a_L + \sum_{i\in Q'} a_i\right) - \sum_{i\in Q'} \rho_i(Q'\cap[i-1]) -  \rho_k([k-1]) \quad\quad\text{($F$ is submodular)} \\
& = f\left(a_H + a_L + \sum_{i\in Q'} a_i\right) - f\left(\sum_{i\in Q'} a_i\right) -  \rho_k([k-1]) \\
& \geq f\left(a_H + a_L + \sum_{i\in T} a_i\right) - f\left(\sum_{i\in T} a_i\right) -  \rho_k([k-1])  \quad\quad\text{{($a_H+a_L\geq 0$, $\sum_{i\in Q'}a_i \leq \sum_{i\in T}a_i$, so Lemma \ref{lemma:f_concave} applies)}} \\
& \geq f\left(a_H + a_L + \sum_{i\in T} a_i\right) - f\left(\sum_{i\in T} a_i\right) -  \rho_k(T) \quad\quad\text{($F$ is submodular, $T\subset [k-1]$)} \\
& =  f\left(a_H + a_L + \sum_{i\in T} a_i\right) - f\left(a_L + \sum_{i\in T} a_i \right) \quad\quad\text{($a_k = a_L$)}\\
& \geq f\left(a_H + a_u + \sum_{i\in T} a_i\right) - f\left(a_u + \sum_{i\in T} a_i\right)   \quad\quad\text{{(by Lemma \ref{lemma:f_concave})}}\\
& = f\left(a_H + \sum_{i = 1}^{k-1} a_i\right) -  f\left(\sum_{i = 1}^{k-1} a_i\right) \quad\quad\text{{($T \cup\{u\}=[k-1]$ by construction)}} \\
& = f\left(a_H + \sum_{i = 1}^{k-1} a_i\right) - \sum_{i = 1}^{k-1} \rho_i   \\
& = \zeta_{k+1}^{[k-1]}. 
\end{align*}
\endgroup
Therefore, for every $Q\subseteq [k]$ with $|Q|\leq k-1$, $\zeta^Q_{k+1}\geq \zeta_{k+1}^{[k-1]}$. That is, $\zeta_{k+1} = \zeta_{k+1}^{[k-1]}$.
\end{proof}

\begin{lemma}
\label{lemma:h_base_case_special2}
Suppose $k, k+1 \in \mathcal{I}_H$ and $d_L \geq 1$. Let $l = \mathcal{L}_{d_L}$, which is the largest index in $[k-1]$ such that $a_{l} = a_L$. Then $\zeta_{k+1} = \min\{\zeta_{k+1}^{[k-1]}, \zeta_{k+1}^{[k]\backslash\{l\}} \}$. 
\end{lemma}
\begin{proof} We partition all the feasible supports $Q$ into two cases. \\

\textit{Case 1.} We first consider all $Q\subseteq [k]$ with $|Q| \leq k-1$, such that $l\notin Q$. Let $q = |Q\backslash [l]|$. We observe that $q \leq |[k]\backslash[l]| = k-l$, and for all $i\in Q\backslash [l]$, $a_i = a_H$. In other words, $q$ is the number of higher-weighted items in $Q$ with indices greater than $l$. In this case, we show that $\zeta_{k+1}^Q \geq \zeta_{k+1}^{[k]\backslash\{l\}}$. 
\begingroup
\allowdisplaybreaks
\begin{align*}
\zeta_{k+1}^Q 
& = f\left(a_H + \sum_{i\in Q, i< l} a_i + \sum_{i\in Q, i > l} a_i\right) - \sum_{i\in Q, i< l} \rho_i - \sum_{i\in Q, i > l} \rho_i \\
& = f\left(a_H + \sum_{i\in Q, i< l} a_i\right) - \sum_{i\in Q, i< l} \rho_i + f\left(a_H + \sum_{i\in Q, i< l} a_i + \sum_{i\in Q, i > l} a_H\right) - f\left(a_H + \sum_{i\in Q, i< l} a_i\right) - \sum_{i\in Q, i > l} \rho_i \\
& \geq f\left(a_H + \sum_{i\in Q, i< l} a_i\right) - \sum_{i\in Q, i< l} \rho_i + f\left(a_H + \sum_{1\leq i< l} a_i + \sum_{i\in Q, i > l} a_H\right) - f\left(a_H + \sum_{1\leq i< l} a_i\right) - \sum_{i\in Q, i > l} \rho_i \\
& \quad\quad\text{($f$ is concave and $\sum_{i\in Q, i< l} a_i\leq \sum_{1\leq i< l} a_i$)} \\
& = f\left(a_H + \sum_{i\in Q, i< l} a_i\right) - \sum_{i\in Q, i< l} \rho_i + \sum_{p = 1}^{q} \left[f\left(\sum_{i=1}^{l-1}a_i + a_H + p\cdot a_H\right)-f\left(\sum_{i=1}^{l-1}a_i + a_H + (p-1)\cdot a_H\right)\right]  \\ 
& \quad \quad - \sum_{i\in Q, i > l} \rho_i  \quad\quad\text{(because $Q\backslash[l]\subseteq \mathcal{I}_H$ and $|Q\backslash[l]|= q$)}\\
& \geq f\left(a_H + \sum_{i\in Q, i< l} a_i\right) - \sum_{i\in Q, i< l} \rho_i + \sum_{p = 1}^{q} \left[f\left(\sum_{i=1}^{l-1}a_i + a_H + p\cdot a_H\right)-f\left(\sum_{i=1}^{l-1}a_i + a_H + (p-1)\cdot a_H\right)\right]  \\ 
& \quad \quad - \sum_{p=1}^q \rho_{l+p}  \quad\quad\text{($\{l+1,\dots,l+q\}\subseteq \mathcal{I}_H$ with size $q$; by Lemma \ref{lemma:descending_zeta}, $\sum_{i=l+1}^{l+q} \rho_i \geq \sum_{i\in Q\backslash [l]} \rho_i$)}\\
&= f\left(a_H + \sum_{i\in Q, i< l} a_i\right) - \sum_{i\in Q, i< l} \rho_i + \sum_{p = 1}^{q} \left[f\left(\sum_{i=1}^{l-1}a_i + a_H + p\cdot a_H\right)-f\left(\sum_{i=1}^{l-1}a_i + a_H + (p-1)\cdot a_H\right)\right]  \\ 
& \quad \quad - \sum_{p=1}^q  \left[f\left(\sum_{i=1}^{l-1}a_i + a_L + p\cdot a_H\right)-f\left(\sum_{i=1}^{l-1}a_i + a_L + (p-1)\cdot a_H\right)\right]  \\
& \geq f\left(a_H + \sum_{i\in Q, i< l} a_i\right) - \sum_{i\in Q, i< l} \rho_i + \sum_{p = 1}^{k-l} \left[f\left(\sum_{i=1}^{l-1}a_i + a_H + p\cdot a_H\right)-f\left(\sum_{i=1}^{l-1}a_i + a_H + (p-1)\cdot a_H\right)\right]  \\ 
& \quad \quad - \sum_{p=1}^{k-l}  \left[f\left(\sum_{i=1}^{l-1}a_i+a_L  + p\cdot a_H\right)-f\left(\sum_{i=1}^{l-1}a_i +a_L+ (p-1)\cdot a_H\right)\right]  \\
& \text{(Let $\varrho_1^p = f\left(\sum_{i=1}^{l-1}a_i + a_H + p\cdot a_H\right)-f\left(\sum_{i=1}^{l-1}a_i + a_H + (p-1)\cdot a_H\right)$, and } \\
& \text{$\varrho_2^p = f\left(\sum_{i=1}^{l-1}a_i+a_L  + p\cdot a_H\right)-f\left(\sum_{i=1}^{l-1}a_i +a_L+ (p-1)\cdot a_H\right)$. By Lemma \ref{lemma:f_concave}, $\varrho_1^p\leq \varrho_2^p$} \\
& \text{for any $1\leq p\leq k-l$.) } \\
& = f\left(a_H + \sum_{i\in Q, i< l} a_i\right) - \sum_{i\in Q, i< l} \rho_i([i-1]) + f\left(\sum_{i=1}^{l-1}a_i + a_H + (k-l)a_H\right)-f\left(\sum_{i=1}^{l-1}a_i + a_H\right) - \sum_{i=l+1}^{k} \rho_i \\
& \geq f\left(a_H + \sum_{i\in Q, i< l} a_i\right) - \sum_{i\in Q, i< l} \rho_i(Q\cap[i-1]) + f\left(\sum_{i=1}^{l-1}a_i + a_H + (k-l)a_H\right)-f\left(\sum_{i=1}^{l-1}a_i + a_H\right) - \sum_{i=l+1}^{k} \rho_i  \\
&  \quad \quad \text{($F$ is submodular)}\\
& =   F(\{k+1\}\cup {(}Q\cap [l-1]{)}) - F(Q\cap[l-1]) + f\left(\sum_{i=1}^{l-1}a_i + a_H + (k-l)a_H\right)-f\left(\sum_{i=1}^{l-1}a_i + a_H\right) - \sum_{i=l+1}^{k} \rho_i\\
& \geq F(\{k+1\}\cup [l-1]) - F([l-1]) + f\left(\sum_{i=1}^{l-1}a_i + a_H + (k-l)a_H\right)-f\left(\sum_{i=1}^{l-1}a_i + a_H\right) - \sum_{i=l+1}^{k} \rho_i \\
&  \quad \quad \text{($F$ is submodular)}\\
&= f\left(a_H + \sum_{i=1}^{l-1} a_i \right) -\sum_{i=1}^{l-1} \rho_i + f\left(\sum_{i=1}^{l-1}a_i + a_H + (k-l)a_H\right)-f\left(\sum_{i=1}^{l-1}a_i + a_H\right) - \sum_{i=l+1}^{k} \rho_i \\
& = f\left(a_H + \sum_{i=1}^{l-1} a_i + \sum_{i=l+1}^k a_i \right) - \sum_{i=1}^{l-1} \rho_i - \sum_{i=l+1}^{k} \rho_i  \\
& = \zeta_{k+1}^{[k]\backslash \{l\}}.
\end{align*}
\endgroup

\textit{Case 2.} Next we consider the remaining $Q\subseteq [k]$ with $|Q| \leq k-1$, which satisfies $l\in Q$. If $Q$ does not contain all the elements in $\mathcal{I}_L\cap [k]$, we let $l'\in\mathcal{I}_L\cap [k]$ be any lower-weighted item that is not included in $Q$. By definition of $l$, $l' < l$. We observe that 
\begingroup
\allowdisplaybreaks
\begin{align*}
\zeta_{k+1}^{Q} & = f\left(a_H + \sum_{i\in Q\backslash\{l\}} a_i + a_L \right) - \sum_{i\in Q\backslash\{l\}} \rho_i -\rho_l \\
& \geq f\left(a_H + \sum_{i\in Q\backslash\{l\}} a_i + a_L \right) - \sum_{i\in Q\backslash\{l\}} \rho_i -\rho_{l'} \quad\quad\text{(because $\rho_{l'}\geq \rho_l$ by Lemma \ref{lemma:descending_zeta})}\\
& = \zeta_{k+1}^{Q\cup\{l'\}\backslash\{l\}} \\
& \geq \zeta_{k+1}^{[k]\backslash \{l\}} \quad\quad\text{(follows from \textit{Case 1}).}
\end{align*} 
\endgroup
Thus it suffices to consider all $Q\subseteq [k]$ with $|Q| \leq k-1$ such that $\mathcal{I}_L\cap [k] \subseteq Q$. Given a fixed $0\leq s \leq k-1 - d_L$, recall that $H^s$ is the set of the first $s$ higher-weighted items in the natural ordering of $N$. For any $Q\subseteq [k]$ that satisfies $|Q| \leq k-1$, $\mathcal{I}_L\cap [k] \subseteq Q$ and $|Q\cap\mathcal{I}_H| = s$,  $\zeta_{k+1}^{Q} \geq \zeta_{k+1}^{\mathcal{L}(d_L)\cup H^s}$ by Lemma \ref{lemma:general_Q_form}. Hence if $Q^* = \arg\min_{Q\subseteq [k],|Q| \leq k-1, \mathcal{I}_L\cap [k] \subseteq Q} \zeta_{k+1}^{Q}$, then $Q^*$ must assume the form of $\mathcal{L}(d_L)\cup H^s$ for some $0\leq s\leq k-1-d_L$. 
\begingroup
\allowdisplaybreaks
\begin{align*}
\zeta_{k+1}^{\mathcal{L}(d_L)\cup H^s} & =  f\left(a_H + \sum_{i\in \mathcal{L}(d_L)\cup H^s} a_i \right) - \sum_{i\in \mathcal{L}(d_L)\cup H^s} \rho_i \\
& = F(\{k+1\}\cup\mathcal{L}(d_L)\cup H^s) - F(\mathcal{L}(d_L)\cup H^s) \\
& \geq F(\{k+1\}\cup\mathcal{L}(d_L)\cup H^{k-1-d_L}) - F(\mathcal{L}(d_L)\cup H^{k-1-d_L}) \quad\quad \text{(because $F$ is submodular)}\\
& = F(\{k+1\}\cup[k-1]) - F([k-1]) \\
& = f\left(a_{k+1} + \sum_{i=1}^{k-1}\right) - \sum_{i=1}^{k-1} \rho_i \\
& = \zeta_{k+1}^{[k-1]}. 
\end{align*}
\endgroup
In summary, given any  $Q\subseteq [k]$ such that $|Q| \leq k-1$, if $Q$ contains all the lower-weighted items before $k$, then $\zeta_{k+1}^Q \geq \zeta_{k+1}^{[k-1]}$; otherwise, $\zeta_{k+1}^Q \geq \zeta_{k+1}^{[k]\backslash \{l\}}$. Therefore, $\zeta_{k+1} = \min\left\{\zeta_{k+1}^{[k]\backslash \{l\}}, \zeta_{k+1}^{[k-1]}\right\}$.
\end{proof}

\begin{lemma}
\label{lemma:h_base_case_special3}
If $[k+1]\subseteq \mathcal{I}_H$, then $\zeta_{k+1} = \zeta_{k+1}^{[k-1]}$.
\end{lemma}
\begin{proof}
This result immediately follows from Lemma \ref{lemma:same_a}. 
\end{proof}

\begin{corollary}
\label{coro:EPI_base_cases}
The lifted-EPI coefficients $\zeta_j$ for $j = k$ and $k+1$ are given by \[
\zeta_{j} = \begin{cases}
\zeta_{j}^{[k-1]}, & j\in\mathcal{I}_L, \\
 & \\
\min\left\{ \zeta_{\mathcal{H}_{i-1}},  \zeta_{j}^{\mathcal{H}(\min\{i-1, d_L\}) \cup \mathcal{L}(d_L-i+1)\cup (\mathcal{I}_H\cap [k-1])}\right\}, & j = \mathcal{H}_i, 1\leq i \leq d_H,  
\end{cases}
\] where $\zeta_{\mathcal{H}_{0}} = \zeta_{j}^{[k-1]}$.
\end{corollary}
\begin{proof}
We know that $\zeta_k = \zeta_k^{[k-1]}$ from the EPI coefficients, so when $k\in\mathcal{I}_L$, the proposed assignment is correct. If $k = \mathcal{H}_1$, then $ \zeta_{j}^{\mathcal{H}(\min\{0, d_L\}) \cup \mathcal{L}(d_L)\cup (\mathcal{I}_H\cap [k-1])}= \zeta_k^{[k-1]}$, and $\zeta_k = \min\{\zeta_k^{[k-1]}, \zeta_k^{[k-1]}\}$, which is also correct. \\

The case of $k+1\in\mathcal{I}_L$ follows from Lemma \ref{lemma:l_base_case}. Suppose $k+1\in\mathcal{I}_H$. Then $k+1 = \mathcal{H}_1$ or $\mathcal{H}_2$. When $k+1 = \mathcal{H}_1$, Lemma \ref{lemma:h_base_case_special1} shows that $\zeta_{k+1} = \zeta_{k+1}^{[k-1]}$. When $k+1 = \mathcal{H}_2$, $k= \mathcal{H}_1$ and $\zeta_k = \zeta_k^{[k-1]} = \zeta_{k+1}^{[k-1]}$. Lemmas \ref{lemma:h_base_case_special2} and \ref{lemma:h_base_case_special3} prove that $\zeta_{k+1} = \min\left\{\zeta_{k+1}^{[k-1]} = \zeta_{ \mathcal{H}_1}, \zeta_{k+1}^{[k]\backslash \{l\}}\right\}$. 
\end{proof}

We have now cleared the base cases. For a strong induction, our induction hypothesis is that for all $j \in[k, J-1]$ the following holds: 
\[
\zeta_j = 
\begin{cases}
\zeta_j^{[k-1]}, & \text{ if } j\in\mathcal{I}_L\backslash [k-1], \\
 & \\
\min\left\{ \zeta_{\mathcal{H}_{i-1}},  \zeta_j^{\mathcal{H}(\min\{i-1, d_L\}) \cup \mathcal{L}(d_L-i+1)\cup (\mathcal{I}_H\cap [k-1])}\right\}, & \text{ if } j = \mathcal{H}_i, i \in [d_H],  
\end{cases}
\] where $\zeta_{\mathcal{H}_{0}} =   \zeta_j^{\mathcal{H}(\min\{0, d_L\}) \cup \mathcal{L}(d_L-0)\cup (\mathcal{I}_H\cap [k-1])} =  \zeta_j^{[k-1]}$. We next show that the proposed coefficients are correct for $j = J$, given the induction hypothesis, to complete the induction step.

\begin{lemma}
\label{lemma:induction_low}
Suppose the induction hypothesis holds. If $J\in\mathcal{I}_L$, then $\zeta_J = \zeta_J^{[k-1]}$.
\end{lemma}
\begin{proof}
As defined earlier in this section, $L^t$ is the set of the first $t$ lower-weighted items in $N$, and $H^s$ is the set of the first $s$ higher-weighted items. Thanks to Lemma \ref{lemma:general_Q_form}, we know that $\arg\min_{Q\subseteq [J-1],|Q|\leq k-1} \zeta_J^Q$ must have the form $L^t\cup H^s$, where $0\leq t\leq |\mathcal{I}_L|$, $0\leq s\leq |\mathcal{I}_H|$ and $s+t\leq k-1$. Recall that $d_L = |\mathcal{I}_L\cap[k-1]|$ and $|\mathcal{I}_H\cap[k-1]| = k-1-d_L$. We prove the stated lemma by cases.\\

\textit{Case 1. } Suppose $t\leq d_L$ and $s\leq k-1-d_L$. This means that both the lower- and the higher-weighted items we include in the candidate set $L^t\cup H^s$ all belong to $[k-1]$. In this case, $L^t\cup H^s \subseteq [k-1]$. By Lemma \ref{lemma:no_k_in_sol}, $\zeta_J^{L^t\cup H^s} \geq \zeta_J^{[k-1]}$.\\

\textit{Case 2. } Suppose $t > d_L$ and $s < k-1-d_L$. This means that the higher-weighted items we include in the candidate set $L^t\cup H^s$ all belong to $[k-1]$; in other words, $H^s\subseteq [k-1]$. Meanwhile some lower-weighted items in $L^t\cup H^s$ are taken from $N\backslash [k-1]$. We let $q = t - d_L$, which is strictly positive by assumption. We construct two sets $W = L^t\cup H^s\cap[k-1]\subseteq [k-1]$ and $U = L^t\cup H^s\backslash [k-1] = L^t\backslash [k-1]$. By design, $|U| = q$ and $W\cup U =L^t\cup H^s$. We also observe that $[k-1]\backslash W \subseteq \mathcal{I}_H$, and 
\begin{equation}
\label{eq:low_induct_property}
\sum_{i\in W} a_i + qa_L\leq \sum_{i\in W} a_i + qa_H \leq \sum_{i\in [k-1]}a_i.
\end{equation} The latter follows from $t+s=|W|+q \leq k-1$, which implies $q\leq k-1-|W|$. For any $L^t\cup H^s$ in this case,
\begingroup
\allowdisplaybreaks
\begin{align*}
\zeta_J^{L^t\cup H^s} & = f\left(a_J + \sum_{i\in W} a_i + \sum_{i\in U} a_i \right) - \sum_{i\in W} \rho_i - \sum_{i\in U} \zeta_i \\
& \geq f\left(a_L + \sum_{i\in W} a_i + q a_L \right) - \sum_{i\in W} \rho_i - q \zeta_J^{[k-1]} \quad\quad\text{(by the induction hypothesis and $U\subseteq L^t$)} \\
& = f\left(a_L + \sum_{i\in W} a_i \right) + \sum_{p=1}^q \left[f\left(a_L + \sum_{i\in W} a_i + pa_L \right) - f\left(a_L + \sum_{i\in W} a_i + (p-1)a_L \right)\right]- \sum_{i\in W} \rho_i - q \zeta_J^{[k-1]} \\ 
& = f\left(a_L + \sum_{i\in W} a_i \right) + \sum_{p=1}^q \left[f\left(a_L + \sum_{i\in W} a_i + pa_L \right) - f\left(\sum_{i\in W} a_i + pa_L \right)\right]- \sum_{i\in W} \rho_i - q \zeta_J^{[k-1]} \\ 
& \geq f\left(a_L + \sum_{i\in W} a_i \right) + \sum_{p=1}^q \left[f\left(a_L + \sum_{i\in [k-1]} a_i \right) - f\left(\sum_{i\in [k-1]} a_i \right)\right]- \sum_{i\in W} \rho_i - q \zeta_J^{[k-1]} \\ 
& \quad\quad\text{(by concavity of $f$ and \eqref{eq:low_induct_property})} \\
& = f\left(a_L + \sum_{i\in W} a_i \right)  + q \zeta_J^{[k-1]} - \sum_{i\in W} \rho_i([i-1])- q \zeta_J^{[k-1]} \\ 
&\geq F(\{J\}\cup W) - \sum_{i\in W} \rho_i(W\cap [i-1]) \quad\quad\text{(by submodularity of $F$)} \\
& = F(\{J\}\cup W) - F(W) \\
& \geq F(\{J\}\cup [k-1]) - F([k-1])\quad\quad\text{(by submodularity of $F$)} \\
& = \zeta_J^{[k-1]}.
\end{align*}
\endgroup

\textit{Case 3. } Suppose $t < d_L$ and $s > k-1-d_L$. In this case, the lower-weighted items we include in the candidate set $L^t\cup H^s$ all belong to $[k-1]$, and some higher-weighted items in $L^t\cup H^s$ are taken from $N\backslash [k-1]$. We define $W = (L^t\cup H^s)\cap [k-1]$. Let $l = \mathcal{L}_{t+1}$ be the ($t+1$)-th lower-weighted item. In this case, $t< d_L$, so $l\leq k-1$, $l\notin W$ and $\zeta_l = \rho_l$. Moreover, we let $q = |(L^t\cup H^s)\backslash [k-1]| = s - (k-1-d_L)$. Since $s > k-1-d_L$, $q >0$. In addition, $t+s = t+q+k-1-d_L \leq k-1$, so $t+q -d_L\leq 0$. Recall that $\mathcal{H}(q)$ is the set of the first $q$ higher-weighted items strictly after $k-1$. We notice that $L^t\cup H^s = W\cup \mathcal{H}(q)$. The set $V = W\cup \{l\}\cup \mathcal{H}(q-1) \subseteq [\mathcal{H}_q-1]$ has cardinality $(t+k-1-d_L)+1+q-1 = k-1 + (t+q-d_L) \leq k-1$. Thus $V$ corresponds to a feasible solution to the lifting problem \eqref{eq:EPI_lifting_simplified} for $\zeta_{\mathcal{H}_q}$, and $\zeta_{\mathcal{H}_{q}}^V \geq \zeta_{\mathcal{H}_{q}}$. 
\begingroup
\allowdisplaybreaks
\begin{align*}
\zeta_J^{L^t\cup H^s} & = f\left(a_J + \sum_{i\in W} a_i + \sum_{i\in \mathcal{H}(q)} a_i \right) - \sum_{i\in W} \rho_i - \sum_{i\in\mathcal{H}(q)} \zeta_i \\
& = f\left(a_l + \sum_{i\in W} a_i + \sum_{i\in \mathcal{H}(q-1)} a_i + a_H \right) - \sum_{i\in W\cup \{l\}} \rho_i - \sum_{i\in\mathcal{H}(q-1)} \zeta_i - \zeta_{\mathcal{H}_{q}} + \rho_l  \\
& = f\left(\sum_{i\in V} a_i + a_H \right) - \sum_{i\in V}  \zeta_i - \zeta_{\mathcal{H}_{q}} + \rho_l \\
& = \zeta_{\mathcal{H}_{q}}^{V} - \zeta_{\mathcal{H}_{q}} + \rho_l \\
& \geq \rho_l \\
& \geq \zeta_J^{[k-1]} \quad\quad \text{(because $l, J\in\mathcal{I}_L$ and Lemma \ref{lemma:descending_zeta} applies).}
\end{align*}
\endgroup

We have now considered every $L^t\cup H^s$, for any $0\leq t\leq |\mathcal{I}_L|$, $0\leq s\leq |\mathcal{I}_H|$ such that  $s+t\leq k-1$. In all the cases, $\zeta_J^{L^t\cup H^s} \geq \zeta_J^{[k-1]}$. Hence we conclude that $\zeta_J = \zeta_J^{[k-1]}$.
\end{proof}

\begin{lemma}
\label{lemma:induction_high_H1}
Suppose the induction hypothesis holds. If $J = \mathcal{H}_1$, then $\zeta_J = \zeta_J^{[k-1]}$. 
\end{lemma}
\begin{proof}
Due to Lemma \ref{lemma:general_Q_form}, it is sufficient for us to show that $\zeta_J^Q \geq \zeta_J^{[k-1]}$ for all $Q\subseteq [J-1]$, $|Q|\leq k-1$ such that $Q = L^t\cup H^s$ for some $0\leq t\leq |\mathcal{I}_L|$ and $0\leq s\leq |\mathcal{I}_H|$. Since $J = \mathcal{H}_1$, the higher-weighted items that can be included in $Q$ must belong to $[k-1]$. That is, $s \leq |\mathcal{I}_H\cap [k-1]| = k-1-d_L$. \\

\textit{Case 1. } Suppose $t \leq d_L$. In other words, the lower weighted items we include in $Q$ are exclusively from $[k-1]$. Then $Q = L^t\cup H^s\subseteq [k-1]$. According to Lemma \ref{lemma:no_k_in_sol}, $\zeta_J^{Q} \geq \zeta_J^{[k-1]}$.    \\ 

\textit{Case 2. } Suppose $t \geq d_L + 1$. Now at least one lower-weighted item indexed between $k$ and $J-1$ is in $Q$. Let the number of such lower-weighted items be $q>0$. We define $W = Q\cap [k-1]$ and $U = Q\backslash [k-1]$. Then $[k-1]\backslash W \subseteq \mathcal{I}_H$, $|W| = d_L + s$ and $U = \{k,\dots, k+q-1\}$. Since $U\subseteq \mathcal{I}_L$, $\zeta_i = \zeta_k^{[k-1]}$ for all $i\in U$ by the induction hypothesis. To ensure the cardinality of $Q$ is at most $k-1$, $s \leq k-2-d_L$. This means that $[k-1]\cap\mathcal{I}_H\backslash H^s \neq \emptyset$. We use $u$ to denote an arbitrary element in $[k-1]\cap\mathcal{I}_H\backslash H^s$, which satisfies $a_u = a_H$ and $\zeta_u = \rho_u$. We observe that $W\cup \{u\}\subseteq [k-1]$. Since $t+s = d_L + q + s \leq k-1$, $q \leq k-1-(d_L+s) = k-1-|W|$. It follows that
\begin{equation}
\label{eq:induct_high_case_12_property}
\sum_{i\in W} a_i + a_H + (q-1) a_L \leq \sum_{i\in W} a_i + q a_H \leq \sum_{i\in [k-1]} a_i. 
\end{equation} For any $Q$ of the given type in this case, 
\begingroup
\allowdisplaybreaks
\begin{align*}
\zeta_J^{Q} & = f\left(a_H + \sum_{i\in W} a_i + \sum_{i\in U} a_L \right) - \sum_{i\in W} \rho_i - \sum_{i\in U} \zeta_i \\
& = f\left(a_H + \sum_{i\in W} a_i + qa_L \right) - \sum_{i\in W} \rho_i - q \zeta_k^{[k-1]} \quad \quad \text{(by induction hypothesis)}\\
& =  f\left(\sum_{i\in W} a_i + a_H + qa_L \right) - f\left(\sum_{i\in W}  a_i + a_H \right)- q \zeta_k^{[k-1]} + f\left(\sum_{i\in W\cup\{u\}} a_i\right) - \sum_{i\in W} \rho_i  \\ 
& = \sum_{p=1}^q \left[f\left(a_L + \sum_{i\in W} a_i + a_H + (p-1)a_L\right) - f\left( \sum_{i\in W} a_i + a_H + (p-1)a_L\right)\right]- q \zeta_k^{[k-1]} \\
& \quad \quad + f\left(\sum_{i\in W\cup\{u\}} a_i\right) - \sum_{i\in W} \rho_i  \\
& \geq \sum_{p=1}^q \left[f\left(a_L + \sum_{i\in [k-1]} a_i \right) - f\left( \sum_{i\in [k-1]} a_i \right)\right]- q \zeta_k^{[k-1]} + f\left(\sum_{i\in W\cup\{u\}} a_i\right) - \sum_{i\in W} \rho_i  \\
& \quad\quad\text{(by Lemma \ref{lemma:f_concave} and \eqref{eq:induct_high_case_12_property})}\\ 
& =  q \zeta_k^{[k-1]} -  q \zeta_k^{[k-1]} + F(W\cup\{u\}) - \sum_{i\in W} \rho_i([i-1]) \\
& \geq F(W\cup\{u\}) - \sum_{i\in W} \rho_i(W\cap[i-1]) \quad\quad\text{(because $F$ is submodular)}\\
& = F(W\cup\{u\}) - F(W) \\
& \geq F([k-1]\cup\{u\}) - F([k-1]) \quad\quad\text{(again because $F$ is submodular)}\\
& = f\left(a_H +\sum_{i=1}^{k-1} a_i\right) - \sum_{i=1}^{k-1}\zeta_i = \zeta_J^{[k-1]}. 
\end{align*}
\endgroup

Therefore, $\zeta_J^Q \geq \zeta_J^{[k-1]}$ for all $Q\subseteq [J-1]$ with $|Q|\leq k-1$. We conclude that $\zeta_{\mathcal{H}_1} = \zeta_J^{[k-1]}$. 
\end{proof}

\begin{lemma}
\label{lemma:induction_high_Hgeq2}
Suppose the induction hypothesis holds. If $J = \mathcal{H}_i$ for some $2\leq i \leq d_H$, then \[\zeta_J = \min\left\{ \zeta_{\mathcal{H}_{i-1}},  \zeta_j^{\mathcal{H}(\min\{i-1, d_L\}) \cup \mathcal{L}(d_L-i+1)\cup (\mathcal{I}_H\cap [k-1])}\right\}.\] 
\end{lemma}
\begin{proof}
In this case, there exists at least one higher-weighted item before $J$ and strictly after $k-1$. The coefficient $\zeta_{\mathcal{H}_{i-1}}$ is the optimal objective value of the $\mathcal{H}_{i-1}$-th lifting problem \eqref{eq:EPI_lifting_simplified}. Since both $\mathcal{H}_{i-1},\mathcal{H}_{i}\in\mathcal{I}_H$, all $Q\subseteq [\mathcal{H}_{i-1}-1]$ with $|Q|\leq k-1$ are the supports for all the feasible solutions $x$ to both the $\mathcal{H}_{i-1}$-th and the $\mathcal{H}_{i}$-th lifting problem \eqref{eq:EPI_lifting_simplified}. Thus $\zeta_J^Q \geq \zeta_{\mathcal{H}_{i-1}} \geq \min\left\{ \zeta_{\mathcal{H}_{i-1}},  \zeta_j^{\mathcal{H}(\min\{i-1, d_L\}) \cup \mathcal{L}(d_L-i+1)\cup (\mathcal{I}_H\cap [k-1])}\right\}$. \\

The following discussion focuses on $Q\subseteq [J-1]$ with $|Q|\leq k-1$, such that $Q$ is not a subset of $[\mathcal{H}_{i-1}-1]$. We aim to show that $\zeta_J^Q \geq \min\left\{ \zeta_{\mathcal{H}_{i-1}},  \zeta_j^{\mathcal{H}(\min\{i-1, d_L\}) \cup \mathcal{L}(d_L-i+1)\cup (\mathcal{I}_H\cap [k-1])}\right\}$ for any such set $Q$. This statement is true as long as it holds for $Q$ in the form of $L^t\cup H^s$ for some $0\leq t\leq |\mathcal{I}_L|$ and $0\leq s\leq |\mathcal{I}_H|$, as a result of Lemma \ref{lemma:general_Q_form}. Since $Q$ is not a subset of $[\mathcal{H}_{i-1}-1]$, $Q$ contains at least one item from the set $\{\mathcal{H}_{i-1},\mathcal{H}_{i-1}+1, \dots, \mathcal{H}_{i}-1\}$. In this set, $\mathcal{H}_{i-1}\in\mathcal{I}_H$ and $\{\mathcal{H}_{i-1}+1, \dots, \mathcal{H}_{i}-1\}\subseteq \mathcal{I}_L$.  \\

\textit{Case 1.} Suppose $\mathcal{H}_{i-1} \in Q = L^t\cup H^s$. In this case, $Q$ contains all the $s = (i-1) + (k-1-d_L) = k+i -2-d_L$ higher-weighted items up to and including $\mathcal{H}_{i-1}$ because of the form it assumes. This assumption implies that $k+i -2-d_L \leq k-1$, so $i-1\leq d_L$. To ensure that $|Q|\leq k-1$, $t\leq k-1-(i-1)-(k-1-d_L) = d_L+1-i$. Recall that $d_L=|\mathcal{I}_L\cap [k-1]|$. Therefore, $L^t = \mathcal{L}(t) \subseteq [k-1]$. For any $0\leq t\leq d_L{+}1-i$, we define $W =\mathcal{L}(t) \cup (\mathcal{I}_H \cap [k-1])$, and the set $Q = L^t\cup H^s$ satisfies
\begingroup
\allowdisplaybreaks
\begin{align*}
\zeta_J^{Q} & = f\left(a_J + \sum_{l\in Q} a_l\right) - \sum_{l\in Q} \zeta_l \\
& =  f\left(a_H + \sum_{l\in W} a_i + (i-1)a_H \right) - \sum_{l\in W} \rho_l - \sum_{l\in \mathcal{H}(i-1)}\zeta_l \\
& \geq  f\left(a_H + \sum_{l\in W} a_i + (i-1)a_H \right)  - \sum_{l\in W} \rho_l - \sum_{l\in \mathcal{H}(i-1)}\zeta_l \\
&\quad\quad + \sum_{p=t+1}^{d_L-i+1} \left( \rho_{\mathcal{L}_{p}}(\{J\}\cup W\cup\mathcal{H}(i-1)\cup \mathcal{L}(p-1))    - \rho_{\mathcal{L}_{p}}([\mathcal{L}_{p}-1])\right) \\
&\quad\quad\text{(because $[\mathcal{L}_{p}-1]\subseteq W\cup \mathcal{L}(p-1)$ for $t+1\leq p\leq d_L-i+1$, and $F$ is submodular)}\\
& =  f\left(a_H + \sum_{l\in W\cup\mathcal{L}(d_L-i+1)} a_i + (i-1)a_H \right) - \sum_{l\in W\cup\mathcal{L}(d_L-i+1)} \rho_l - \sum_{l\in \mathcal{H}(i-1)}\zeta_l \\
& = f\left(a_H + \sum_{l\in \mathcal{L}(d_L-i+1) \cup (\mathcal{I}_H\cap [k-1])} a_i + (i-1)a_H \right) - \sum_{l\in \mathcal{L}(d_L-i+1) \cup (\mathcal{I}_H\cap [k-1])} \rho_l - \sum_{l\in \mathcal{H}(i-1)}\zeta_l \\ 
& = \zeta_J^{\mathcal{L}(d_L-i+1) \cup (\mathcal{I}_H\cap [k-1]) \cup \mathcal{H}(i-1)} \\
& \geq \min\left\{ \zeta_{\mathcal{H}_{i-1}},  \zeta_j^{\mathcal{H}(\min\{i-1, d_L\}) \cup \mathcal{L}(d_L-i+1)\cup (\mathcal{I}_H\cap [k-1])}\right\}.
\end{align*}
\endgroup

\textit{Case 2.} Suppose $\mathcal{H}_{i-1} \notin Q = L^t\cup H^s$. In this case, $Q$ must contain the lower-weighted items $\mathcal{H}_{i-1}+1,\dots, \mathcal{H}_{i-1}+q$ for some $q>0$, so that $Q$ is not a subset of $[\mathcal{H}_{i-1}-1]$. Let $W = Q\cap [\mathcal{H}_{i-1}-1]$, then $Q = W\cup \{\mathcal{H}_{i-1}+1,\dots, \mathcal{H}_{i-1}+q\}$ and $|W| = k-1-q \leq k-2$. {Given that $\mathcal{H}_{i-1}+q\in Q$, we observe that $W\backslash [k-1]$ consists of only lower-weighted items, and there are at least $|W\backslash [k-1]|+q$ higher-weighted items in $[k-1]\backslash W$. Thus, for any $p\in\{1,\dots,q\}$, 
\begingroup
\allowdisplaybreaks
\begin{align*}
 & f\left(\sum_{l\in W\cup\{\mathcal{H}_{i-1}\}} a_l + p a_L \right) - f\left(\sum_{l\in W\cup\{\mathcal{H}_{i-1}\}} a_l + (p-1)a_L \right) \\
= \; &  f\left(\sum_{l\in W} a_l + a_H + pa_L \right) - f\left(\sum_{l\in W} a_l + a_H + (p-1)a_L \right) \\
= \; & f\left(\sum_{l\in W\cap [k-1]} a_l + \sum_{l\in |W\backslash [k-1]|} a_L+ a_H +pa_L \right) - f\left(\sum_{l\in W\cap [k-1]} a_l + \sum_{l\in |W\backslash [k-1]|} a_L + a_H + (p-1)a_L \right) \\
\geq \; &  f\left(\sum_{l\in W\cap [k-1]} a_l + \sum_{l\in |W\backslash [k-1]|+p} a_H+ a_L \right) - f\left(\sum_{l\in W\cap [k-1]} a_l + \sum_{l\in |W\backslash [k-1]|+p} a_H  \right)\quad \text{(by Lemma \ref{lemma:f_concave})}\\
\geq \; & f\left(\sum_{l\in [k-1]} a_l + a_L \right) - f\left(\sum_{l\in [k-1]} a_l \right) \quad \text{(by the aforementioned observation and Lemma \ref{lemma:f_concave})}\\
= \; & \zeta^{[k-1]}_{\mathcal{H}_{i-1}+p} \: . 
\end{align*}
\endgroup}

We further derive that
\begingroup
\allowdisplaybreaks
\begin{align*}
\zeta_J^{Q} & = f\left(a_J + \sum_{l\in W} a_l + qa_L \right) - \sum_{l\in W}\zeta_l - \sum_{p=1}^q \zeta_{\mathcal{H}_{i-1}+p} \\
& = f\left(a_H + \sum_{l\in W} a_l\right) - \sum_{l\in W}\zeta_l + f\left(a_H + \sum_{l\in W} a_l + qa_L \right) - f\left(a_H + \sum_{l\in W} a_l\right) - \sum_{p=1}^q \zeta_{\mathcal{H}_{i-1}+p} \\
& = f\left(a_{\mathcal{H}_{i-1}} + \sum_{l\in W} a_l\right) - \sum_{l\in W}\zeta_l + \sum_{p=1}^q \left[ f\left(\sum_{l\in W\cup\{\mathcal{H}_{i-1}\}} a_l + p a_L \right) - f\left(\sum_{l\in W\cup\{\mathcal{H}_{i-1}\}} a_l + (p-1)a_L \right) - \zeta_{\mathcal{H}_{i-1}+p}\right] \\
&{ \geq \zeta_{\mathcal{H}_{i-1}}^W + \sum_{p=1}^q \left[ \zeta_{\mathcal{H}_{i-1}+p}^{[k-1]}  - \zeta_{\mathcal{H}_{i-1}+p}\right]} \\
&{ \geq \zeta_{\mathcal{H}_{i-1}}^W \quad\quad\text{(for any $1\leq p\leq q$, $[k-1]$ is a feasible solution support for the $(\mathcal{H}_{i-1}+p)$-th lifting problem)} }\\
& \geq \zeta_{\mathcal{H}_{i-1}} \quad\quad\text{($W$ is a feasible solution support for the $\mathcal{H}_{i-1}$-th lifting problem)} \\
& \geq \min\left\{ \zeta_{\mathcal{H}_{i-1}},  \zeta_j^{\mathcal{H}(\min\{i-1, d_L\}) \cup \mathcal{L}(d_L-i+1)\cup (\mathcal{I}_H\cap [k-1])}\right\}.
\end{align*}
\endgroup
So far we have shown that for any $Q\subseteq [J-1]$ with $|Q|\leq k-1$, $\zeta_J^Q \geq \min\left\{ \zeta_{\mathcal{H}_{i-1}},  \zeta_j^{\mathcal{H}(\min\{i-1, d_L\}) \cup \mathcal{L}(d_L-i+1)\cup (\mathcal{I}_H\cap [k-1])}\right\}$. Hence $\zeta_J = \min\left\{ \zeta_{\mathcal{H}_{i-1}},  \zeta_j^{\mathcal{H}(\min\{i-1, d_L\}) \cup \mathcal{L}(d_L-i+1)\cup (\mathcal{I}_H\cap [k-1])}\right\}.$ 
\end{proof}

\begin{lemma}
\label{lemma:induction_high}
Suppose the induction hypothesis holds. If $J\in\mathcal{I}_H$, then $J = \mathcal{H}_i$ for some $i \in [d_H]$. The $J$-th lifted-EPI coefficient is   
\[ \zeta_J = \min\left\{ \zeta_{\mathcal{H}_{i-1}},  \zeta_j^{\mathcal{H}(\min\{i-1, d_L\}) \cup \mathcal{L}(d_L-i+1)\cup (\mathcal{I}_H\cap [k-1])}\right\},\] where $\zeta_{\mathcal{H}_{0}} =  \zeta_J^{[k-1]}$.
\end{lemma}
\begin{proof}
This induction step for $J\in\mathcal{I}_H$ holds by Lemmas \ref{lemma:induction_high_H1} and \ref{lemma:induction_high_Hgeq2}.  
\end{proof}

With all the lemmas established above, we now prove Proposition \ref{prop:EPI_lifted_coeff}. 
\begin{proof} (Proposition \ref{prop:EPI_lifted_coeff})
The proposed lifted-EPI coefficients $\zeta_j$ are correct in the base cases $j = k$ and $k+1$ according to Corollary \ref{coro:EPI_base_cases}. Given our induction hypothesis that the proposed coefficients hold for all $j\in[k,J-1]$, Lemmas \ref{lemma:induction_low} and \ref{lemma:induction_high} show that the proposed $\zeta_J$ is the optimal objective of the $J$-th lifting problem \eqref{eq:EPI_lifting_simplified}. Hence we conclude that the proposed lifted-EPI coefficients $\zeta_j$ for all $j\in[k,n]$ are indeed the desired optimal objective values of the corresponding lifting problems \eqref{eq:EPI_lifting_simplified}. In other words, our lifted-EPIs are exact from lifting the EPIs.
\end{proof}

Now we know that the lifting coefficients given in Proposition \ref{prop:EPI_lifted_coeff} are exact. In the next corollaries, we infer the strength of the lifted-EPIs. 
\begin{customcoro}{3.6.1}
\label{coro:lifted_EP_strength}
The lifted-EPIs are facet-defining for $\conv{\mathcal{P}^2_k}$. 
\end{customcoro}
\begin{proof}
For any $S\subseteq N$ with $|S| = k$, the cardinality constraint in $\mathcal{P}^2_k(S)$ is redundant. Thus the EPIs are facet-defining for such $\conv{\mathcal{P}^2_k(S)}$ \cite{edmonds2003submodular}. Since the lifted-EPIs are exactly lifted from the EPIs, they are facet-defining for $\conv{\mathcal{P}^2_k}$. 
\end{proof}

\begin{customcoro}{3.6.2}
\label{coro:lifted_vs_epi}
For any $\conv{\mathcal{P}^2_k}$, the lifted-EPIs are at least as strong as the \textit{approximate lifted inequalities} proposed in \cite{yu2017polyhedral}, Proposition 11. Although \citet{yu2017polyhedral} call such inequalities the lifted inequalities, to distinguish them from the lifted-EPIs with exact lifting coefficients, we refer to them as the approximate lifted inequalities (ALIs). An ALI has the form 
\[ w\geq \sum_{i=1}^k\rho_i x_i + \sum_{i=k+1}^n \phi_ix_i, \]
where $\rho_i$ for $ i\in [k]$ are the EPI coefficients. For each $i>k$, let $T$ with $|T|= k-1$ be a subset of $[i-1]$ such that the sum of the weights are as high as possible. Then $\phi_i = f(a_i + \sum_{j\in T}a_j) - f(\sum_{j\in T}a_j)$. 
\end{customcoro}
\begin{proof}
Let an EPI with respect to $[k]$ be given. The lifted-EPI $w\geq \sum_{i=1}^k\rho_i x_i + \sum_{i=k+1}^n \zeta_ix_i$ is exactly lifted from this base EPI. The proof of Proposition 11 in \cite{yu2017polyhedral} shows that $\phi_i\leq \zeta_i$ for $i\in [k+1, n]$.
\end{proof}

\begin{example} Suppose $N = [6]$, $a = [ 4, 100,  100,  100,   4,   4]$ and $k = 2$. Let us consider the concave function $f(a^\top x) = \sqrt{a^\top x}$. The ALI \cite{yu2017polyhedral} with a permutation of $N$, $\delta = (2,5,1,6,4,3)$, is
\[w \geq 0.198x_1+ 10x_2+ 4.142x_3+ 4.142x_4+ 0.198x_5+ 0.198x_6,\] which coincides with the lifted-EPI, that we exactly lift from the base EPI for $S = \{2,5\}$. Another permutation $\delta = (5,2,3,1,4,6)$ yields an ALI
\[w \geq 0.198x_1+ 8.198x_2+ 4.142x_3+ 4.142x_4+ 2x_5+ 0.198x_6. \] Consider the EPI that is associated with $S = \{2,5\}$ and $\delta$. The corresponding lifted-EPI is
\[w\geq 0.828x_1+ 8.198x_2+ 5.944x_3+5.944x_4+ 2x_5+ 0.828x_6.\] In this example, the lifted-EPI dominates the ALI. 
\end{example}

\section{Exact Lifting of Separation Inequalities}
\label{sect:sepa}
In this section, we exactly lift the SIs proposed in \cite{yu2017polyhedral} to obtain strong valid linear inequalities for $\conv{\mathcal{P}^2_k}$. We refer the readers to Section \ref{sect:prelim} for a detailed introduction to the SIs \eqref{eq:sepa_coeff} and the definitions of $\mathcal{P}^1_k(\mathcal{I}_L)$ and $\mathcal{P}^1_k(\mathcal{I}_H)$. In particular, recall that $i_0\in \{0,1,\dots, k-1\}$ is a fixed parameter used to construct an SI. In Section \ref{sect:L_sepa}, we propose the \textit{lower-separation inequalities} (lower-SIs) that are exactly lifted from the SIs of $\conv{\mathcal{P}^1_k(\mathcal{I}_L)}$. In Section \ref{sect:H_sepa}, we propose another class of inequalities that are exactly lifted from the SIs of $\conv{\mathcal{P}^1_k(\mathcal{I}_H)}$. We call these lifted cuts the \textit{higher-separation inequalities} (higher-SIs). \\

Before analyzing the lifting procedures, we show some useful properties of the coefficients in any SI constructed with an integer $0\leq i_0\leq k-1$. In the lemmas below, we let $N=[n]$ be the ground set, in which each item has weight $\alpha\in\mathbb{R}_+$. For ease of notation, we assume that the permutation $\delta$ used to construct SI is $(1,2,\dots,n)$, so we omit $\delta$ in the indices.

\begin{lemma}
\label{lemma:average_terms}
For any $r\in [k-i_0]$, \[r\psi = \frac{r}{k-i_0}\left[ f(k\alpha) - f(i_0\alpha) \right] \leq f((i_0+r)\alpha) - f(i_0\alpha).\] 
\end{lemma}
\begin{proof}
The stated inequality is equivalent to $\frac{k-i_0}{r}[f((i_0+r)\alpha) - f(i_0\alpha)] \geq f(k\alpha) - f(i_0\alpha)$ because $\frac{k-i_0}{r} >0$. We observe that
\begingroup
\allowdisplaybreaks
\begin{align*}
& \quad \quad \frac{k-i_0}{r} [f((i_0+r)\alpha) - f(i_0\alpha)]  \\
& = f((i_0+r)\alpha) - f(i_0\alpha) + \frac{k-i_0-r}{r} [f((i_0+r)\alpha) - f(i_0\alpha)]  \\
& = f((i_0+r)\alpha) - f(i_0\alpha) + \frac{k-i_0-r}{r} \sum_{i=1}^r [f((i_0+i)\alpha) - f((i_0+i-1)\alpha)] \\
& \geq f((i_0+r)\alpha) - f(i_0\alpha) + \frac{k-i_0-r}{r} \sum_{i=1}^r [f((i_0+r)\alpha) - f((i_0+r-1)\alpha)] \quad \text{(by concavity of $f$)}  \\
& = f((i_0+r)\alpha) - f(i_0\alpha) + (k-i_0-r)[f((i_0+r)\alpha) - f((i_0+r-1)\alpha)]  \\
& \geq f((i_0+r)\alpha) - f(i_0\alpha) + \sum_{l=1}^{k-i_0-r} [f((i_0+r+l)\alpha) - f((i_0+r-1+l)\alpha)] \quad \text{(by concavity of $f$)} \\
& = f((i_0+r)\alpha) - f(i_0\alpha) + f(k\alpha) - f((i_0+r)\alpha)  \\
& =  f(k\alpha) - f(i_0\alpha). 
\end{align*} Therefore the stated relation holds. 
\endgroup
\end{proof}

\begin{lemma}
\label{lemma:descending_sepa}
In the SI \eqref{eq:sepa_coeff}, $\rho_1 \geq \rho_2 \geq \dots \geq \rho_{i_0} \geq \psi$.
\end{lemma}
\begin{proof}
The descending trend among $\rho_i$, for $i\in [i_0]$, follows from concavity of $f$. By Lemma \ref{lemma:average_terms}, $\psi \leq f((i_0+1)\alpha) - f(i_0 \alpha)$. Moreover, $f((i_0+1)\alpha) - f(i_0 \alpha) \leq f(i_0\alpha) - f((i_0-1) \alpha) = \rho_{i_0}$ because of Lemma \ref{lemma:f_concave}. Thus $\rho_{i_0} \geq \psi$, which completes the proof.  
\end{proof}

\subsection{Lower-separation inequalities}
\label{sect:L_sepa}
Without loss of generality, we index the items in $N$ in a way such that, $[|\mathcal{I}_L|]$ are the lower-weighted items, and $[|\mathcal{I}_L|+1, n]$ are higher-weighted. We assume that $ |\mathcal{I}_L|\geq k$ so that SIs are defined for $\conv{\mathcal{P}^1_k(\mathcal{I}_L)}$. Let any such SI \eqref{eq:sepa_coeff} constructed with some $i_0\in\{0,1,\dots,k-1\}$ be given. Suppose the permutation of $\mathcal{I}_L$ used to construct this SI is $\delta$. Again without loss of generality, we assume that the permutation  $\delta=(1,2,\dots, |\mathcal{I}_L|) $. This can be achieved by re-indexing the lower-weighted items in $N$. Thus we omit $\delta$ in the discussion below. \\

We would like to lift this arbitrary SI to derive an inequality of the form 
\begin{equation}
\label{eq:lower_sepa}
w\geq \sum_{i = 1}^{i_0} \rho_i x_i + \sum_{i = i_0+1}^{|\mathcal{I}_L|} \psi x_i + \sum_{j = |\mathcal{I}_L|+1}^n \eta_jx_j.
\end{equation} 
In this expression, $\eta_j$ is the optimal objective value of the $j$-th lifting problem \eqref{eq:l_lifting_prob} for $j \in [|\mathcal{I}_L|+1, n]$.
\begingroup
\allowdisplaybreaks
\begin{subequations}
\label{eq:l_lifting_prob}
\begin{alignat}{2}
\eta_j := \min \hspace{0.2cm} & w - \sum_{i=1}^{i_0} \rho_i x_i - \sum_{i=i_0+1}^{|\mathcal{I}_L|} \psi x_i - \sum_{i = |\mathcal{I}_L|+1}^{j-1} \eta_i x_i&& \\
\textrm{s.t.} \quad & w\geq f\left(a_H + \sum_{i = 1}^{j-1} a_i x_i\right), &&\label{eq:l_sepa_constr_w}\\
& \sum_{i=1}^{j-1} x_i \leq k-1, && \\
& x \in \{0,1\}^{j-1}. && 
\end{alignat}
\end{subequations}
\endgroup
We call such inequalities the \textit{lower-SIs}. \\

In the $j$-th lifting problem \eqref{eq:l_lifting_prob}, any feasible $x$ has a corresponding support $X = \{i\in [j-1]: x_i = 1\}$. On the other hand, for any $X\subseteq [j-1]$ with $|X|\leq k-1$, there exists a unique feasible solution $x$ such that $x_i = 1$ if $i\in X$, and $0$ otherwise. We will later analyze the optimal objective of \eqref{eq:l_lifting_prob} in terms of the feasible supports. Since we are minimizing the objective function, given any feasible $x$, the lowest objective value is attained when constraint \eqref{eq:l_sepa_constr_w} is tight. We denote the best objective value evaluated at a feasible $x$ with support $X$ by 
\[\eta^{X} = f\left(a_H + \sum_{i\in X}a_i\right) - \sum_{i\in[i_0] \cap X} \rho_i  - \sum_{i\in[i_0+1, |\mathcal{I}_L|]\cap X} \psi - \sum_{i \in [|\mathcal{I}_L|+1, j-1]\cap X} \eta_i.\]

We first note that, the lifted coefficients $\eta_j$'s are descending. 
\begin{lemma}
\label{lemma:descending_eta}
For any $|\mathcal{I}_L|+1\leq j_1 < j_2 \leq n$, $\eta_{j_1} \geq \eta_{j_2}$. 
\end{lemma}
\begin{proof}
This result immediately follows from Proposition 1.3 on page 264 of \cite{wolsey1999integer}.
\end{proof}

Recall that $H^s$ is the set of the first $s$ higher-weighted items in $N$. In this section, by our assumed indexing, $H^s = [|\mathcal{I}_L|+1, |\mathcal{I}_L|+s]$. The next lemma characterizes a general form of an optimal solution support to any lifting problem \eqref{eq:l_lifting_prob}.

\begin{lemma}
\label{lemma:l_lift_sol_form}
For any $j\in [|\mathcal{I}_L|+1, n]$, $\eta_j = \eta^{Q^*}$ for some $Q^*\subseteq [j-1]$ with  $|Q^*|\leq k-1$, such that $Q^*=[t]\cup H^s$, for some $0\leq t\leq k-1$ and $0\leq s\leq |\mathcal{I}_H|$ such that $t+s\leq k-1$. 
\end{lemma}
\begin{proof}
Consider any $Q\subseteq [j-1]$ that satisfies $|Q|\leq k-1$. We let $\overline{t} = |\mathcal{I}_L\cap Q|$ and $\overline{s} = |\mathcal{I}_H\cap Q|$. Then 
\begingroup
\allowdisplaybreaks
\begin{align*}
\eta^Q & = f\left(a_H + \overline{t}a_L + \overline{s}a_H\right) - \sum_{i\in[i_0]\cap Q} \rho_i  - \sum_{i\in[i_0+1,|\mathcal{I}_L|]\cap Q} \psi - \sum_{i \in [|\mathcal{I}_L|+1,j-1]\cap Q} \eta_i \\
& \geq f\left(a_H + \overline{t}a_L + \overline{s}a_H\right) - \sum_{i\in[i_0]\cap [\overline{t}]} \rho_i  - \sum_{i\in[i_0+1,|\mathcal{I}_L|]\cap [\overline{t}]} \psi - \sum_{i\in H^{\overline{s}}} \eta_i \quad\quad\text{(from Lemmas \ref{lemma:descending_sepa} and \ref{lemma:descending_eta})}\\
& = \eta^{[\overline{t}]\cup H^{\overline{s}}}. 
\end{align*}
\endgroup Therefore the set of all the feasible supports $Q$ in the form of $[t]\cup H^s$ contains an optimal support $Q^* = \arg\min_{Q\subseteq [j-1], |Q|\leq k-1} \eta^Q$, such that $\eta^{Q^*} = \eta_j$. 
\end{proof}

Lemma \ref{lemma:l_lift_sol_form} suggests that there must exist an optimal solution support for any lifting problem \eqref{eq:l_lifting_prob} that has the form $[t]\cup H^s$, which concatenates the first $t$ lower-weighted items with the first $s$ higher-weighted items. We next compare all feasible solutions of this form in Lemmas \ref{lemma:l_head_eta} , \ref{lemma:l_tail_eta} and \ref{lemma:i_0_loses}. It turns out that for any fixed number of higher-weighted items $s\leq k-1$, the support $[k-1-s]\cup H^s$ always has the lowest objective value. This result is formalized in Lemma \ref{lemma:card_k-1}. \\

\begin{lemma}
\label{lemma:l_head_eta} 
For any $0\leq s\leq \min\{|\mathcal{I}_H|, k-1\}$, $\eta^{[t-1]\cup H^s} \geq \eta^{[t]\cup H^s}$ for all $t\in [\min\{k-1-s, i_0\}]$. 
\end{lemma}
\begin{proof} Given any $[t-1]\cup H^s$ that satisfies the stated properties, 
\begingroup
\allowdisplaybreaks
\begin{align*} 
\eta^{[t-1]\cup H^s} & = f\left(a_H + (t-1)a_L + sa_H\right) - \sum_{i=1}^{t-1} \rho_i - \sum_{i = |\mathcal{I}_L|+1}^{|\mathcal{I}_L|+s} \eta_i \\
& = f\left(ta_L + (s+1)a_H\right) - \left[ f\left(ta_L + (s+1)a_H\right) -  f\left((t-1)a_L + (s+1)a_H\right)\right] - \sum_{i=1}^{t-1} \rho_i - \sum_{i = |\mathcal{I}_L|+1}^{|\mathcal{I}_L|+s} \eta_i \\
& \geq  f\left(ta_L + (s+1)a_H\right) - \left[ f\left(ta_L \right) -  f\left((t-1)a_L\right)\right] - \sum_{i=1}^{t-1} \rho_i - \sum_{i = |\mathcal{I}_L|+1}^{|\mathcal{I}_L|+s} \eta_i \quad\quad \text{(because $f$ is concave)}\\
& = f\left(a_H + ta_L + sa_H\right) - \rho_t - \sum_{i=1}^{t-1} \rho_i - \sum_{i = |\mathcal{I}_L|+1}^{|\mathcal{I}_L|+s} \eta_i\\ 
& = \eta^{[t]\cup H^s}.
\end{align*}
\endgroup 
\end{proof}

\begin{lemma}
\label{lemma:l_tail_eta}
Let any $0\leq s\leq \min\{|\mathcal{I}_H|, k-1\}$ be given. If $k-1-s \geq i_0$, then $\eta^{[t]\cup H^s} \geq \min \{\eta^{[i_0]\cup H^s}, \eta^{[k-1-s]\cup H^s}\}$ for all $t\in [i_0, k-1-s]$.
\end{lemma}
\begin{proof}
We first observe that for every $t\in[i_0, k-2-s]$, \[\eta_j^{[t+1]\cup H^s} - \eta_j^{[t]\cup H^s}  = f\left((s+1)a_H + ta_L + a_L\right) - f\left((s+1)a_H + ta_L\right)  - \psi.\] Since $f$ is concave, $f\left((s+1)a_H + ta_L + a_L\right) - f\left((s+1)a_H + ta_L\right)$ decreases as $t$ increases. This implies that when $t$ becomes bigger, $\eta_j^{[t+1]\cup H^s} - \eta_j^{[t]\cup H^s}$ shrinks. With this observation, we prove this lemma by contradiction. Suppose there exists $i_0 < q < k-1-s$, such that $\eta_j^{[q]\cup H^s} < \eta^{[i_0]\cup H^s}$ and $\eta_j^{[q]\cup H^s} < \eta^{[k-1-s]\cup H^s}$. Then 
\[\eta^{[q]\cup H^s}-\eta^{[i_0]\cup H^s} = \sum_{p=i_0}^{q-1} \left[\eta^{[p+1]\cup H^s} - \eta^{[p]\cup H^s}\right] < 0,\] 
and \[\eta^{[k-1-s]\cup H^s}-\eta^{[q]\cup H^s} = \sum_{p=q}^{k-2-s} \left[\eta^{[p+1]\cup H^s} - \eta^{[p]\cup H^s}\right] > 0.\] 
Hence, there exists $q_1\in[i_0, q-1]$ such that $ \eta^{[q_1+1]\cup H^s} - \eta^{[q_1]\cup H^s} < 0$. There also exists $q_2 \in[q,k-2-s]$ such that $\eta^{[q_2+1]\cup H^s} - \eta^{[q_2]\cup H^s} > 0$. This contradicts our observation that $\eta_j^{[t+1]\cup H^s} - \eta_j^{[t]\cup H^s}$ decreases as $t$ gets larger. Thus no such $q$ exists. We conclude that $\eta^{[t]\cup H^s} \geq \min \{\eta^{[i_0]\cup H^s}, \eta^{[k-1-s]\cup H^s}\}$ for all $t\in[i_0, k-1-s]$. 
\end{proof}

\begin{lemma}
\label{lemma:i_0_loses}
Let any $0\leq s\leq \min\{|\mathcal{I}_H|, k-1\}$ be given. If $k-1-s \geq i_0$, then $\eta^{[i_0]\cup H^s}\geq \eta^{[k-1-s]\cup H^s}$. In other words, $\eta^{[t]\cup H^s} \geq \eta^{[k-1-s]\cup H^s}$ for all $t\in[i_0, k-1-s]$.
\end{lemma}
\begin{proof} The difference $\eta^{[i_0]\cup H^s} - \eta^{[k-1-s]\cup H^s}$ turns out to be non-negative. 
\begingroup
\allowdisplaybreaks
\begin{align*}
&  \quad\quad \eta^{[i_0]\cup H^s} - \eta^{[k-1-s]\cup H^s}  \\
& = f\left(a_H + i_0a_L + sa_H\right) - f\left(a_H + (k-1-s)a_L + sa_H\right) +(k-1-s-i_0) \psi \\
& = f\left( i_0a_L + (s+1)a_H\right) - f\left((k-1-s)a_L + (s+1)a_H\right) +\frac{(k-i_0)-(1+s)}{k-i_0} \left[f(ka_L)-f(i_0a_L) \right] \\ 
& =  f\left( i_0a_L + (s+1)a_H\right) - f\left((k-1-s)a_L + (s+1)a_H\right) + f(ka_L)-f(i_0a_L) - (1+s)\psi \\
& \geq f\left( i_0a_L + (s+1)a_H\right) - f\left((k-1-s)a_L + (s+1)a_H\right) + f(ka_L)-f(i_0a_L)  + f(i_0a_L) - f((i_0+1+s)a_L) \\
&  \quad \quad \text{(by Lemma \ref{lemma:average_terms})}\\
& = f\left( i_0a_L + (s+1)a_H\right) - f((i_0+1+s)a_L)  - \left[ f\left((k-1-s)a_L + (s+1)a_H\right) - f(ka_L)\right]\\
& = f\left( (i_0+1+s)a_L + (s+1)(a_H-a_L)\right) - f((i_0+1+s)a_L)  - \left[ f\left(ka_L + (s+1)(a_H-a_L)\right) - f(ka_L)\right]\\
& \geq 0, \quad \quad  \text{(because $i_0+1+s\leq k$ and $f$ is concave).} 
\end{align*}
\endgroup
\end{proof}

\begin{lemma}
\label{lemma:card_k-1}
Given any $0\leq s\leq \min\{k-1, |\mathcal{I}_H|\}$, $\eta^{[k-1-s]\cup H^s} \leq \eta^{[t]\cup H^s}$ for any $t$ such that $[t]\cup H^s$ is a feasible solution support to the $(|\mathcal{I}_L|+s+1)$-th lifting problem. 
\end{lemma}
\begin{proof}
If $s \geq k-1-i_0$, then to ensure $|[t]\cup H^s| \leq k-1$, $0\leq t\leq k-1-s \leq i_0$. In this case, Lemma \ref{lemma:l_head_eta} immediately suggests that $\eta^{[t]\cup H^s} \geq \eta^{[k-1-s]\cup H^s}$. On the other hand, if $s\leq k-2-i_0$, then any $t\leq k-1-s$ makes $[t]\cup H^s$ is a feasible solution support in the $(|\mathcal{I}_L|+s+1)$-th lifting problem. For all $t\leq i_0$, $\eta^{[t]\cup H^s} \geq \eta^{[i_0]\cup H^s}$ again by Lemma \ref{lemma:l_head_eta}. It then follows from Lemma \ref{lemma:i_0_loses} that for any $t\leq k-1-s$, $\eta^{[t]\cup H^s} \geq \min \{\eta^{[i_0]\cup H^s}, \eta^{[k-1-s]\cup H^s}\} = \eta^{[k-1-s]\cup H^s}$.
\end{proof}

\begin{proposition}
\label{prop:l_sepa_coeff}
The exact lifting coefficients from the lifting problems \eqref{eq:l_lifting_prob} are \[\eta_j = \begin{cases}
\eta^{[k-1]}, & j = |\mathcal{I}_L|+1, \\
\min\{\eta_{j-1}, \eta^{[k-1-s]\cup H^s}\}, & j = |\mathcal{I}_L|+1+s, s\in [n-1-|\mathcal{I}_L|].
\end{cases}
\]
\end{proposition}
\begin{proof}
Recall that the optimal solution support for any lifting problem \eqref{eq:l_lifting_prob} has the form $[t]\cup H^s$ for some $s,t\leq k-1$ according to Lemma \ref{lemma:l_lift_sol_form}. In addition, according to Lemma \ref{lemma:card_k-1}, such a support has $s+t = k-1$. Therefore, $\eta_{|\mathcal{I}_L|+1} = \eta^{[k-1]\cup H^0} = \eta^{[k-1]}$ as stated in this proposition. For any $j > |\mathcal{I}_L|+1$, we let $j = |\mathcal{I}_L|+1+s$, where $s\in\{1,2,\dots, n-1-|\mathcal{I}_L|\}$. We first consider any solution support $Q\subseteq [j-1]$ with $|Q|\leq k-1$, such that $j-1\notin Q$. Such solutions are feasible to both the $j$-th and the $(j-1)$-th lifting problems \eqref{eq:l_lifting_prob}. Thus for any such $Q$, $\eta^Q \geq \eta_{j-1} \geq \min\{\eta_{j-1}, \eta^{[k-1-s]\cup H^s}\}$. The remaining feasible supports are $Q\subseteq [j-1]$ with $|Q|\leq k-1$ that contain $j-1$. Due to Lemma \ref{lemma:l_lift_sol_form} we only need to consider those with the form $[t]\cup H^s$ for some $0\leq t\leq k-1-s$. Thanks to Lemma \ref{lemma:card_k-1}, we know that $\eta^{[k-1-s]\cup H^s}$ has the lowest objective value among all these supports. Therefore, in this case, $\eta^Q \geq \eta^{[k-1-s]\cup H^s} \geq  \min\{\eta_{j-1}, \eta^{[k-1-s]\cup H^s}\}$ as well. We conclude that the proposed assignments are indeed the exact lifting coefficients.
\end{proof}

\citet{yu2017polyhedral} show that the SIs constructed with all $i_0$ such that $0\leq i_0\leq k-1$, together with the trivial 0-1 bounds and the cardinality constraint, give the convex hull of $\mathcal{P}^1_k(\mathcal{I}_L)$, as well as $\mathcal{P}^1_k(\mathcal{I}_H)$. We thus infer the following corollary regarding the strength of our lower-SIs. 
\begin{corollary}
\label{eq:l_sepa_strength}
Based on any SI that is facet-defining for $\conv{\mathcal{P}^1_k(\mathcal{I}_L)}$, the lower-SIs given by Proposition \ref{prop:l_sepa_coeff} are facet-defining for $\conv{\mathcal{P}^2_k}$. 
\end{corollary}

\subsection{Higher-separation inequalities}
\label{sect:H_sepa}
Next, we lift the SIs of $\conv{\mathcal{P}^1_k(\mathcal{I}_H)}$. Throughout this section, we impose the following assumption.
\begin{assumption}
\label{assumption}
For a given $i_0\in \{0,1,\dots, k-2\}$, the weights $a_L$ and $a_H$ satisfy 
\begin{equation}
\label{eq:assump}
f(a_L + (i_0+1)a_H) - f(a_L+i_0a_H) \leq \frac{f(ka_H)-f(i_0a_H)}{k-i_0}.
\end{equation}
\end{assumption} Note that \eqref{eq:assump} is always true when $i_0 = k-1$. This is because $f(a_L + ka_H) - f(a_L+(k-1)a_H)  = f(a_L + (k-1)a_H + a_H) - f(a_L+(k-1)a_H) \leq f((k-1)a_H + a_H) - f((k-1)a_H)$, where the inequality follows from concavity of $f$. {The right-hand side of \eqref{eq:assump} is the average marginal contribution of $k-i_0$ units of the higher-weighted items, which matches the coefficient $\psi$ in the SI, associated with $i_0$, of $\conv{\mathcal{P}^1_k(\mathcal{I}_H)}$. Intuitively, Assumption \ref{assumption} suggests that $\psi$ dominates the marginal contribution of one unit of the higher-weighted item when it is added to a collection of at least one lower-weighted item and at least $i_0$ higher-weighted items. Under this assumption, we will be able to quantify the net effect of adding or removing a higher-weighted item to the objective value of the lifting problem, given any feasible support with a fixed number of lower-weighted items (see Lemmas \ref{lemma:h_tail_eta} and \ref{lemma:h_card_k-1}). This is crucial to the derivation of the exact lifting coefficients.}  

{
\begin{remark}
Assumption \ref{assumption} is satisfied when $a_H/q \leq a_L$ for some real number $q\geq 1$ that depends on the given parameters $i_0$, $k$, and the function $f$. For example, for $f(\cdot) = \sqrt{\cdot}$, $k=2$, and $i_0=0$, this assumption holds when $a_H/8\leq a_L \leq a_H$. A higher value of $q$ means that a wider range of $a_L$ will satisfy Assumption \ref{assumption} given a fixed $a_H$. We observe empirically that, when $k$ is low, $q$ is high across the feasible choices of $i_0$. For a fixed $k$, $q$ is usually high when $k-i_0$ is low. When $f$ is twice differentiable, a high curvature of $f$ at $i_0a_H$ for a fixed $a_H$ tends to suggest a high $q$ as well. 
\end{remark}
}

Similar to the setups in Section \ref{sect:L_sepa}, we re-index $N$ such that the first $|\mathcal{I}_H|$ items are higher-weighted, and the items $|\mathcal{I}_H|+1$ to $n$ are lower-weighted. Suppose we are given an arbitrary SI for $\conv{\mathcal{P}^1_k(\mathcal{I}_H)}$ constructed with $i_0\in\{0,1,\dots, k-1\}$. In this section, we assume that Assumption \ref{assumption} holds for this given $i_0$. Moreover, we assume that $|\mathcal{I}_H| > k$ for this SI to be defined. Without loss of generality, $\delta = (1,2,\dots, |\mathcal{I}_H|)$ is the permutation associated with the given SI. This allows us to drop $\delta$ and simplify the notation. \\

In the order of $j = |\mathcal{I}_H|+1,|\mathcal{I}_H|+2, \dots, n$, we sequentially solve the lifting problem \eqref{eq:h_lifting_prob} 
\begingroup
\allowdisplaybreaks
\begin{subequations}
\label{eq:h_lifting_prob}
\begin{alignat}{2}
\gamma_j := \min \hspace{0.2cm} & w - \sum_{i=1}^{i_0} \rho_i x_i - \sum_{i=i_0+1}^{|\mathcal{I}_H|} \psi x_i - \sum_{i = |\mathcal{I}_H|+1}^{j-1} \gamma_i x_i&& \\
\textrm{s.t.} \quad & w\geq f\left(a_L + \sum_{i = 1}^{j-1} a_i x_i\right), &&\label{eq:h_sepa_constr_w}\\
& \sum_{i=1}^{j-1} x_i \leq k-1, && \\
& x \in \{0,1\}^{j-1}. && 
\end{alignat}
\end{subequations}
\endgroup
With the optimal objective values $\gamma_j$, we construct inequality \eqref{eq:higher_sepa}, which is exactly lifted from the given SI. 
\begin{equation}
\label{eq:higher_sepa}
w\geq \sum_{i = 1}^{i_0} \rho_i x_i + \sum_{i = i_0+1}^{|\mathcal{I}_H|} \psi x_i + \sum_{j = |\mathcal{I}_H|+1}^n \gamma_jx_j.
\end{equation} 
We call such inequalities the \textit{higher-SIs}. \\

Similar to the discussion in Section \ref{sect:L_sepa}, We denote the best objective value evaluated at a feasible $x$ with support $X$ by 
\[\gamma^{X} = f\left(a_L + \sum_{i\in X}a_i\right) - \sum_{i\in[i_0]\cap X} \rho_i  - \sum_{i\in [i_0+1,|\mathcal{I}_H|]\cap X} \psi - \sum_{i \in [|\mathcal{I}_H|+1,j-1] \cap X} \gamma_i.\]

Lemma \ref{lemma:descending_gamma} captures the observation that the lifted coefficients $\gamma_j$ decreases as $j$ becomes larger.
\begin{lemma}
\label{lemma:descending_gamma}
For any $|\mathcal{I}_H|+1\leq j_1 < j_2 \leq n$, $\gamma_{j_1} \geq \gamma_{j_2}$. 
\end{lemma}
\begin{proof}
This result immediately follows from Proposition 1.3 on page 264 of \cite{wolsey1999integer}.
\end{proof}

We remind the readers that $L^t$ denotes the set of the first $t$ lower-weighted items in $N$. The next lemma argues that there exists an optimal solution support to the $j$-th lifting problem \eqref{eq:h_lifting_prob}, that is the concatenation of the first $t$ lower-weighted items and the first $s$ higher-weighted items for some $s,t\leq k-1$. 

\begin{lemma}
\label{lemma:h_lift_sol_form}
For any $j\in[|\mathcal{I}_H|+1, n]$, $\gamma_j = \gamma^{Q^*}$ for some $Q^*\subseteq [j-1]$ with  $|Q^*|\leq k-1$, such that $Q^*=L^t\cup [s]$. Specifically, $0\leq s\leq k-1$ and $0\leq t\leq |\mathcal{I}_L|$ such that $t+s\leq k-1$. 
\end{lemma}
\begin{proof}
This result follows from Lemmas \ref{lemma:descending_sepa} and \ref{lemma:descending_gamma}. We refer the readers to the proof of Lemma \ref{lemma:l_lift_sol_form} for more details.
\end{proof}

Thanks to Lemma \ref{lemma:h_lift_sol_form}, we know that the support with the lowest objective, among all the feasible solution supports in given special form, gives the optimal objective of \eqref{eq:h_lifting_prob}. Lemmas \ref{lemma:h_head_eta} and \ref{lemma:h_tail_eta} explore and compare the objectives  of these candidate solutions. 

\begin{lemma}
\label{lemma:h_head_eta} 
For any $0\leq t\leq \min\{|\mathcal{I}_L|, k-1\}$, $\gamma^{L^t\cup [s-1]} \geq \gamma^{L^t\cup [s]}$ for all $s\in [\min\{k-1-t, i_0\}]$. 
\end{lemma}
{
\begin{proof}
This proof follows the same arguments for the proof of Lemma \ref{lemma:l_head_eta}.
\end{proof}
}

\begin{lemma}
\label{lemma:h_tail_eta}
Recall that $i_0$ is the parameter used to construct the base SI. Suppose Assumption \ref{assumption} holds for this $i_0$. Let any $0\leq t\leq \min\{|\mathcal{I}_L|, k-1\}$ be given. If $k-1-t \geq i_0$, then $\gamma^{L^t\cup [s]} \geq \gamma^{L^t\cup [k-1-t]}$ for all $s\in[i_0, k-1-t]$.
\end{lemma}
\begin{proof}
We first deduce the following relation from Assumption \ref{assumption}:
\begingroup
\allowdisplaybreaks
\begin{align*} 
0 &\geq f(a_L + (i_0+1)a_H) - f(a_L + i_0a_H) - \psi\\
&\geq f(ma_L + (i_0+p)a_H) - f(ma_L + (i_0+p-1)a_H) - \psi, 
\end{align*}
\endgroup for any $m,p\geq 1$. Given any $L^t\cup [s]$ described in the lemma, 
\begingroup
\allowdisplaybreaks
\begin{align*} 
\gamma^{L^t\cup [s]} & = f\left(a_L + ta_L + sa_H\right) - \sum_{i=1}^{i_0} \rho_i - (s-i_0)\psi - \sum_{i = |\mathcal{I}_H|+1}^{|\mathcal{I}_H|+t} \gamma_i \\
& =  f\left((t+1)a_L + i_0a_H\right) + \sum_{i=1}^{s-i_0} \left[ f\left((t+1)a_L + (i_0+i)a_H\right) -f\left((t+1)a_L + (i_0+i-1)a_H\right) -\psi \right]\\
&\quad\quad - \sum_{i=1}^{i_0} \rho_i - \sum_{i = |\mathcal{I}_H|+1}^{|\mathcal{I}_H|+t} \gamma_i \\
& \geq f\left((t+1)a_L + i_0a_H\right) + \sum_{i=1}^{k-1-t-i_0} \left[ f\left((t+1)a_L + (i_0+i)a_H\right) -f\left((t+1)a_L + (i_0+i-1)a_H\right) -\psi \right]\\
&\quad\quad - \sum_{i=1}^{i_0} \rho_i - \sum_{i = |\mathcal{I}_H|+1}^{|\mathcal{I}_H|+t} \gamma_i  \quad\quad \text{(follows from Assumption \ref{assumption})}\\
 & = f\left(a_L + ta_L + (k-1-t)a_H\right) - \sum_{i=1}^{i_0} \rho_i - (k-1-t-i_0)\psi - \sum_{i = |\mathcal{I}_H|+1}^{|\mathcal{I}_H|+t} \gamma_i \\
 & = \gamma^{L^t\cup [k-1-t]}. 
\end{align*}
\endgroup
\end{proof}
In fact, $\gamma^{L^t\cup [s]}$ may be lower than $\gamma^{L^t\cup [k-1-t]}$ when Assumption \ref{assumption} is violated, despite the fact that its counterpart Lemma \ref{lemma:i_0_loses} is true in general. Lemma \ref{lemma:h_card_k-1} summarizes Lemmas \ref{lemma:h_head_eta} and \ref{lemma:h_tail_eta}. It establishes that, under Assumption \ref{assumption}, $L^t\cup[k-1-t]$ has the lowest objective among all the supports that contains exactly $t$ lower-weighted items. 

\begin{lemma}
\label{lemma:h_card_k-1}
Suppose Assumption \ref{assumption} holds for a given $i_0\in \{0,1,\dots, k-1\}$. For any $0\leq t\leq \min\{k-1, |\mathcal{I}_L|\}$, $\gamma^{L^t\cup [k-1-t]} \leq \gamma^{L^{t}\cup [s]}$ for any $t$ such that $L^{t}\cup [s]$ is a feasible solution support to the $(|\mathcal{I}_H|+t+1)$-th lifting problem. 
\end{lemma}
\begin{proof}
If $k-1-t\leq i_0$, then $0\leq s\leq k-1-t$. In this case, $\gamma^{L^t\cup [k-1-t]} \leq \gamma^{L^{t}\cup [s]}$ is immediate from Lemma \ref{lemma:h_head_eta}. Otherwise, suppose $s\geq i_0$ is viable. For any $0\leq s\leq i_0$, $\gamma^{L^t\cup [s]} \geq \gamma^{L^{t}\cup [i_0]}$. Then combining this observation with Lemma \ref{lemma:h_tail_eta}, we conclude that $\gamma^{L^t\cup [k-1-t]} \leq \gamma^{L^{t}\cup [s]}$.
\end{proof}

Next we provide the explicit form of the lifting coefficients. 
\begin{proposition}
\label{prop:h_sepa_coeff}
Suppose Assumption \ref{assumption} holds for a given $i_0\in \{0,1,\dots, k-1\}$. The exact lifting coefficients from the lifting problems \eqref{eq:h_lifting_prob} are \[\gamma_j = \begin{cases}
\gamma^{[k-1]}, & j = |\mathcal{I}_H|+1, \\
\min\{\gamma_{j-1}, \gamma^{L^t\cup [k-1-t]}\}, & j = |\mathcal{I}_H|+1+t, t\in[n-1-|\mathcal{I}_H|].
\end{cases}
\]
\end{proposition}
\begin{proof}
When $j = |\mathcal{I}_H|+1$, any feasible support contains only the higher-weighted items. Thus $\gamma_{|\mathcal{I}_H|+1} = \gamma^{[k-1]}$ immediately follows from Lemma \ref{lemma:h_card_k-1}. When $j > |\mathcal{I}_H|+1$, we represent $j$ as $|\mathcal{I}_H|+1+t$, where $t\in [n-1-|\mathcal{I}_H|]$. All the solutions to the $j$-th lifting problem \eqref{eq:h_lifting_prob} with $x_{j-1} = 0$, are feasible to the $j-1$-th lifting problem. Thus the objective evaluated at these solutions are no lower than $\gamma_{j-1}$. On the other hand, if $j-1$ is included in the support, then we know that all such solution supports have worse objective values than $L^{t}\cup [k-1-t]$ by Lemma \ref{lemma:h_card_k-1}. Therefore, $\min\{\gamma_{j-1}, \gamma^{L^t\cup [k-1-t]}\}$ is the lowest attainable objective value in the $j$-th lifting problem \eqref{eq:h_lifting_prob}. This completes the proof.
\end{proof}

\begin{corollary}
\label{eq:h_sepa_strength}
Based on any SI that is facet-defining for $\conv{\mathcal{P}^1_k(\mathcal{I}_H)}$, the  higher-SIs given by Proposition \ref{prop:h_sepa_coeff} are facet-defining for $\conv{\mathcal{P}^2_k}$. 
\end{corollary}

\section{Full Description of $\conv{\mathcal{P}^2_2}$}
\label{sect:conv}
In the previous sections, we propose the lifted-EPIs, the lower-SIs, and the  higher-SIs. These inequalities are shown to be facet-defining for $\conv{\mathcal{P}^2_k}$ under certain conditions. The readers may wonder to what extent these strong valid inequalities can narrow the relaxed feasible space toward its convex hull. To provide insights into this, we construct the convex hull of $\mathcal{P}^2_2$, with the help of the proposed inequalities, where $\mathcal{P}^2_2$ contains two types of weights and has two as its cardinality upper bound. \\

Throughout this section, we require Assumption \ref{assumption} to hold for $i_0 = 0$. In other words, 
\[f(a_L + a_H) - f(a_L) \leq \frac{f(2a_H)}{2}.\]
In Section \ref{sect:p22_valid_cuts}, we will first describe an additional single constraint, which we call the \textit{super-average inequality}, and prove its validity for $\conv{\mathcal{P}^2_2}$. Then we present the explicit forms of the lifted-EPIs, the lower-SIs, and the  higher-SIs specific to $\conv{\mathcal{P}^2_2}$. In Section \ref{sect:p22_polar}, we enumerate all the facets of $\conv{\mathcal{P}^2_2}$ by examining its polar. Lastly, we show that these proposed inequalities together with the 0-1 bounds and cardinality constraint fully characterize $\conv{\mathcal{P}^2_2}$ in Section \ref{sect:p22_final_conv}.

\subsection{Valid inequalities for $\conv{\mathcal{P}^2_2}$}
\label{sect:p22_valid_cuts}
We begin this subsection with a summary of the properties of $f$ that will be helpful for describing the valid inequalities. First, by the definition of concave functions, 
\begin{equation}
\label{eq:concave_def}
f(a_L + a_H)  = f\left(\frac{1}{2}\cdot 2a_L + \frac{1}{2}\cdot 2a_H\right)  \geq f(2a_L)/2 + f(2a_H)/2.  
\end{equation}
In addition, 
\begin{equation}
\begin{aligned}
\label{eq:one_l_item}
f(a_L) & = [f(a_L) + f(a_L) - f(0)]/2  \\
& \geq [f(a_L) + f(2a_L) - f(a_L)]/2 \quad \text{(because $f$ is concave)} \\
& = f(2a_L)/2.
\end{aligned}
\end{equation}
With exactly the same reasoning, we derive 
\begin{equation}
\label{eq:one_h_item}
f(a_H) \geq f(2a_H)/2. 
\end{equation}

Relying on these properties, we propose a new single valid inequality for $\mathcal{P}^2_2$.
\begin{proposition}
\label{prop:super-average}
The inequality
\begin{equation}
\label{eq:p2_super_average}
w\geq \sum_{i\in \mathcal{I}_L} \frac{f(2 a_L)}{2} x_i + \sum_{i\in \mathcal{I}_H} \frac{f(2 a_H)}{2} x_i
\end{equation}
is valid for $\mathcal{P}^2_2$. We call this inequality the super-average inequality. 
\end{proposition}
\begin{proof}
We need to show that inequality \eqref{eq:p2_super_average} is satisfied at all the feasible points of $\mathcal{P}^2_2$. In particular, it is sufficient to check validity at any point $(f(x),x)$ in $\mathcal{P}^2_2$. We represent such points by $P(S_1,S_2)$, where $S_1=\{i\in \mathcal{I}_L: x_i = 1\}$ and $S_2=\{i\in \mathcal{I}_H: x_i = 1\}$. These points fall into one of the following classes: $P(\emptyset,\emptyset)$, $P(\{i\},\emptyset)$, $P(\{i_1,i_2\},\emptyset)$, $P(\emptyset,\{j\})$, $P(\emptyset,\{j_1,j_2\})$, and $P(\{i\},\{j\})$, where $i ,i_1,i_2\in \mathcal{I}_L$ and $j,j_1,j_2\in \mathcal{I}_H$. We first observe that $P(\emptyset,\emptyset)$, $P(\{i_1,i_2\},\emptyset)$ and $P(\emptyset,\{j_1,j_2\})$ satisfy inequality \eqref{eq:p2_super_average} by construction. It follows from \eqref{eq:concave_def} that inequality \eqref{eq:p2_super_average} is valid for $P(\{i\},\{j\})$. Inequality \eqref{eq:p2_super_average} is also valid for $P(\{i\},\emptyset)$ and $P(\emptyset,\{j\})$ due to properties \eqref{eq:one_l_item} and \eqref{eq:one_h_item}, respectively. Therefore, inequality \eqref{eq:p2_super_average} is valid for $\mathcal{P}^2_2$.
\end{proof}
It is worth noting that the validity of inequality \eqref{eq:p2_super_average} does not require Assumption \ref{assumption}. In the inequalities we describe below, $l$ denotes an arbitrary lower-weighted item in $N$, and $h$ is any higher-weighted item. 

\begin{proposition}
Given Assumption \ref{assumption} for $i_0=0$, the lifted-EPIs for $\mathcal{P}^2_2$ are 
\begin{equation}
\label{eq:p2_EPI_form1}
w\geq f(a_L) x_l + \sum_{i\in \mathcal{I}_L\backslash \{l\}} [f(2a_L) - f(a_L)] x_i + \sum_{i\in \mathcal{I}_H} [f(a_L + a_H) - f(a_L)] x_i, 
\end{equation}
 and 
\begin{equation}
\label{eq:p2_EPI_form2}
w \geq f(a_H) x_h + \sum_{i\in \mathcal{I}_L} [f(a_L+a_H) - f(a_H)] x_i + \sum_{i\in \mathcal{I}_H\backslash \{h\}} [f(2a_H) - f(a_H)] x_i.
\end{equation}
\end{proposition}
\begin{proof}
Inequality \eqref{eq:p2_EPI_form1} is lifted from the EPI with respect to $\{l\}$ and any permutation $\delta$ of $N$, such that $\delta_1 = l$. By Proposition \ref{prop:EPI_lifted_coeff}, $\zeta_l = f(a_L)$ and for all $i\in\mathcal{I}_L\backslash \{l\}$, $\zeta_i = f(2a_L) - f(a_L)$. Let the first higher-weighted item in $\delta$ be $h$, then $\zeta_h = f(a_L+a_H) - f(a_L)$. The next higher-weighted item $i$ has coefficient $\min\{f(a_L+a_H) - f(a_L), f(2a_H) - [f(a_L+a_H) - f(a_L)]\}$. We observe the second candidate minus the former gives
\begingroup
\allowdisplaybreaks
\begin{align} 
\label{eq:candidate_diff}
f(2a_H) - 2[f(a_L+a_H) - f(a_L)] \geq 0 
\end{align}
\endgroup as a consequence of Assumption \ref{assumption}. Thus this second higher-weighted item also takes on the lifting coefficient $f(a_L+a_H) - f(a_L)$. Following the same reasoning, we can iteratively show that $\zeta_i = f(a_L+a_H) - f(a_L)$ for every $i\in\mathcal{I}_H$. \\

Inequality \eqref{eq:p2_EPI_form2} is lifted from the EPI with respect to $\{h\}$ and permutation $\delta$, in which $\delta_1 = h$. Again due to Proposition \ref{prop:EPI_lifted_coeff}, $\zeta_h = f(a_H)$ and $\zeta_i = f(a_L+a_H) - f(a_H)$ for all $i\in\mathcal{I}_L$. Moreover, the second higher-weighted item in $\delta$ has coefficient $f(2a_H)-f(a_H)$. Now, the third higher-weighted item in $\delta$ has the coefficient $\min\{f(2a_H)-f(a_H), f(2a_H) - [f(2a_H)-f(a_H)] = f(a_H)\}$, which is $f(2a_H)-f(a_H)$. This follows from Lemma \ref{lemma:f_concave}. Iteratively, we can apply the same reasoning to show that $\zeta_i = f(2a_H) - f(a_H)$ for every $i\in\mathcal{I}_H\backslash \{h\}$.
\end{proof}

As we noted in Corollary \ref{coro:lifted_EP_strength}, inequalities \eqref{eq:p2_EPI_form1} and \eqref{eq:p2_EPI_form2} are facet-defining for $\conv{\mathcal{P}^2_2}$. Next we state the explicit forms of the lower- and the  higher-SIs.

\begin{proposition}
Suppose Assumption \ref{assumption} holds for $i_0=0$. When $|\mathcal{I}_L| \geq 2$, the lower-SIs for $\mathcal{P}^2_2$ are 
\begin{equation}
\label{eq:p2_lower_average}
w \geq \sum_{i\in \mathcal{I}_L} \frac{f(2a_L)}{2} x_i + \left[f(a_L+a_H)-\frac{f(2a_L)}{2}\right] x_h + \sum_{i\in \mathcal{I}_H\backslash \{h\}} \left[f(2a_H)-f(a_L+a_H)+\frac{f(2a_L)}{2}\right] x_i,   
\end{equation}  for $i_0 = 0$, and for $i_0 = 1$,
\begin{equation}
\label{eq:p2_l_redundant}
w\geq f(a_L) x_l + \sum_{i\in \mathcal{I}_L\backslash \{l\}} [f(2a_L) - f(a_L)] x_i + \sum_{i\in \mathcal{I}_H} [f(a_L + a_H) - f(a_L)] x_i.
\end{equation}
\end{proposition}
\begin{proof}
We show that the coefficients constructed according to Proposition \ref{prop:l_sepa_coeff} is identical with those in the given inequalities. When $i_0=0$, $\psi = f(2a_L)/2$, which is the lifting coefficient for all $i\in\mathcal{I}_L$. We can interpret $h$ as the first higher-weighted item being lifted, then $\eta_h = f(a_L+a_H)-f(2a_L)/2$. The next higher-weighted item to be lifted takes on the coefficient $\min\{f(a_L+a_H)-f(2a_L)/2, f(2a_H)-f(a_L+a_H)+f(2a_L)/2\}$. We observe that 
\begingroup
\allowdisplaybreaks
\begin{align*} 
& f(a_L+a_H)-f(2a_L)/2 - [ f(2a_H)-f(a_L+a_H)+f(2a_L)/2 ] \\
& \geq 2f(a_L+a_H)-f(2a_L)-f(2a_H) \\
& \geq 0 \quad\quad \text{(due to \eqref{eq:concave_def}).}
\end{align*}
\endgroup Therefore, this higher-weighted item has lifting coefficient $f(2a_H)-f(a_L+a_H)+f(2a_L)/2$. Following the exact same reasoning, we infer that $\eta_i = f(2a_H)-f(a_L+a_H)+f(2a_L)/2$ for all $i\in\mathcal{I}_H\backslash \{h\}$. Hence the lower-SI with $i_0=0$ constructed according to Proposition \ref{prop:l_sepa_coeff} is the same as \eqref{eq:p2_lower_average}. \\

In the case of $i_0 = 1$, $\rho_l = f(a_L)$ and the remaining lower-weighted items take the coefficient $\psi = f(2a_L) - f(a_L)$. The first higher-weighted item $i$ has $\eta_i = f(a_L+a_H)-f(a_L)$. For the higher-weighted item $j$ right after $i$ in the permutation of $N$, $\eta_j = \min\{f(a_L+a_H)-f(a_L), f(2a_H) - [f(a_L+a_H)-f(a_L)]\} = f(a_L+a_H)-f(a_L)$. We have justified this in \eqref{eq:candidate_diff} which relies on Assumption \ref{assumption}. Hence, $\eta_i = f(a_L+a_H)-f(a_L)$ for all $i\in\mathcal{I}_H$, and inequality \eqref{eq:p2_l_redundant} is exactly the lower-SI with $i_0=1$. 
\end{proof}

\begin{proposition}
Suppose Assumption \ref{assumption} holds for $i_0=0$. When $|\mathcal{I}_H| \geq 2$, the  higher-SIs are 
\begin{equation}
\label{eq:p2_higher_average}
w\geq \left[f(a_L + a_H) - \frac{f(2 a_H)}{2}\right] x_l + \sum_{i \in\mathcal{I}_L\backslash \{l\}} \left[f(2a_L) - f(a_L + a_H) + \frac{f(2 a_H)}{2}\right] x_i +  \sum_{i\in\mathcal{I}_H} \frac{f(2 a_H)}{2} x_i
\end{equation} for $i_0 = 0$, and for $i_0 = 1$,
\begin{equation}
\label{eq:p2_h_redundant}
w \geq f(a_H) x_h + \sum_{i\in \mathcal{I}_L} [f(a_L+a_H) - f(a_H)] x_i + \sum_{i\in \mathcal{I}_H\backslash \{h\}} [f(2a_H) - f(a_H)] x_i.
\end{equation} 
\end{proposition}
\begin{proof}
We construct the  higher-SIs as given in Proposition \ref{prop:h_sepa_coeff}, then show that they match the given inequalities. Recall that Proposition \ref{prop:h_sepa_coeff} is true when Assumption \ref{assumption} for $i_0=0$ holds. When $i_0 = 0$, all the higher-weighted items have coefficient $\psi = f(2a_H)/2$. Suppose $l$ is the first lower-weighted item in a fixed permutation of $N$. Then $\gamma_l = f(a_L+a_H) -f(2a_H)/2$. Let $i\in\mathcal{I}_L$ be right after $l$. The lifting coefficient $\gamma_i = \min\{ f(a_L+a_H) -f(2a_H)/2, f(2a_L) - [ f(a_L+a_H) -f(2a_H)/2]\}$. We examine the difference between the two candidates, which is 
\begingroup
\allowdisplaybreaks
\begin{align*} 
& f(a_L+a_H)-f(2a_H)/2 - [ f(2a_L) - f(a_L+a_H) + f(2a_H)/2 ] \\
& \geq 2f(a_L+a_H)-f(2a_L)-f(2a_H) \\
& \geq 0 \quad\quad \text{(by \eqref{eq:concave_def}).}
\end{align*} 
\endgroup Thus $\gamma_i = f(2a_L) - f(a_L+a_H) + f(2a_H)/2$. In fact, we can iteratively show that all the lifting coefficients for the lower-weighted items are $f(2a_L) - f(a_L+a_H) + f(2a_H)/2$ with the same argument. Therefore, inequality \eqref{eq:p2_higher_average} is correct. \\

Now suppose $i_0 = 1$. The corresponding SI over $\mathcal{I}_H$ is $w\geq f(a_H)x_h + \sum_{i\in\mathcal{I}_H\backslash \{h\}} [f(2a_H) - f(a_H)] x_i$. By Proposition \ref{prop:h_sepa_coeff}, the coefficient of the first lifted lower-weighted item is $f(a_L+a_H) - f(a_H)$. Let $j\in\mathcal{I}_L$ be the second lifted item. Then $\gamma_j = \min\{f(a_L+a_H) - f(a_H), f(2a_L) - f(a_L+a_H) + f(a_H)\}$. Furthermore, 
\begingroup
\allowdisplaybreaks
\begin{align*} 
&  f(2a_L) - f(a_L+a_H) + f(a_H) - [f(a_L+a_H) - f(a_H)] \\
& = [f(a_L)- f(0)] + [f(2a_L) - f(a_L)] - [f(a_L+a_H) - f(a_H)] - [f(a_L+a_H) - f(a_H)] \\
& \geq  [f(a_L)- f(0)] + [f(2a_L) - f(a_L)] -  [f(a_L)- f(0)] -  [f(2a_L) - f(a_L)] \quad\quad \text{(due to Lemma \ref{lemma:f_concave})}\\ 
& = 0. 
\end{align*} 
\endgroup
Therefore $\gamma_j = f(a_L+a_H) - f(a_H)$. By iteratively applying the same argument, we conclude that the for all $i\in\mathcal{I}_L$, $\gamma_i = f(a_L+a_H) - f(a_H)$. Hence inequality \eqref{eq:p2_h_redundant} is the higher-SI with $i_0=1$. 
\end{proof}

Note that the lower-SI \eqref{eq:p2_l_redundant} and the higher-SI \eqref{eq:p2_h_redundant}  coincide with the lifted-EPIs \eqref{eq:p2_EPI_form1} and \eqref{eq:p2_EPI_form2}. To avoid confusion, we will refer to these inequalities as the lifted-EPIs, and refer to inequalities \eqref{eq:p2_lower_average} and \eqref{eq:p2_higher_average} as the lower- and  higher-SIs, respectively. \\

Lastly, the trivial inequalities
\begin{equation}
\label{eq:p2_trivial}
0\leq x_i \leq 1, \quad \text{for all $i\in N$}, 
\end{equation}
and the cardinality constraint
\begin{equation}
\label{eq:p2_cardinality}
\sum_{i\in N} x_i \leq 2, 
\end{equation} are naturally valid for $\mathcal{P}^2_2$.

\subsection{Polarity and facets of $\conv{\mathcal{P}^2_2}$}
\label{sect:p22_polar}
Our next goal is to prove that the inequalities provided in Section \ref{sect:p22_valid_cuts} fully describe $\conv{\mathcal{P}^2_2}$. We show this by enumerating the extreme rays of the polar $\Pi$ of $\conv{\mathcal{P}^2_2}$, where 
\[\Pi = \{(-\pi_w, \pi, -\pi_0)\in\mathbb{R}^{n+2}: -\pi_w w + \pi^\top x + \pi_0 \leq 0, \forall (w,x)\in \conv{\mathcal{P}^2_2}\}.\]
It is well-known that, for any full-dimensional polyhedron, any non-zero element of its polar is an extreme ray, if and only if the corresponding inequality is a facet of the polyhedron (see Theorem 5.2, pg. 99 \cite{wolsey1999integer}). In our context, if $\conv{\mathcal{P}^2_2}$ is full dimensional, then $\pi_w w\geq \pi_0 + \pi^\top x$ is a facet of $\conv{\mathcal{P}^2_2}$ if and only if $(-\pi_w, \pi, -\pi_0)$ is an extreme ray of $\Pi$, where $(-\pi_w, \pi)\neq \mathbf{0}$. The trivial inequalities \eqref{eq:p2_trivial} and the cardinality constraint \eqref{eq:p2_cardinality} are the trivial facets of $\conv{\mathcal{P}^2_2}$. To obtain all the other non-trivial facets of $\conv{\mathcal{P}^2_2}$, it is sufficient to find all the optimal solutions to problem \eqref{eq:separation_primal} given any $(\overline{w}, \overline{x})\in \mathbb{R}\times [0,1]^n$ such that $\sum_{i=1}^n \overline{x}_i \leq 2$. 

\begingroup
\allowdisplaybreaks
\begin{subequations}
\label{eq:separation_primal}
\begin{alignat}{2}
 \max \hspace{0.2cm} & \pi_0 + \sum_{i\in N} \overline{x}_i\pi_i && \\
\textrm{s.t.} \quad & \pi_0 + \sum_{i\in S} \pi_i \leq f\left(\sum_{i\in S} a_i\right), && \text{ for all $S\subseteq N$ with $|S|\leq 2$}. \label{constr:primal}
\end{alignat}
\end{subequations}
\endgroup

This is because all such optimal solutions are the desired extreme rays of $\Pi$. We note that $\pi_w \geq 0$ because $(1, \mathbf{0})$ is the recession direction of $\conv{\mathcal{P}^2_2}$. Therefore, $\pi_w$ is normalized to one in problem \eqref{eq:separation_primal} to avoid unboundedness.  In this subsection, we first show that $\conv{\mathcal{P}^2_2}$ is full-dimensional. Then we prove that the optimal solutions to problem \eqref{eq:separation_primal} for any $(\overline{w}, \overline{x})\in \mathbb{R}\times [0,1]^n$ with $\sum_{i=1}^n \overline{x}_i \leq 2$, are exactly the coefficients of the proposed non-trivial inequalities  \eqref{eq:p2_super_average},  \eqref{eq:p2_EPI_form1}, \eqref{eq:p2_EPI_form2}, \eqref{eq:p2_lower_average} and  \eqref{eq:p2_higher_average}.

\begin{proposition}
The polyhedron $\conv{\mathcal{P}^2_2}$ is full-dimensional. 
\end{proposition}
\begin{proof}
Let $\mathbf{0}\in\mathbb{R}^n$ be a zero vector, and $e^i\in\mathbb{R}^n$ be a vector with 1 in the $i$-th entry and 0 everywhere else. The points $(0,\mathbf{0})$, $(1,\mathbf{0})$, and $\{(f(a^\top e^i), e^i)\}_{i=1}^{n}$ all lie in $\conv{\mathcal{P}^2_2}$ and are affine independent. Hence, dim$(\conv{\mathcal{P}^2_2})=n+1$. 
\end{proof}

We proceed to enumerate the optimal solutions to problem \eqref{eq:separation_primal}. For any $\overline{x}\in [0,1]^n$, we define $l = \arg \max_{i\in\mathcal{I}_L} \overline{x}_i$ and $h = \arg\max_{i\in \mathcal{I}_H} \overline{x}_i$. We partition the set of $(\overline{w}, \overline{x})\in \mathbb{R}\times [0,1]^n$ with $\sum_{i=1}^n \overline{x}_i \leq 2$ into the following five subsets, where $\overline{x}$ additionally satisfies
\begin{itemize}
\item[(c1)] $2\overline{x}_l > \sum_{i\in N} \overline{x}_i$;
\item[(c2)] $2\overline{x}_h > \sum_{i\in N} \overline{x}_i$; 
\item[(c3)] $2\overline{x}_l < \sum_{i\in \mathcal{I}_L} \overline{x}_i$ and $2\overline{x}_h < \sum_{i\in \mathcal{I}_H} \overline{x}_i$; 
\item[(c4)] $\sum_{i\in \mathcal{I}_L} \overline{x}_i \leq 2\overline{x}_l \leq \sum_{i\in N} \overline{x}_i$, $2\overline{x}_h \leq \sum_{i\in N} \overline{x}_i$, and $2\overline{x}_h - \sum_{i\in \mathcal{I}_H} \overline{x}_i \leq 2\overline{x}_l - \sum_{i\in \mathcal{I}_L} \overline{x}_i$;
\item[(c5)] $\sum_{i\in \mathcal{I}_H} \overline{x}_i \leq 2\overline{x}_h \leq \sum_{i\in N} \overline{x}_i$, $2\overline{x}_l \leq \sum_{i\in N} \overline{x}_i$, and $2\overline{x}_l - \sum_{i\in \mathcal{I}_L} \overline{x}_i \leq 2\overline{x}_h - \sum_{i\in \mathcal{I}_H} \overline{x}_i$.
\end{itemize}
These subsets are pairwise disjoint and their union is the original set of $(\overline{w},\overline{x})$. Later, we may refer to these subsets as categories as well. Given any $(\overline{w}, \overline{x})$ from each subset, we show that the corresponding problem \eqref{eq:separation_primal} has the coefficients of one of the five classes of inequalities, \eqref{eq:p2_super_average},  \eqref{eq:p2_EPI_form1}, \eqref{eq:p2_EPI_form2}, \eqref{eq:p2_lower_average} and  \eqref{eq:p2_higher_average}, as its optimal solution. {Problem} \eqref{eq:separation_dual} is the dual problem of problem \eqref{eq:separation_primal}, where $Q(S)$ is the dual variable associated with constraint \eqref{constr:primal}.
\begingroup
\allowdisplaybreaks
\begin{subequations}
\label{eq:separation_dual}
\begin{alignat}{2}
{\min} \hspace{0.2cm} & \sum_{S\subseteq N, |S|\leq 2} Q(S) f\left(\sum_{i\in S} a_i\right) && \label{eq:separation_dual_obj}\\
\textrm{s.t.} \quad & \sum_{S:|S|\leq 2, S\ni i} Q(S) = \overline{x}_i , && \text{ for all $i\in N$, } \label{eq:separation_dual_x}\\
& \sum_{S:|S|\leq 2} Q(S) = 1, && \label{eq:separation_dual_one} \\
& Q(S) \geq 0, && \text{ for all $S$ such that $|S|\leq 2$.} \label{eq:separation_dual_nonnegative}
\end{alignat}
\end{subequations}
\endgroup
This dual linear program is crucial in the succeeding discussions because we will use strong duality to show the optimality of the proposed primal feasible solutions. \\

Now, let any $(\overline{w}, \overline{x})\in \mathbb{R}\times [0,1]^n$ with $\sum_{i=1}^n \overline{x}_i \leq 2$ be given. Recall that $l = \arg \max_{i\in\mathcal{I}_L} \overline{x}_i$ and $h = \arg\max_{i\in \mathcal{I}_H} \overline{x}_i$.

\begin{proposition}(c1) 
\label{prop:separate_EPI-L}
If $2\overline{x}_l > \sum_{i\in N} \overline{x}_i$, then 
\[ \overline{\pi}_i = \begin{cases} 0, & i = 0, \\
f(a_L), & i = l, \\
f(2a_L)-f(a_L), & i\in\mathcal{I}_L\backslash \{l\}, \\
f(a_L+a_H)-f(a_L), & i\in\mathcal{I}_H,
\end{cases} \] is an optimal solution to problem \eqref{eq:separation_primal} associated with $\overline{x}$. This $\overline{\pi}$ corresponds to the coefficients of the lifted-EPI \eqref{eq:p2_EPI_form1}. 
\end{proposition}
\begin{proof}
Given that inequality \eqref{eq:p2_EPI_form1} is valid for $\mathcal{P}^2_2$, $\overline{\pi}$ is a feasible solution to problem \eqref{eq:separation_primal}. To prove its optimality, we first propose a solution to the dual problem \eqref{eq:separation_dual}:
\[\overline{Q}(S) = 
\begin{cases}
 \overline{x}_i, & S = \{l,i\}, i\in N\backslash\{l\}, \\
2 \overline{x}_l - \sum_{j\in N} \overline{x}_j, & S = \{l\}, \\
1- \overline{x}_l, & S = \emptyset, \\
0, & \text{for all other $S\subseteq N$ with $|S|\leq 2$.} 
\end{cases}
\] Since $\overline{x}\in [0,1]^n$ and $2\overline{x}_l > \sum_{i\in N} \overline{x}_i$, $\overline{Q}(S) \geq 0$ for all $S\subseteq N$ with $|S|\leq 2$. Thus constraint \eqref{eq:separation_dual_nonnegative} is satisfied. We observe that 
\[ \sum_{S:|S|\leq 2} \overline{Q}(S) = \sum_{i\in N\backslash\{l\}}\overline{x}_i + 2 \overline{x}_l - \sum_{j\in N} \overline{x}_j + 1- \overline{x}_l = 1,\] so constraint \eqref{eq:separation_dual_one} is also satisfied by the proposed solution. For any $i\in N\backslash \{l\}$, 
\[ \sum_{S:|S|\leq 2, S\ni i} \overline{Q}(S) = \overline{Q}(\{l,i\}) =  \overline{x}_i.\] In addition, 
\[ \sum_{S:|S|\leq 2, S\ni l} \overline{Q}(S) = \sum_{i\in N\backslash\{l\}} \overline{Q}(\{l,i\}) + \overline{Q}(\{l\}) = \sum_{i\in N\backslash\{l\}}\overline{x}_i + 2 \overline{x}_l - \sum_{j\in N} \overline{x}_j = \overline{x}_l.\] Therefore, \eqref{eq:separation_dual_x} is satisfied, and $\overline{Q}(\cdot)$ is a feasible solution to the dual problem \eqref{eq:separation_dual}. \\

The objective of \eqref{eq:separation_primal} evaluated at $\overline{\pi}$ is  \[ f(a_L) \overline{x}_l +  \sum_{i\in \mathcal{I}_L\backslash\{l\}}  [f(2a_L)-f(a_L)]\overline{x}_i + \sum_{i\in \mathcal{I}_H}  [f(a_L+a_H)-f(a_L)] \overline{x}_i. \] The dual objective evaluated at $\overline{Q}(\cdot)$ is 
\begin{align}
& \quad \quad f(0)\overline{Q}(\emptyset) + f(a_L) \overline{Q}(\{l\}) + f(2a_L) \sum_{i\in \mathcal{I}_L\backslash\{l\}}\overline{Q}(\{l, i\}) + f(a_L+a_H)\sum_{i\in \mathcal{I}_H} \overline{Q}(\{l, i\}) \\
&= f(a_L) (2 \overline{x}_l - \sum_{j\in N} \overline{x}_j ) +  \sum_{i\in \mathcal{I}_L\backslash\{l\}}f(2a_L)\overline{x}_i + \sum_{i\in \mathcal{I}_H}f(a_L+a_H)\overline{x}_i \\
&=  f(a_L) \left(2\overline{x}_l - \overline{x}_l - \sum_{j\in \mathcal{I}_L\backslash\{l\}} \overline{x}_j - \sum_{j\in \mathcal{I}_H} \overline{x}_j\right) +  \sum_{i\in \mathcal{I}_L\backslash\{l\}}f(2a_L)\overline{x}_i + \sum_{i\in \mathcal{I}_H}f(a_L+a_H)\overline{x}_i \\
& = f(a_L)\overline{x}_l +  \sum_{i\in \mathcal{I}_L\backslash\{l\}}[f(2a_L)-f(a_L)]\overline{x}_i + \sum_{i\in \mathcal{I}_H}[f(a_L+a_H)-f(a_L)]\overline{x}_i,
\end{align} which is identical with the primal objective at $\overline{\pi}$. By strong duality, $\overline{\pi}$ is optimal in problem \eqref{eq:separation_primal}. 
\end{proof}

\begin{proposition} (c2) 
\label{prop:separate_EPI-H}
 If $2\overline{x}_h > \sum_{i\in N} \overline{x}_i$, then 
\[ \overline{\pi}_i = \begin{cases} 0, & i = 0, \\
f(a_H), & i = h, \\
f(a_L+a_H)-f(a_H), & i\in\mathcal{I}_L, \\
f(2a_H)-f(a_H), & i\in\mathcal{I}_H\backslash \{h\}, 
\end{cases} \] is an optimal solution to problem \eqref{eq:separation_primal} associated with $\overline{x}$. This optimal solution is exactly the coefficients of the lifted-EPI \eqref{eq:p2_EPI_form2}.
\end{proposition}
\begin{proof}
The proposed solution is feasible in \eqref{eq:separation_primal} due to the validity of inequality \eqref{eq:p2_EPI_form2} for $\mathcal{P}^2_2$. Similar to the proof of Lemma \ref{prop:separate_EPI-H}, we again construct a dual solution as the following: 
\[\overline{Q}(S) = 
\begin{cases}
 \overline{x}_i, & S = \{h,i\}, i\in N\backslash\{h\}, \\
2 \overline{x}_h - \sum_{j\in N} \overline{x}_j, & S = \{h\}, \\
1- \overline{x}_h, & S = \emptyset, \\
0, & \text{for all other $S\subseteq N$ with $|S|\leq 2$.} 
\end{cases}
\] 
First we show the feasibility of $\overline{Q}(\cdot)$. Given that $\overline{x}\in [0,1]^n$ and $2\overline{x}_l > \sum_{i\in N} \overline{x}_i$, constraint \eqref{eq:separation_dual_nonnegative} is satisfied. In addition, 
\[ \sum_{S:|S|\leq 2} \overline{Q}(S) = \sum_{i\in N\backslash\{h\}}\overline{x}_i + 2 \overline{x}_h - \sum_{j\in N} \overline{x}_j + 1- \overline{x}_h = 1,\] indicating that constraint \eqref{eq:separation_dual_one} is also satisfied. For any $i\in N\backslash \{h\}$, 
\[ \sum_{S:|S|\leq 2, S\ni i} \overline{Q}(S) = \overline{Q}(\{h,i\}) =  \overline{x}_i.\] Moreover, 
\[ \sum_{S:|S|\leq 2, S\ni h} \overline{Q}(S) = \sum_{i\in N\backslash\{h\}} \overline{Q}(\{h,i\}) + \overline{Q}(\{h\}) = \sum_{i\in N\backslash\{h\}}\overline{x}_i + 2 \overline{x}_h - \sum_{j\in N} \overline{x}_j = \overline{x}_h.\] Hence, constraints \eqref{eq:separation_dual_x} are satisfied. \\

The objective of \eqref{eq:separation_dual} evaluated at $\overline{Q}(\cdot)$ is 
\begin{align*}
& \quad \quad f(0)\overline{Q}(\emptyset) + f(a_H) \overline{Q}(\{h\}) + f(2a_H) \sum_{i\in \mathcal{I}_H\backslash\{h\}}\overline{Q}(\{h, i\}) + f(a_L+a_H)\sum_{i\in \mathcal{I}_L} \overline{Q}(\{h, i\}) \\
&=  f(a_H) \left(2\overline{x}_h - \overline{x}_h- \sum_{j\in \mathcal{I}_H\backslash\{h\}} \overline{x}_j - \sum_{j\in \mathcal{I}_L} \overline{x}_j\right) +  \sum_{i\in \mathcal{I}_H\backslash\{h\}}f(2a_H)\overline{x}_i + \sum_{i\in \mathcal{I}_L}f(a_L+a_H)\overline{x}_i \\
& = f(a_H)\overline{x}_h  + \sum_{i\in \mathcal{I}_L}[f(a_L+a_H)-f(a_H)]\overline{x}_i +  \sum_{i\in \mathcal{I}_H\backslash\{h\}}[f(2a_H)-f(a_H)]\overline{x}_i\\
& = \overline{\pi}_0 + \sum_{i\in N} \overline{\pi}_i \overline{x}_i. 
\end{align*}
By strong duality, we conclude that $\overline{\pi}$ is optimal in problem \eqref{eq:separation_primal}. 
\end{proof}

Before characterizing the optimal solution to problem \eqref{eq:separation_primal} for category (c3) of $(\overline{w}, \overline{x})$, we state a useful lemma.

\begin{lemma}
\label{lemma:average_primal_dual_sol} (\citet{yu2017polyhedral} Lemma 4 and Proposition 5)
Suppose $N = \mathcal{I}_L$ or $N = \mathcal{I}_H$. In either case, we denote $a_L$, or $a_H$, by $\alpha$. Let any $(\overline{w},\overline{x})\in\mathbb{R}\times [0,1]^n$ that satisfies $\sum_{i\in N} \overline{x}_i \leq 2$ be given, in which $\overline{x}^{\max} = \max_{i\in N} \overline{x}_i$. If $2\overline{x}^{\max} < \sum_{i\in N}\overline{x}_i$, then 
\begin{equation}
\overline{\pi}_i = \begin{cases}
0, & i=0,\\
f(2\alpha)/2, & i\in N,
\end{cases} 
\end{equation} is an optimal solution to the primal problem \eqref{eq:separation_primal} associated with $\overline{x}$. There exists a corresponding optimal solution $\overline{Q}(S)$, for all $S\subseteq N$ with $|S|\leq 2$, to the dual problem \eqref{eq:separation_dual}; in particular, $\overline{Q}(\emptyset) = 1-\sum_{i\in N}\overline{x}_i /2$. 
\end{lemma}

\begin{proposition} (c3) 
\label{prop:separate_super_average}
If $2\overline{x}_h < \sum_{i\in \mathcal{I}_H} \overline{x}_i$ and $2\overline{x}_l < \sum_{i\in \mathcal{I}_L} \overline{x}_i$, then 
\[ \overline{\pi}_i = \begin{cases} 0, & i = 0, \\
f(2a_L)/2, & i \in \mathcal{I}_L, \\
f(2a_H)/2, & i\in\mathcal{I}_H, 
\end{cases} \] is an optimal solution to problem \eqref{eq:separation_primal} associated with $\overline{x}$. This  solution corresponds to the super-average inequality \eqref{eq:p2_super_average}.
\end{proposition}
\begin{proof}
Feasibility of $\overline{\pi}$ follows from validity of inequality \eqref{eq:p2_super_average}. Next we construct a solution to problem  \eqref{eq:separation_dual}. Let $\overline{x}^L$ be a sub-vector of $\overline{x}$ that contains only $\overline{x}_i$, for all $i\in \mathcal{I}_L$, and we define $\overline{x}^H$ similarly. \\

By Lemma \ref{lemma:average_primal_dual_sol}, problem \eqref{eq:separation_primal} associated with $\overline{x}^J$ has optimal objective $f(2a_J)/2\cdot \sum_{i\in \mathcal{I}_J} \overline{x}_i$, for $J\in \{L, H\}$. Again for $J\in \{L, H\}$, an optimal dual solution in problem \eqref{eq:separation_dual} associated with $\overline{x}^J$, which we denote by $\overline{Q}^J(\cdot)$, attains the same objective. Given the feasibility of $\overline{Q}^J(S)$ in the corresponding dual problem \eqref{eq:separation_dual}, the following properties hold: 
\begin{itemize}
\item $\overline{Q}^J(S)\geq 0$ for any $S\subseteq \mathcal{I}_J$ such that $|S|\leq 2$, 
\item $\sum_{S\subseteq \mathcal{I}_J, |S|\leq 2} \overline{Q}^J(S) = 1$, 
\item $\sum_{S\subseteq \mathcal{I}_J, |S|\leq 2, S\ni i} \overline{Q}^J(S) = \overline{x}_i$ for any $i\in \mathcal{I}_J$,
\item $\overline{Q}^J(\emptyset) = 1-\sum_{i\in \mathcal{I}_J} \overline{x}_i/2$, 
\end{itemize} where $J\in \{L, H\}$. We claim that 
\[ \overline{Q}(S) = \begin{cases}
1-\sum_{i\in N} \overline{x}_i/2, & S = \emptyset \\
\overline{Q}^L(S), & S\subseteq \mathcal{I}_L, 1\leq |S|\leq 2 \\
\overline{Q}^H(S), & S\subseteq \mathcal{I}_H, 1\leq |S|\leq 2 \\
0, & \text{for all other $S\subseteq N$ with $|S|\leq 2$.} 
\end{cases} \] is optimal in the dual problem \eqref{eq:separation_dual} associated with $\overline{x}$. Since $\overline{x}\in[0,1]^n$ and $\sum_{i\in N} \overline{x}_i \leq 2$, $1-\sum_{i\in N} \overline{x}_i/2 \geq 0$. Given the non-negativity of $\overline{Q}^L(S)$ and $\overline{Q}^H(S)$, $\overline{Q}(S)$ satisfies constraint \eqref{eq:separation_dual_nonnegative}. Next, we check for constraint \eqref{eq:separation_dual_one}. 
\begin{align*}
\sum_{S:|S|\leq 2} \overline{Q}(S) & =  \overline{Q}(\emptyset) + \sum_{S\subseteq \mathcal{I}_L:1\leq |S|\leq 2} \overline{Q}^L(S) + \sum_{S\subseteq \mathcal{I}_H:1\leq |S|\leq 2} \overline{Q}^H(S) + 0 \\
& =  \overline{Q}(\emptyset) + \sum_{S\subseteq \mathcal{I}_L:|S|\leq 2} \overline{Q}^L(S) -  \overline{Q}^L(\emptyset) + \sum_{S\subseteq \mathcal{I}_H:|S|\leq 2} \overline{Q}^H(S) - \overline{Q}^H(\emptyset) \\
& = 1-\sum_{i\in N} \overline{x}_i/2  + (1- 1+\sum_{i\in \mathcal{I}_L} \overline{x}_i/2) + (1- 1+ \sum_{i\in \mathcal{I}_H} \overline{x}_i/2) \\
& = 1.
\end{align*}

For any $i\in \mathcal{I}_L$, 
\[ \sum_{S:|S|\leq 2, S\ni i} \overline{Q}(S) = \sum_{S\subseteq \mathcal{I}_L:|S|\leq 2, S\ni i} \overline{Q}^L(S) + 0 = \overline{x}_i; \] and for any $i\in \mathcal{I}_H$, 
\[ \sum_{S:|S|\leq 2, S\ni i} \overline{Q}(S) = \sum_{S\subseteq \mathcal{I}_H:|S|\leq 2, S\ni i} \overline{Q}^H(S) + 0 = \overline{x}_i. \] Hence, constraints \eqref{eq:separation_dual_x} are satisfied. \\

To show the optimality of $\overline{\pi}$, we note that the objective of \eqref{eq:separation_dual} evaluated at $\overline{Q}(\cdot)$ is 
\begin{align*}
& \quad \quad \sum_{S\subseteq N: |S|\leq 2} \overline{Q}(S) f\left(\sum_{i\in S}a_i\right) \\
& = \overline{Q}(\emptyset)\cdot f(0) + \sum_{S\subseteq N: 1\leq |S|\leq 2} \overline{Q}(S) f\left(\sum_{i\in S}a_i\right) \\ 
& = 0 + \sum_{S\subseteq \mathcal{I}_L: |S|\leq 2} \overline{Q}^L(S) f\left(\sum_{i\in S}a_L\right) -  \overline{Q}^L(\emptyset) \cdot f(0)+  \sum_{S\subseteq \mathcal{I}_H: |S|\leq 2} \overline{Q}^H(S) f\left(\sum_{i\in S}a_H\right) -  \overline{Q}^H(\emptyset) \cdot f(0) \\
& = f(2a_L)/2\cdot \sum_{i\in \mathcal{I}_L} \overline{x}_i + f(2a_H)/2\cdot \sum_{i\in \mathcal{I}_H} \overline{x}_i.
\end{align*} The last inequality holds due to the optimality of $\overline{Q}^L(S)$ and $\overline{Q}^H(S)$. This objective value coincides with the objective of \eqref{eq:separation_primal} at $\overline{\pi}$. By strong duality, $\overline{\pi}$ is optimal in problem \eqref{eq:separation_primal}. 
\end{proof}

We next show a lemma and its corollary, which are crucial to characterizing the optimal solution to the primal problem \eqref{eq:separation_primal} associated with $\overline{x}$ in category (c4).

\begin{lemma}
\label{lemma:h_average_sepa_subproblem_primal_bounded}
If $\overline{x}$ falls in category (c4), then problem \eqref{eq:h_average_sepa_subproblem_primal} is feasible with a bounded optimal objective. 
\end{lemma} 
\begingroup
\allowdisplaybreaks
\begin{subequations}
\label{eq:h_average_sepa_subproblem_primal}
\begin{alignat}{2}
\min \hspace{0.2cm} & \left(2\overline{x}_l-\sum_{i\in\mathcal{I}_L} \overline{x}_i\right)y+ \sum_{i\in\mathcal{I}_H} \overline{x}_ir_i&& \label{eq:h_average_sepa_subproblem_primal_obj}\\
\textrm{s.t.} \quad & y + r_i \geq 0 , && \text{ for all $i\in \mathcal{I}_H$, } \label{eq:h_average_sepa_subproblem_primal_yr}\\
& r_i + r_j \geq 0, && \text{ for all $i,j\in\mathcal{I}_H$ such that $i<j$.} \label{eq:h_average_sepa_subproblem_primal_rr}
\end{alignat}
\end{subequations}
\endgroup
\begin{proof}
We first note that $y=0$, $r_i = 0$ for all $i\in\mathcal{I}_H$ is a feasible solution, so it suffices to show that \eqref{eq:h_average_sepa_subproblem_primal} does not have a feasible and objective-improving ray. For a contradiction, we assume that such a ray, $d\in\mathbb{R}^{1+|\mathcal{I}_H|}$, exists. We denote its entries by $d_y$ and $d_{r_i}$ for $i\in\mathcal{I}_H$. Then given any feasible solution $\overline{y}$ and $\overline{r}_i$ where $i\in\mathcal{I}_H$, the following properties hold for any $\lambda\in\mathbb{R}_+$: 
\[\overline{y} + \overline{r}_i +\lambda (d_y +d_{r_i}) \geq 0 \text{ for all $i\in \mathcal{I}_H$; }\]
\[r_i + r_j +\lambda (d_{r_i} +d_{r_j})\geq 0  \text{ for all $i,j\in\mathcal{I}_H$ such that $i<j$;}\]
\[ \lambda\left[ \left(2\overline{x}_l-\sum_{i\in\mathcal{I}_L} \overline{x}_i\right)d_y+ \sum_{i\in\mathcal{I}_H} \overline{x}_i d_{r_i} \right] < 0.\] 
It follows that $d$ must satisfy 
\begin{equation}
\label{eq:ray_property_1}
d_y + d_{r_i} \geq 0, \text{ for all $i\in \mathcal{I}_H$,}
\end{equation}
\begin{equation}
\label{eq:ray_property_2}
d_{r_i} +d_{r_j} \geq 0, \text{ for all $i,j\in\mathcal{I}_H$ such that $i<j$,}
\end{equation}
\begin{equation}
\label{eq:ray_property_3}
\left(2\overline{x}_l-\sum_{i\in\mathcal{I}_L} \overline{x}_i\right)d_y+ \sum_{i\in\mathcal{I}_H} \overline{x}_i d_{r_i} < 0. 
\end{equation}
Since $2\overline{x}_l \geq \sum_{i\in\mathcal{I}_L}\overline{x}_i$ and $\overline{x}_i \geq 0$ for all $i\in \mathcal{I}_H$ by assumption, we infer from \eqref{eq:ray_property_3} that $d$ contains at least one strictly negative entry. \\

If $d_y <0$, then by \eqref{eq:ray_property_1}, $d_{r_i}>0$ for all $i\in\mathcal{I}_H$. In this case, 
\begingroup
\allowdisplaybreaks
\begin{align*}
0 & >  \left(2\overline{x}_l-\sum_{i\in\mathcal{I}_L} \overline{x}_i\right)d_y+ \sum_{i\in\mathcal{I}_H} \overline{x}_i d_{r_i} \\
&\geq \left(2\overline{x}_l-\sum_{i\in\mathcal{I}_L} \overline{x}_i\right)d_y+ \sum_{i\in\mathcal{I}_H} \overline{x}_i (-d_y)  \text{  because $d_{r_i} \geq -d_y$ by \eqref{eq:ray_property_1} }\\ 
& =  d_y \left(2\overline{x}_l-\sum_{i\in\mathcal{I}_L} \overline{x}_i - \sum_{i\in\mathcal{I}_H} \overline{x}_i\right) \\
& =  d_y \left(2\overline{x}_l-\sum_{i\in N} \overline{x}_i \right). \\
\end{align*}
\endgroup
However, $d_y \left(2\overline{x}_l-\sum_{i\in N} \overline{x}_i \right) \geq 0$ because $d_y <0$ and $2\overline{x}_l \leq \sum_{i\in N}\overline{x}_i$ by assumption. Therefore, this case is invalid. \\

The remaining case is $d_y \geq 0$. Given our observation that the ray $d$ contains at least one strictly negative entry, $d_{r_i} < 0$ for at least one $i\in \mathcal{I}_H$. In fact, due to \eqref{eq:ray_property_2}, there can be exactly one $i\in \mathcal{I}_H$ such that $d_{r_i} < 0$. We let this index be $J$, and abbreviate $d_{r_J}$ to be $d^*$. In this case, \eqref{eq:ray_property_3} implies that  
\begingroup
\allowdisplaybreaks
\begin{align*}
0 & >  \left(2\overline{x}_l-\sum_{i\in\mathcal{I}_L} \overline{x}_i\right)d_y+ \sum_{i\in\mathcal{I}_H} \overline{x}_i d_{r_i} \\
& \geq \left(2\overline{x}_l-\sum_{i\in\mathcal{I}_L} \overline{x}_i\right)(-d^*)+ \overline{x}_J d^* + \sum_{i\in\mathcal{I}_H\backslash \{J\}} \overline{x}_i (-d^*) \\
& \quad \quad \text{because $d_y \geq -d^*$ by \eqref{eq:ray_property_1}, and $d_{r_i} \geq -d^*$ for $i\in \mathcal{I}_H\backslash \{J\}$ by \eqref{eq:ray_property_2}, }\\ 
&\geq \left(2\overline{x}_l-\sum_{i\in\mathcal{I}_L} \overline{x}_i\right)(-d^*)+ \overline{x}_h d^* + \sum_{i\in\mathcal{I}_H\backslash \{h\}} \overline{x}_i (-d^*)  \text{  because $\overline{x}_h \geq \overline{x}_J$, }\\ 
& =  -d^* \left(2\overline{x}_l-\sum_{i\in\mathcal{I}_L} \overline{x}_i - \overline{x}_h + \sum_{i\in\mathcal{I}_H\backslash \{h\}} \overline{x}_i\right) \\
& =  -d^* \left(2\overline{x}_l-\sum_{i\in\mathcal{I}_L} \overline{x}_i - 2\overline{x}_h + \sum_{i\in\mathcal{I}_H} \overline{x}_i\right).
\end{align*}
\endgroup
Since $-d^*>0$ by construction and $2\overline{x}_h - \sum_{i\in \mathcal{I}_H} \overline{x}_i \leq 2\overline{x}_l - \sum_{i\in \mathcal{I}_L} \overline{x}_i$ by assumption, 
\[-d^* \left(2\overline{x}_l-\sum_{i\in\mathcal{I}_L} \overline{x}_i - 2\overline{x}_h + \sum_{i\in\mathcal{I}_H} \overline{x}_i\right) \geq 0. \] This again violates property \eqref{eq:ray_property_3}. \\

By contradiction, we have shown that \eqref{eq:h_average_sepa_subproblem_primal} does not contain a feasible, objective-improving ray. Hence we conclude that \eqref{eq:h_average_sepa_subproblem_primal} is feasible with a bounded optimal objective. 
\end{proof}

\begin{corollary}
\label{coro:h_average_sepa_subproblem_feasible}
If $\overline{x}$ falls under category (c4), then problem \eqref{eq:h_average_sepa_subproblem_dual} is feasible. 
\end{corollary}
\begingroup
\allowdisplaybreaks
\begin{subequations}
\label{eq:h_average_sepa_subproblem_dual}
\begin{alignat}{2}
\max \hspace{0.2cm} & 0 && \\
\textrm{s.t.} \quad & \sum_{i\in\mathcal{I}_H} Q(\{l,i\}) = 2\overline{x}_l-\sum_{i\in\mathcal{I}_L} \overline{x}_i, &&  \label{eq:h_average_sepa_subproblem_hsum} \\
& Q(\{l,j\}) + \sum_{i\in \mathcal{I}_H\backslash \{j\}} Q(\{i,j\}) = \overline{x}_j, && \text{ for all $j\in\mathcal{I}_H$,} \label{eq:h_average_sepa_subproblem_x} \\
& Q(\{l,j\}) \geq 0, && \text{ for all $j\in\mathcal{I}_H$,} \label{eq:h_average_sepa_subproblem_nonnegative1} \\
& Q(\{i,j\}) \geq 0, && \text{ for all $i,j\in\mathcal{I}_H$ such that $i<j$.} \label{eq:h_average_sepa_subproblem_nonnegative2}
\end{alignat}
\end{subequations}
\endgroup
\begin{proof}
Problem \eqref{eq:h_average_sepa_subproblem_dual} is the dual linear program of \eqref{eq:h_average_sepa_subproblem_primal}. The variables $Q(\{l,i\})$ for all $i\in\mathcal{I}_H$ correspond to the primal constraints \eqref{eq:h_average_sepa_subproblem_primal_yr}. Here $l$ is included solely as a placeholder to ensure notational consistency with the proof of Proposition \ref{prop:separate_higher_average}. The variables $Q(\{i,j\})$ for all $i,j\in\mathcal{I}_H$ with $i<j$ are the dual variables for constraints \eqref{eq:h_average_sepa_subproblem_primal_rr}. Constraint \eqref{eq:h_average_sepa_subproblem_hsum} corresponds to the primal variable $y$, and constraints \eqref{eq:h_average_sepa_subproblem_x} correspond to $r_i$ for $i\in \mathcal{I}_H$. This corollary follows from Lemma \ref{lemma:h_average_sepa_subproblem_primal_bounded}.
\end{proof}

\begin{proposition}
\label{prop:separate_higher_average} 
If $\overline{x}$ belongs to category (c4),
then 
\[ \overline{\pi}_i = \begin{cases} 0, & i = 0, \\
f(2a_H)/2, & i \in \mathcal{I}_H, \\
f(a_L + a_H) - f(2a_H)/2, & i = l, \\
f(2a_L) - f(a_L + a_H) + f(2a_H)/2, & i\in\mathcal{I}_L\backslash \{l\}, 
\end{cases} \] is an optimal solution to problem \eqref{eq:separation_primal} associated with $\overline{x}$. This optimal solution is the set of coefficients for the higher-SI  \eqref{eq:p2_higher_average}.
\end{proposition}
\begin{proof}
The proposed solution $\overline{\pi}_i$ is feasible, shown by the validity of inequality \eqref{eq:p2_higher_average} for $\mathcal{P}^2_2$. For its optimality, we construct a feasible solution to problem \eqref{eq:separation_dual} with the same objective value. Consider 
\[ \overline{Q}(S) = \begin{cases}
1-\sum_{i\in N} \overline{x}_i/2, & S = \emptyset, \\
\overline{x}_i, & S = \{l,i\}, i\in \mathcal{I}_L\backslash \{l\}, \\
\hat{Q}(S), & S  = \{l,i\}, i\in \mathcal{I}_H, \\
\hat{Q}(S), & S  = \{i,j\}, i,j\in \mathcal{I}_H, i<j, \\
0, & \text{for all other $S\subseteq N$ with $|S|\leq 2$,} 
\end{cases} \]
where $\hat{Q}(\cdot)$ is any feasible solution to problem \eqref{eq:h_average_sepa_subproblem_dual}. Such $\hat{Q}(\cdot)$ exists, as a result of Corollary \ref{coro:h_average_sepa_subproblem_feasible}. \\

We now show that $\overline{Q}(\cdot)$ is feasible to problem \eqref{eq:separation_dual} associated with $\overline{x}$. Since $\overline{x}\in[0,1]^n$ and $\sum_{i\in N} \overline{x}_i \leq 2$, $\overline{Q}(\emptyset)\geq 0$ and $\overline{Q}(\{l,i\}) \geq 0$ for all $i\in \mathcal{I}_L\backslash \{l\}$. By constraints \eqref{eq:h_average_sepa_subproblem_nonnegative1} and \eqref{eq:h_average_sepa_subproblem_nonnegative2}, $\hat{Q}(\cdot) \geq 0$.  Thus \eqref{eq:separation_dual_nonnegative} is satisfied by the proposed solution. \\

We observe that 
\[ \sum_{i\in \mathcal{I}_H} \hat{Q}(\{l,i\}) = 2\overline{x}_l - \sum_{i\in \mathcal{I}_L}\overline{x}_i, \] due to constraint \eqref{eq:h_average_sepa_subproblem_hsum}. Furthermore, 
\begingroup
\allowdisplaybreaks
\begin{align*}
\sum_{i,j\in \mathcal{I}_H, i< j} \overline{Q}(\{i,j\})
& = \sum_{j\in \mathcal{I}_H}\sum_{i\in \mathcal{I}_H\backslash\{j\}} \hat{Q}(\{i,j\})/2 \\
& =  \sum_{j\in \mathcal{I}_H} \hat{Q}(\{l,j\})/2 + \sum_{j\in \mathcal{I}_H}\sum_{i\in \mathcal{I}_H\backslash\{j\}} \hat{Q}(\{i,j\})/2 - \sum_{i\in \mathcal{I}_H} \hat{Q}(\{l,i\})/2 \\
& = \sum_{j\in \mathcal{I}_H} \left[ \hat{Q}(\{l,j\}) + \sum_{i\in \mathcal{I}_H\backslash\{j\}} \hat{Q}(\{i,j\}) \right]/2 - \sum_{i\in \mathcal{I}_H} \hat{Q}(\{l,i\})/2\\ 
& = \sum_{j\in \mathcal{I}_H} \overline{x}_j/2 - \left(\overline{x}_l - \sum_{i\in \mathcal{I}_L}\overline{x}_i/2 \right). 
\end{align*}
\endgroup The last equality follows from constraints \eqref{eq:h_average_sepa_subproblem_hsum} and \eqref{eq:h_average_sepa_subproblem_x}. With these observations, we deduce that

\begingroup
\allowdisplaybreaks
\begin{align*}
\sum_{S:|S|\leq 2} \overline{Q}(S) & =  \overline{Q}(\emptyset) + \sum_{i\in \mathcal{I}_L\backslash\{l\}} \overline{Q}(\{l,i\}) + \sum_{i\in \mathcal{I}_H} \overline{Q}(\{l,i\}) + \sum_{i,j\in \mathcal{I}_H, i< j} \overline{Q}(\{i,j\}) + 0 \\
& = 1-\sum_{i\in N} \overline{x}_i/2 + \sum_{i\in \mathcal{I}_L\backslash\{l\}} \overline{x}_i + 2\hat{x}_l - \sum_{i\in \mathcal{I}_L}\hat{x}_i +  \sum_{j\in \mathcal{I}_H} \overline{x}_j/2 - \left(\overline{x}_l - \sum_{i\in \mathcal{I}_L}\overline{x}_i/2 \right) \\
& = 1-\sum_{i\in N} \overline{x}_i/2 + \sum_{i\in \mathcal{I}_L} \overline{x}_i  - \sum_{i\in \mathcal{I}_L}\hat{x}_i +  \sum_{j\in \mathcal{I}_H} \overline{x}_j/2 + \sum_{i\in \mathcal{I}_L}\overline{x}_i/2  \\
& = 1-\sum_{i\in N} \overline{x}_i/2 + \sum_{i\in N}\overline{x}_i/2  \\
& = 1.
\end{align*}
\endgroup Therefore, constraint \eqref{eq:separation_dual_one} is also satisfied. \\

To check for \eqref{eq:separation_dual_x}, we note that for any $i\in \mathcal{I}_L\backslash \{l\}$, 
\[ \sum_{S:|S|\leq 2, S\ni i} \overline{Q}(S) = \overline{Q}(\{l,i\}) = \overline{x}_i, \] 
and 
\begingroup
\allowdisplaybreaks
\begin{align*}
\sum_{S:|S|\leq 2, S\ni l} \overline{Q}(S) & = \sum_{i\in \mathcal{I}_L\backslash \{l\}} \overline{Q}(\{l,i\}) + \sum_{i\in \mathcal{I}_H} \hat{Q}(\{l,i\})\\
&  =  \sum_{i\in \mathcal{I}_L\backslash \{l\}} \overline{x}_i + 2\overline{x}_l - \sum_{i\in \mathcal{I}_L} \overline{x}_i \quad \text{by \eqref{eq:h_average_sepa_subproblem_hsum},} \\
& = \overline{x}_l. 
\end{align*}
\endgroup
For any $i\in \mathcal{I}_H$, 
\[ \sum_{S:|S|\leq 2, S\ni i} \overline{Q}(S) = Q(\{l,i\}) + \sum_{j\in \mathcal{I}_H\backslash \{i\}} Q(\{i,j\}) = \overline{x}_i, \] which immediately follows from \eqref{eq:h_average_sepa_subproblem_x}.  \\

Now that we have shown feasibility of $\overline{Q}(\cdot)$ in problem \eqref{eq:separation_dual}, the remaining task is to examine its corresponding objective. 
\begingroup
\allowdisplaybreaks
\begin{align*}
& \sum_{S\subseteq N: |S|\leq 2} \overline{Q}(S) f\left(\sum_{i\in S}a_i\right) \\
& = \overline{Q}(\emptyset)\cdot f(0) + \sum_{i\in \mathcal{I}_L\backslash\{l\}} f(2a_L)\overline{Q}(\{l,i\}) + \sum_{i\in \mathcal{I}_H} f(a_L+a_H)\overline{Q}(\{l,i\}) + \sum_{i,j\in \mathcal{I}_H, i< j} f(2a_H)\overline{Q}(\{i,j\}) \\
& =  f(2a_L)\sum_{i\in \mathcal{I}_L\backslash\{l\}} \overline{x}_i  + f(a_L+a_H)\left(2\overline{x}_l - \sum_{i\in\mathcal{I}_L} \overline{x}_i\right) + f(2a_H)\left[\sum_{j\in \mathcal{I}_H} \overline{x}_j/2 - \left(\overline{x}_l - \sum_{i\in \mathcal{I}_L}\overline{x}_i/2 \right) \right]  \\
& =  f(2a_L)\sum_{i\in \mathcal{I}_L\backslash\{l\}} \overline{x}_i  + f(a_L+a_H)\left(\overline{x}_l - \sum_{i\in \mathcal{I}_L\backslash \{l\}}\overline{x}_i \right) + f(2a_H)\left[\sum_{i\in \mathcal{I}_H}\overline{x}_i/2 + \sum_{i\in \mathcal{I}_L\backslash\{l\}}\overline{x}_i/2 + \overline{x}_l/2 - \overline{x}_l\right]\\
& =  f(2a_L)\sum_{i\in \mathcal{I}_L\backslash\{l\}} \overline{x}_i  + f(a_L+a_H)\left(\overline{x}_l - \sum_{i\in \mathcal{I}_L\backslash \{l\}}\overline{x}_i \right) + \frac{f(2a_H)}{2}\left[\sum_{i\in \mathcal{I}_H}\overline{x}_i + \sum_{i\in \mathcal{I}_L\backslash\{l\}}\overline{x}_i - \overline{x}_l\right]\\
& = \left[f(a_L + a_H) - \frac{f(2a_H)}{2} \right]\overline{x}_l + \left[f(2a_L) - f(a_L + a_H) + \frac{f(2a_H)}{2} \right] \sum_{i\in \mathcal{I}_L\backslash\{l\}} \overline{x}_i + \frac{f(2a_H)}{2} \sum_{i \in \mathcal{I}_H}\overline{x}_i \\
& = \overline{\pi}_0 + \sum_{i\in N} \overline{\pi}_i\overline{x}_i. 
\end{align*}
\endgroup
By strong duality, $\overline{\pi}$ is optimal in problem \eqref{eq:separation_primal}. 
\end{proof}

\begin{proposition}
\label{prop:separate_lower_average} 
If $\overline{x}$ belongs to category (c5), 
then 
\[ \overline{\pi}_i = \begin{cases} 0, & i = 0, \\
f(2a_L)/2, & i \in \mathcal{I}_L, \\
f(a_L + a_H) - f(2a_L)/2, & i = h, \\
f(2a_H) - f(a_L + a_H) + f(2a_L)/2, & i\in\mathcal{I}_H\backslash \{h\}, 
\end{cases} \] is an optimal solution to problem \eqref{eq:separation_primal} associated with $\overline{x}$. This optimal solution is exactly the coefficients of the lower-SI \eqref{eq:p2_lower_average}.
\end{proposition}
\begin{proof}
We can prove the counterparts of Lemma \ref{lemma:h_average_sepa_subproblem_primal_bounded} and Corollary \ref{coro:h_average_sepa_subproblem_feasible} for this case, by replacing $H$ by $L$, $h$ by $l$ in the notation. Then by switching notation again in the proof of Proposition \ref{prop:separate_higher_average}, we establish this proposition. 
\end{proof}

By now we have found the optimal solutions to problem \eqref{eq:separation_primal} associated with all possible $(\overline{w},\overline{x})\in\mathbb{R}\times[0,1]^n$ such that $\sum_{i=1}^n \overline{x}_i \leq 2$. These optimal solutions, or extreme rays in the polar of $\conv{\mathcal{P}^2_2}$, match the proposed inequalities, namely the super-average inequality \eqref{eq:p2_super_average}, the lifted-EPIs \eqref{eq:p2_EPI_form1}, \eqref{eq:p2_EPI_form2}, the lower-SI \eqref{eq:p2_lower_average}, and the higher-SI \eqref{eq:p2_higher_average}. Problem \eqref{eq:separation_primal} can also be thought of as the separation problem for any $(\overline{w},\overline{x})$, whose optimal solution is the most violated inequality at this point. We proceed to draw conclusions on the complete linear description of $\conv{\mathcal{P}^2_2}$ in Section \ref{sect:p22_final_conv}.

\subsection{Convex hull description of $\mathcal{P}^2_2$}
\label{sect:p22_final_conv}
In this subsection, we formalize the full linear characterization of $\conv{\mathcal{P}^2_2}$ in Theorem \ref{thm:conv}. After that, we make a remark on the separation of the proposed non-trivial inequalities. Depending on the sizes of $\mathcal{I}_L$ and $\mathcal{I}_H$, some of the five subsets for $(\overline{w}, \overline{x})\in\mathbb{R}\times [0,1]^n$ with $\sum_{i=1}^n \overline{x}_i \leq 2$ could be empty. Therefore, we may not always need the full set of proposed inequalities to define $\conv{\mathcal{P}^2_2}$. We then make a remark to specify these cases.

\begin{theorem}
\label{thm:conv}
Suppose Assumption \ref{assumption} holds for $i_0=0$. We let $\mathcal{S}$ be the set of $(w,x)\in\mathbb{R}^{n+1}$ constructed by the super-average inequality \eqref{eq:p2_super_average}, the lifted-EPIs \eqref{eq:p2_EPI_form1}, \eqref{eq:p2_EPI_form2}, the lower-SI \eqref{eq:p2_lower_average}, and the higher-SI \eqref{eq:p2_higher_average}, together with the trivial inequalities \eqref{eq:p2_trivial} and cardinality constraint \eqref{eq:p2_cardinality}. Then $\mathcal{S} = \conv{\mathcal{P}^2_2}$. 
\end{theorem}
\begin{proof}
Propositions \ref{prop:separate_EPI-L}, \ref{prop:separate_EPI-H}, \ref{prop:separate_super_average}, \ref{prop:separate_higher_average} and \ref{prop:separate_lower_average} prove that the set of all the non-trivial inequalities stated above contains all the facets of $\conv{\mathcal{P}^2_2}$. It follows that $\mathcal{S} = \conv{\mathcal{P}^2_2}$.
\end{proof}

Recall that any $\overline{x}\in [0,1]^n$ falls into one of the following categories: 
\begin{enumerate}
\item[(c1)] $2\overline{x}_l > \sum_{i\in N} \overline{x}_i$;
\item[(c2)] $2\overline{x}_h > \sum_{i\in N} \overline{x}_i$; 
\item[(c3)] $2\overline{x}_l < \sum_{i\in \mathcal{I}_L} \overline{x}_i$ and $2\overline{x}_h < \sum_{i\in \mathcal{I}_H} \overline{x}_i$; 
\item[(c4)] $\sum_{i\in \mathcal{I}_L} \overline{x}_i \leq 2\overline{x}_l \leq \sum_{i\in N} \overline{x}_i$, $2\overline{x}_h \leq \sum_{i\in N} \overline{x}_i$, and $2\overline{x}_h - \sum_{i\in \mathcal{I}_H} \overline{x}_i \leq 2\overline{x}_l - \sum_{i\in \mathcal{I}_L} \overline{x}_i$;
\item[(c5)] $\sum_{i\in \mathcal{I}_H} \overline{x}_i \leq 2\overline{x}_h \leq \sum_{i\in N} \overline{x}_i$, $2\overline{x}_l \leq \sum_{i\in N} \overline{x}_i$, and $2\overline{x}_l - \sum_{i\in \mathcal{I}_L} \overline{x}_i \leq 2\overline{x}_h - \sum_{i\in \mathcal{I}_H} \overline{x}_i$.
\end{enumerate} 

\begin{remark} Based on the discussion in Section \ref{sect:p22_polar}, when any $(\overline{w}, \overline{x})\in\mathbb{R}\times [0,1]^n$ falls under category (c1), the lifted-EPI \eqref{eq:p2_EPI_form1} is the most violated inequality at $(\overline{w},\overline{x})$ if a violation occurs. In particular, the most violated lifted-EPI has $l\in N$ as the first item in the permutation of $N$. When any given $(\overline{w}, \overline{x})$ falls in category (c2), the lifted-EPI \eqref{eq:p2_EPI_form2} is the most violated inequality at this point, with permutation $\delta$ such that $\delta_1 = h$. For any $(\overline{w}, \overline{x})\notin\conv{\mathcal{P}^2_2}$ that satisfies (c3), the super-average inequality \eqref{eq:p2_super_average} should have the highest violation among all the valid inequalities. Lastly, if $(\overline{w}, \overline{x})\notin\conv{\mathcal{P}^2_2}$ satisfies (c4) or (c5), then the most violated cut is the lower-SI \eqref{eq:p2_lower_average}, or the higher-SI \eqref{eq:p2_higher_average}, respectively. More specifically, the most violated lower-SI corresponds to the permutation $\delta$ such that $h$ is the first higher-weighted item. Similarly, the most violated higher-SI is obtained with permutation $\delta$ in which $l$ is the first lower-weighted item. 
\end{remark}

\begin{remark}
We note that any $(\overline{w}, \overline{x})\in\mathbb{R}\times [0,1]^n$ with $\sum_{i=1}^n \overline{x}_i \leq 2$ can only belong to category (c3) when $|\mathcal{I}_L|\geq 3$ and $|\mathcal{I}_H|\geq 3$; otherwise, either $2\overline{x}_l \geq \sum_{i\in \mathcal{I}_L} \overline{x}_i$ or $2\overline{x}_h \geq \sum_{i\in \mathcal{I}_H} \overline{x}_i$ must be true. Thus when either $|\mathcal{I}_L|\leq 2$, or $|\mathcal{I}_H|\leq 2$, the super-average inequality \eqref{eq:p2_super_average} is not needed in the full linear description of $\conv{\mathcal{P}^2_2}$. 
\end{remark}

\begin{remark}
Suppose $|\mathcal{I}_L| = 1$. We further assume that $(\overline{w}, \overline{x})$ satisfies $2\overline{x}_l \leq \sum_{i\in N} \overline{x}_i$, and $2\overline{x}_h \leq \sum_{i\in N} \overline{x}_i$. We observe that $\sum_{i\in \mathcal{I}_L} \overline{x}_i = \overline{x}_l \leq 2\overline{x}_l$. Also, $\sum_{i\in N}\overline{x}_i = \sum_{i\in\mathcal{I}_H}\overline{x}_i + \overline{x}_l$. Thus $2\overline{x}_h - \sum_{i\in\mathcal{I}_H}\overline{x}_i = 2\overline{x}_h - \sum_{i\in N}\overline{x}_i + \overline{x}_l \leq 0 + \overline{x}_l \leq 2\overline{x}_l - \sum_{i\in\mathcal{I}_L}\overline{x}_i$. These observations imply that the category (c5) is empty. Therefore, when $|\mathcal{I}_L| = 1$, the lower-SI \eqref{eq:p2_lower_average} is not necessary in the linear description of $\conv{\mathcal{P}^2_2}$. Similarly, when $|\mathcal{I}_H| = 1$, category (c4) is empty, and the higher-SI \eqref{eq:p2_higher_average} can be omitted from the linear description of $\conv{\mathcal{P}^2_2}$ while not affecting its completeness. 
\end{remark}

\section{Extensions}
\label{sect:ext}
The proposed inequalities for $\conv{\mathcal{P}^2_k}$ can be applied to problem \eqref{prob:cc_concave_sub_min} with more than two distinct weight values. Let $\mathcal{A}$ be the set of distinct weight values. We define $a_{\text{min}}=\min_{a\in \mathcal A}\{a\}$ and $a_{\text{max}}=\max_{a\in \mathcal A}\{a\} $. With any $a_H\in\mathcal{A}$ such that $a_{\text{min}} < a_H \leq a_{\text{max}}$, we construct a new weight vector $\hat{a}$ such that $\hat{a}_i = a_{\text{min}}$ if $a_i < a_H$, and $\hat{a}_i = a_H$ otherwise. This new weight vector contains two distinct weights $a_{\text{min}}$ and $a_H$. 

\begin{proposition}
\label{prop:multi1}
Let $w\geq c^\top x$ denote any valid inequality (e.g., lifted-EPI, lower-SI or higher-SI) for $\conv{\mathcal{P}^2_k}$ with respect to $\hat{a}$. If $f$ is monotone increasing, then $w\geq c^\top x$ is valid for $\conv{\mathcal{P}^m_k}$ that arises from the original multi-weighted problem. 
\end{proposition} 
\begin{proof}
For any $\overline{x}\in\{0,1\}^n$ with $\sum_{i=1}^n \overline{x}_i \leq k$, 
\[c^\top \overline{x}  \leq f(\hat{a}^\top \overline{x}) \leq f(a^\top \overline{x}), \] by validity of $w\geq c^\top x$ for the cardinality-constrained epigraph of $f(\hat{a}^\top x)$ and monotonicity of $f$. 
\end{proof}

It follows from Proposition \ref{prop:multi1} that valid inequalities can be derived similarly when $f$ is monotone decreasing. 

We next introduce another way to generate valid inequalities for $\conv{\mathcal{P}^m_k}$ when $m\geq 3$. Suppose this set is associated with a multi-weighted vector $a\in\mathbb{R}_+^n$ and a normalized concave submodular function $f(a^\top x)$. For any pair of distinct weights in $a$, say $\alpha_1$ and $\alpha_2$, we let $S = \{i\in [n]: a_i \in \{\alpha_1, \alpha_2\}\}$. Consider the case where $|S| \geq k$. Without loss of generality, we assume that the labeling of $[n]$ satisfies $S = [|S|]$. We then extend any valid inequality $w\geq \sum_{i\in S} c_ix_i$ for $\conv{\mathcal{P}^2_k(S)}$ (see \eqref{eq:EPI_base_set}) to the multi-weighted setting. Specifically, for any $i\in[n]\backslash S$, we define $\mathcal{T}_i := \max\{\sum_{j\in T}a_j : T\subseteq [i-1], |T| = k-1\}$. 
\begin{proposition}
\label{prop:multi2}
The inequality $w\geq \sum_{i\in S} c_ix_i + \sum_{i = |S|+1}^n [f(\mathcal{T}_i+a_i) - f(\mathcal{T}_i)]x_i$ is valid for $\conv{\mathcal{P}^m_k}$. 
\end{proposition} 
This proposition generalizes Proposition 11 in \citep{yu2017polyhedral}, which restricts $w\geq \sum_{i\in S} c_ix_i$ to be an EPI and derives ALIs. We omit its proof, because it follows similar arguments.  Here, the inequality $w\geq \sum_{i\in S} c_ix_i$ can be any lifted-EPI, lower-SI or higher-SI for $\conv{\mathcal{P}^2_k(S)}$. When it is a lifted-EPI, the resulting inequality is at least as strong as the corresponding ALI. This observation immediately follows from Corollary \ref{coro:lifted_vs_epi}. The inequalities described above can be used in a branch-and-cut framework when solving the original multi-weighted minimization problems. For certain multi-weighted problem \eqref{prob:cc_concave_sub_min}, our proposed inequalities defined for the subspace involving a pair of distinct weights are valid and even facet-defining for $\conv{\mathcal{P}^m_k}$, as demonstrated in the example below.

\exclude{
\begin{proof} \emph{(same proof technique as in Proposition 11 of \citep{yu2017polyhedral})} 
Consider the valid inequality for $\conv{\mathcal{P}^m_k}$ obtained via lifting $w\geq \sum_{i\in S} c_ix_i$ with respect to the natural order of $[n]$. The lifting coefficient of $j\in [n]\backslash S$ is 
\begingroup
\allowdisplaybreaks
\begin{equation*}
\begin{aligned}
\kappa_j := \min \hspace{0.2cm} & w - \sum_{i\in S} c_i x_i - \sum_{i = |S|+1}^{j-1} \kappa_i x_i&& \\
\textrm{s.t.} \quad & w\geq f\left(a_j + \sum_{i = 1}^{j-1} a_i x_i\right), &&\\
& \sum_{i=1}^{j-1} x_i \leq k-1, && \\
& x \in \{0,1\}^{j-1}. && 
\end{aligned}
\end{equation*}
\endgroup Suppose the optimal solution to this lifting problem is $x^*$, with support $T^*$. By definition, $\mathcal{T}_j\geq \sum_{i\in T^*} a_i$, and due to submodularity of $f(a^\top x)$, $f(\mathcal{T}_j+a_j) - f(\mathcal{T}_j) \leq f(\sum_{i\in T^*} a_i+a_j) - f(\sum_{i\in T^*} a_i)$. We know that $f(\sum_{i\in T^*} a_i) \geq \sum_{i\in S\cap T^*} c_i x_i - \sum_{i \in T^*\backslash S} \kappa_i x_i$. Therefore, $f(\sum_{i\in T^*} a_i+a_j) - f(\sum_{i\in T^*} a_i) \leq \kappa_j$. We conclude that $f(\mathcal{T}_j+a_j) - f(\mathcal{T}_j) \leq \kappa_j$ for all $j\in [n]\backslash S$, and the stated inequality is valid. 
\end{proof}
}

\begin{example}
Let $f(a^\top x) = 64-(a^\top x-8)^2$, $k=2$ and $a = [4,4,6,6,8]$. The inequality \[w \geq 32x_1+32x_2+28x_3+20x_4\] is a lower-SI for the convex hull of $\{(w,x)\in\mathbb{R}\times \{0,1\}^4: w\geq f(4x_1+4x_2+6x_3+6x_4), \mathbf{1}^\top x \leq 2\}$. This inequality is facet-defining for the original $\conv{\mathcal{P}^m_2}$.
\end{example}

We can also obtain strong formulations for mixed-binary conic optimization with our proposed inequalities. Consider the set \[S(F,\mathcal{K}) := \{(x, y)\in\mathbb{B}^n\times \mathbb{R}^m : \exists w\in \mathbb{R}_+ \text{ s.t. } w\geq F(x), {\mathbf{1}^\top x \leq k, } Ay+Bw\in\mathcal{K}\},\] where $\mathcal{K}$ is a convex cone that contains the origin, $F:\mathbb{B}^n\rightarrow \mathbb{R}_+$ is a nonnegative function, and $A,B$ are matrices of proper dimensions. A special case of this set is studied by \citet{atamturk2020submodularity}, in which the set captures a single second-order conic constraint and $F$ is the composition of a square root function and a nonnegative affine function. This mixed-binary set arises in chance-constrained programs and mean-risk minimization. The authors provide its convex hull description, which involves the convex hull of the epigraph of $F$. \citet{kilincc2020conic} extend this result to the general set $S(F,\mathcal{K})$. Based on their work, our proposed inequalities are strong valid inequalities for the convex hull of $S(F,\mathcal{K})$ under a cardinality constraint on $x$, when $F$ is any nonnegative concave function composed with a nonnegative affine function. Recall that $\mathcal{S}$ is the set constructed with our proposed inequalities and is equivalent to $\conv{\mathcal{P}^2_2}$ (see Theorem \ref{thm:conv}). When the affine function contains two weights and the cardinality bound is two, $\conv{S(F,\mathcal{K})} = \{(x, y)\in[0,1]^n\times \mathbb{R}^m : \exists w\in \mathbb{R}_+ \text{ s.t. } (w,x)\in\mathcal{S}, {\mathbf{1}^\top x \leq k, } Ay+Bw\in\mathcal{K}\}$.

{
\section{Computational Study}
\label{sect:comp}
In this section, we test the effectiveness of our proposed inequalities in a branch-and-cut algorithm. We consider instances of cardinality-constrained mean-risk minimization with correlated random variables \cite{atamturk2019lifted,atamturk2020submodularity}:  
\begin{equation}
\label{eq:mean_risk}
\min_{x \in \{0,1\}^n} \left\{-\mu^\top x + \Omega\sqrt{(x^\top Qx)} : \sum_{i=1}^nx_i \leq k\right\}. 
\end{equation}

Here, $Q$ is a positive semidefinite matrix, $\Omega$ is a constant parameter, and $k\in\mathbb{Z}_+$ is the cardinality upper bound. Problem \eqref{eq:mean_risk} can be interpreted as minimizing a stochastic objective over a discrete feasible set. Suppose that the losses on all the investments $i\in N$, denoted by $\tilde p$, are normal random variables with mean $\mu$ and covariance $Q$. Let $\Phi$ be the standard normal cumulative distribution function. We set $\Omega$ to be $\Phi^{-1}(\beta)$ where $0.5 < \beta <1$. Then problem \eqref{eq:mean_risk} is equivalent to the value-at-risk minimization problem $\min_{x \in \{0,1\}^n}\{z : \mathbb{P}\left( \tilde p^\top x \leq r \right) \geq \beta,  \sum_{i=1}^nx_i \leq k\}$ \cite{atamturk2008polymatroids,birge2011introduction,atamturk2019lifted,atamturk2020submodularity}. We denote a diagonal matrix with main diagonal in the vector form, $\nu$, by $\text{diag}(\nu)$. The covariance matrix $Q$ is commonly rewritten as the sum of $Q-\text{diag}(a)$ and $\text{diag}(a)$, such that $a\in\mathbb{R}_+^n$ and $Q-\text{diag}(a) \succeq 0$. Given that $x\in\{0,1\}^n$, the separable quadratic term $x^\top \text{diag}(a) x=a^\top x$. Therefore, problem \eqref{eq:mean_risk} has an equivalent formulation (SOCP): 
\begin{equation}
\label{eq:SOCP}
\min_{(w,y,z,x) \in \mathbb{R}_+^3\times \{0,1\}^n} \left\{-\mu^\top x + \Omega z : w \geq \sqrt{\sum_{i\in N} a_i x_i}, \sum_{i=1}^nx_i \leq k, y\geq \sqrt{x^\top (Q-\text{diag}(a)) x}, z^2\geq w^2 + y^2  \right\}. 
\end{equation}

When $a$ consists of two distinct weights, the proposed inequalities are directly applicable; this case will be discussed in Section \ref{sect:comp_two_weights}. When $a\in\mathbb{R}_+^n$ is a general vector, we may write $a$ as $a^{\text{two}}+a^{\text{res}}$ such that $a^{\text{two}}, a^{\text{res}} \in\mathbb{R}_+^n$, and $a^{\text{two}}$ contains two distinct weights. Problem \eqref{eq:SOCP} may be reformulated as 
\begin{align*}
\min_{(v,w,y,z,x) \in \mathbb{R}_+^4\times \{0,1\}^n} \{-\mu^\top x + \Omega z : & \: w \geq \sqrt{\sum_{i\in N} a^{\text{two}}_i x_i}, \sum_{i=1}^nx_i \leq k, \\
& \: v \geq \sqrt{\sum_{i\in N} a^{\text{res}}_i x_i}, y\geq \sqrt{x^\top (Q-\text{diag}(a)) x}, z^2\geq v^2 + w^2 + y^2  \}. 
\end{align*}
This case will be explored in Section \ref{sect:comp_multi_weights}. To maintain generality of the test instances of problem \eqref{eq:mean_risk}, we do not impose any assumption, such as Assumption \ref{assumption}, on the two weights $a_L$ and $a_H$ in addition to non-negativity. Therefore, we only incorporate lifted-EPIs (LEPIs) and lower-SIs (LSIs) in our branch-and-cut algorithm. We add one valid inequality after exploring every ten branch-and-bound nodes in the following way. At a fractional solution $(\overline{w},\overline{x})$, we generate an LEPI with respect to $\delta = (\delta_1, \delta_2,\dots, \delta_n)$, such that $\overline{x}_{\delta_1} \geq \overline{x}_{\delta_2} \geq \dots \geq \overline{x}_{\delta_n}$. Let $\overline{x}^H$ be the sub-vector of $\overline{x}$ that corresponds to all the higher-weighted items. With the descending order of $\overline{x}^H$ and $i_0 = k-1$, we construct the corresponding LSI. If the violation of LEPI at $(\overline{w},\overline{x})$ is higher than that of LSI, then the LEPI is added to update the relaxation problem. Otherwise, the LSI is added. We refer to this branch-and-cut algorithm as BC-LEPI-LSI. \\

To evaluate the effectiveness of the proposed inequalities, we test our method BC-LEPI-LSI against another branch-and-cut algorithm that incorporates the ALIs \cite{yu2017polyhedral} (see Corollary \ref{coro:lifted_vs_epi}). We add one ALI after exploring every ten branch-and-bound nodes, and such an ALI is constructed according to the descending order of $\overline{x}$. Moreover, we compare the computational performance of BC-LEPI-LSI against directly solving the SOCP using a mixed-integer SOCP solver. Later we refer to this method simply as SOCP.  \\

The experiments are executed on one thread of a Linux server with Intel Haswell E5-2680 processor at 2.5GHz and 128GB of RAM. All the solution methods are implemented in Python 3.6 and Gurobi Optimizer 9.5.1. The internal cut parameters are in the default setting. Multithreading, heuristics and concurrent MIP solver are disabled. The MIP optimality gap is at the default level of 0.01\%, and the time limit for each instance is set to one hour.

\subsection{$a$ with two weights} 
\label{sect:comp_two_weights}
Inspired by \cite{atamturk2020submodularity}, we generate the test instances in the following way. The covariance matrix $Q = Q_0 + \text{diag}(a)$, where $Q_0 = ZGG^\top Z^\top$ following a factor model. In particular, $G\in\mathbb{R}^{r\times r}$ with $G_{ij} \sim U[-1,1]$, and $Z\in\mathbb{R}^{n\times r}$ such that $Z_{ij} \sim U[0,1]$ with probability 0.2 and $Z_{ij} = 0$ otherwise. We compute $\overline{q} = \sum_{i=1}^n Q_{0ii}/n$, and generate $\hat a_i \sim U[0.2\overline{q}, \overline{q}]$. We then set the two weights $a_L = \min(\hat a)$ and $a_H = \text{median}(\hat a)$. The diagonal vector $a$ is constructed by letting $a_i = a_L$ when $\hat a_i < a_H$, and $a_i = a_H$ otherwise. We further generate $\mu_i \sim U[0.7\sqrt{Q_{ii}}, \sqrt{Q_{ii}}]$. In our experiments, we let $n=200$, $r=40$, $k\in \{5, 10, 15\}$, and $\Omega = \Phi^{-1}(\beta)$, with $\beta \in\{0.95, 0.975, 0.99\}$. \\

Table \ref{table:2_weight_result} summarizes the computational performance of BC-LEPI-LSI, BC-ALI and SOCP on problem \eqref{eq:mean_risk} in which diag($a$) contains two distinct weights. The first two columns report the risk tolerance parameter $\beta$ and the cardinality upper bound $k$. The fourth column reports the average running time in seconds. The next column lists the average end gaps, computed by (UB-LB)/UB$\times 100\%$ in which UB and LB are the best upper- and lower-bounds attained at the time limit. The average end gaps are computed across all the trials, including the instances solved to optimality. The sixth and the seventh columns present the average numbers of branch-and-bound nodes visited and the average numbers of cuts added. The statistics are averaged across five trials. Each superscript $^{i}$ means that out of the five trials, $i$ instances are solved within the time limit of one hour, and the remaining $5-i$ instances exceed the time limit. For BC-LEPI-LSI, the average number of total cuts is represented as $\text{m}^{\text{LEPI}} + \text{m}^{\text{LSI}} = \text{m}$ in each test case, where $\text{m}^{\text{LEPI}}$ is the average number of LEPIs added across five trials, and $\text{m}^{\text{LSI}}$ is that of LSIs.  \\

\begin{table}[h!]
\begin{center}
\small
{\caption{Computational performance of BC-LEPI-LSI, BC-ALI and SOCP on problem \eqref{eq:mean_risk} in which diag($a$) decomposed from the covariance matrix $Q$ contains two distinct weights. }}
\begin{tabular}{|c|c|c|c|c|c|c|}
\hline
$\beta$ & $k$ & method & time (s) & end gap & \# nodes & \# cuts \\
\hline
\multirow{9}{*}{0.95}   & \multirow{3}{*}{5}   & BC-LEPI-LSI & $42.7^5$	&	0.0\%	&	2112.0	&	119.4+72.4=191.8\\
                                   &                               & BC-ALI          & $219.9^5$	&	0.0\%	&	7376.4	&	731.2\\
                                   &                               & SOCP            & $1741.6^5$	&	0.0\%	&	92132.4	&	N/A\\
                                    \cline{2-7}
                                   & \multirow{3}{*}{10} & BC-LEPI-LSI & $37.4^5$	&	0.0\%	&	1734.6	&	147.0+14.0=161.0\\
                                   &                               & BC-ALI          & $99.9^5$	&	0.0\%	&	3791.6	&	376.8\\
                                   &                               & SOCP            & $2924.0^2$	&	2.8\%	&	73682.2	&	N/A\\    
                                    \cline{2-7}
                                   & \multirow{3}{*}{15} & BC-LEPI-LSI & $56.0^5$	&	0.0\%	&	2136.4	&	196.6+3.0=199.6\\
                                   &                               & BC-ALI          & $67.4^5$	&	0.0\%	&	2505.8	&	247.8\\
                                   &                               & SOCP            & $2664.1^2$	&	1.6\%	&	59483.4	&	N/A\\  
\hline    
\multirow{9}{*}{0.975} & \multirow{3}{*}{5}   & BC-LEPI-LSI & $109.6^5$	&	0.0\%	&	4764.8	&	280.2+169.6=449.8\\
                                   &                               & BC-ALI          & $608.0^5$	&	0.0\%	&	20233.6	&	2013.4\\
                                   &                               & SOCP            & $2773.5^3$	&	13.8\%	&	191009.0	&	N/A\\  
                                    \cline{2-7}
                                   & \multirow{3}{*}{10} & BC-LEPI-LSI & $258.5^5$	&	0.0\%	&	10196.6	&	773.4+231.2=1004.6\\
                                   &                               & BC-ALI          & $1337.4^4$	&	1.4\%	&	32009.8	&	3196.4\\
                                   &                               & SOCP            & $3316.4^1$	&	13.9\%	&	69974.2	&	N/A\\     
                                    \cline{2-7}
                                   & \multirow{3}{*}{15} & BC-LEPI-LSI & $80.8^5$	&	0.0\%	&	3219.6	&	312.0+6.0=318.0\\
                                   &                               & BC-ALI          & $250.3^5$	&	0.0\%	&	6880.4	&	685.4\\
                                   &                               & SOCP            & --$^0$	&	5.5\%	&	74772.0	&	N/A\\  
\hline  
\multirow{9}{*}{0.99}   & \multirow{3}{*}{5}   & BC-LEPI-LSI & $125.1^5$	&	0.0\%	&	5997.4	&	242.6+343.4=586.0\\
                                   &                               & BC-ALI          & $1284.5^5$	&	0.0\%	&	36510.2	&	3648.4\\
                                   &                               & SOCP            & $3476.2^1$	&	76.0\%	&	160858.0	&	N/A\\  
                                    \cline{2-7}
                                   & \multirow{3}{*}{10} & BC-LEPI-LSI & $339.4^5$	&	0.0\%	&	11414.8	&	633.8+496.8=1130.6\\
                                   &                               & BC-ALI          & $2413.9^3$	&	4.3\%	&	73554.6	&	7351.0\\
                                   &                               & SOCP            & --$^0$	&	26.8\%	&	72702.6	&	N/A\\     
                                    \cline{2-7}
                                   & \multirow{3}{*}{15} & BC-LEPI-LSI & $96.3^5$	&	0.0\%	&	3465.4	&	268.2+72.8=341.0\\
                                   &                               & BC-ALI          & $1440.1^5$	&	0.0\%	&	34640.0	&	3462.4\\
                                   &                               & SOCP            & --$^0$	&	10.6\%	&	57566.0	&	N/A\\  
\hline                        
\end{tabular}
\label{table:2_weight_result}
\end{center}
\end{table}

Our BC-LEPI-LSI algorithm outperforms BC-ALI and SOCP in all the test cases as shown in Table \ref{table:2_weight_result}. BC-LEPI-LSI solves all instances to optimality under six minutes on average. SOCP manages to solve all five instances in only one test case with $\beta = 0.95$ and $k=5$. BC-ALI fails to solve within the one-hour time limit in two test cases (i.e., $k=10$ and $\beta = 0.975, 0.99$ respectively) and has significantly longer average runtime than BC-LEPI-LSI. For instance, when $\beta = 0.99$ and $k=10$, the average runtime of BC-ALI is 34 minutes longer than that of BC-LEPI-LSI; SOCP fails to reach optimality in an hour in all five instances of this test case, with a large average optimality gap of 26.8\%. In general, BC-LEPI-LSI explores fewer branch-and-bound nodes than the other two methods and adds fewer cuts than BC-ALI. As $k$ increases, the number of LSIs being added decreases relative to the number of LEPIs.

\subsection{$a$ with multiple weights} 
\label{sect:comp_multi_weights}

In this section, we do not restrict the number of weights in vector $a$. We construct $Q_0$ and compute $\overline{q}$ the same way described in Section \ref{sect:comp_two_weights}. Then we generate $a$ with $a_i \sim U[0.2\overline{q}, \overline{q}]$ for all $i\in\{1,2,\dots, n\}$. Now given $Q = Q_0 + \text{diag}(a)$, we again let $\mu_i \sim U[0.7\sqrt{Q_{ii}}, \sqrt{Q_{ii}}]$ for all $i$. Next, we decompose $a$ into $a^{\text{two}}$ and $a^{\text{res}}$. We let $a_L = \min(a)$, $a_H = \text{median}(a)$ and let $a^{\text{two}}_i = a_L$ when $a_i < a_H$, and $a^{\text{two}}_i = a_H$ otherwise. As a result, $a^{\text{two}} \in\mathbb{R}_+^n$, and $a^{\text{res}} = a - a^{\text{two}}$ is a non-negative vector as well. We let $n=200$, $r=40$, $k\in \{5, 10, 15\}$, and $\Omega = \Phi^{-1}(\beta)$, with $\beta \in\{0.95, 0.975, 0.99\}$. We note that the ALIs are generated with respect to $a$, while LEPIs and LSIs are constructed with respect to $a^{\text{two}}$ in the branch-and-cut algorithms for this set of experiments. \\

\begin{table}[h!]
\begin{center}
\small
{\caption{Computational performance of BC-LEPI-LSI, BC-ALI and SOCP on problem \eqref{eq:mean_risk} with uniformly generated $a$. }}
\begin{tabular}{|c|c|c|c|c|c|c|}
\hline
$\beta$ & $k$ & method & time (s) & end gap & \# nodes & \# cuts \\
\hline
\multirow{9}{*}{0.95}   & \multirow{3}{*}{5}   & BC-LEPI-LSI & $171.9^5$	&	0.0\%	&	6295.2	&	394.6+140.2=534.8\\
                                   &                               & BC-ALI          & $244.1^5$	&	0.0\%	&	10723.2	&	1071.4\\
                                   &                               & SOCP            & $1702.1^5$	&	0.0\%	&	173698.4	&	N/A\\
                                    \cline{2-7}
                                   & \multirow{3}{*}{10} & BC-LEPI-LSI & $577.5^5$	&	0.0\%	&	15612.2	&	1251.0+248.6=1499.6\\
                                   &                               & BC-ALI          & $958.3^5$	&	0.0\%	&	26572.0	&	2656.6\\
                                   &                               & SOCP            & --$^0$	&	5.8\%	&	82144.0	&	N/A\\
                                    \cline{2-7}
                                   & \multirow{3}{*}{15} & BC-LEPI-LSI & $124.3^5$	&	0.0\%	&	3414.6	&	258.0+73.0=331.0\\
                                   &                               & BC-ALI          & $217.2^5$	&	0.0\%	&	5907.4	&	590.0\\
                                   &                               & SOCP            & $2789.6^2$	&	1.7\%	&	79930.6	&	N/A\\
\hline    
\multirow{9}{*}{0.975} & \multirow{3}{*}{5}   & BC-LEPI-LSI & $827.4^5$	&	0.0\%	&	20236.2	&	1176.2+705.2=1881.4\\
                                   &                               & BC-ALI          & $1254.1^5$	&	0.0\%	&	41768.4	&	4175.6\\
                                   &                               & SOCP            & $3349.5^1$	&	31.1\%	&	166471.6	&	N/A\\
                                    \cline{2-7}
                                   & \multirow{3}{*}{10} & BC-LEPI-LSI & $838.9^5$	&	0.0\%	&	21996.8	&	1498.6+648.8=2147.4\\
                                   &                               & BC-ALI          & $1187.4^4$	&	1.2\%	&	40896.2	&	4088.8\\
                                   &                               & SOCP            & --$^0$	&	12.9\%	&	74527.4	&	N/A\\  
                                    \cline{2-7}
                                   & \multirow{3}{*}{15} & BC-LEPI-LSI & $988.6^5$	&	0.0\%	&	19225.8	&	1575.0+324.2=1899.2\\
                                   &                               & BC-ALI          & $1912.7^3$	&	0.8\%	&	43951.8	&	4393.8\\
                                   &                               & SOCP            & --$^0$	&	6.9\%	&	52940.8	&	N/A\\
\hline  
\multirow{9}{*}{0.99}   & \multirow{3}{*}{5}   & BC-LEPI-LSI & $650.1^5$	&	0.0\%	&	20280.0	&	1175.6+715.4=1891.0\\
                                   &                               & BC-ALI          & $947.5^5$	&	0.0\%	&	38703.6	&	3869.4\\
                                   &                               & SOCP            & $3355.9^1$	&	72.3\%	&	259287.8	&	N/A\\
                                    \cline{2-7}
                                   & \multirow{3}{*}{10} & BC-LEPI-LSI & $1855.7^5$	&	0.0\%	&	38376.0	&	2946.4+842.4=3788.8\\
                                   &                               & BC-ALI          & $3068.4^1$	&	4.7\%	&	88635.6	&	8862.6\\
                                   &                               & SOCP            & --$^0$	&	30.8\%	&	75435.8	&	N/A\\
                                    \cline{2-7}
                                   & \multirow{3}{*}{15} & BC-LEPI-LSI & $1354.1^4$	&	0.6\%	&	24718.6	&	2269.6+184.4=2454.0\\
                                   &                               & BC-ALI          & $2434.1^2$	&	1.8\%	&	70845.8	&	7083.4\\
                                   &                               & SOCP            & --$^0$	&	10.5\%	&	65965.4	&	N/A\\
\hline                        
\end{tabular}
\label{table:multi_weight_result}
\end{center}
\end{table}

Table \ref{table:multi_weight_result} summarizes the computational performance of BC-LEPI-LSI, BC-ALI and SOCP on problem \eqref{eq:mean_risk} in which diag($a$) has no restriction on its number of weights. The layout of this table is consistent with Table \ref{table:2_weight_result}. In this set of experiments with general weight vector $a$, our BC-LEPI-LSI algorithm outperforms BC-ALI and SOCP in all the test cases. In Table \ref{table:multi_weight_result}, BC-LEPI-LSI solves to optimality in all but one test case with $\beta = 0.99$ and $k=15$. In this challenging case, BC-LEPI-LSI achieves a small end gap of 0.6\%. BC-ALI and SOCP have longer average running times than BC-LEPI-LSI and fail to solve in many test cases, especially with higher $\beta$ values. For example, when $\beta = 0.99$ and $k=10$, BC-LEPI-LSI solves all five instances with an average runtime of around 30 minutes, whereas BC-ALI fails to attain optimality in four out of five instances of this test case. SOCP fails in all five instances, resulting in a large average end gap of 30.8\%. Overall, the statistics in Table \ref{table:multi_weight_result} are higher than those in Table \ref{table:2_weight_result}, suggesting that the problem instances with general $a$ are computationally more difficult than the instances in which $a$ contains two distinct weights. As before, BC-LEPI-LSI explores fewer branch-and-bound nodes than the other two methods and adds fewer cuts than BC-ALI. The observation that the number of LSIs being added decreases relative to the number of LEPIs as $k$ increases continues to hold in this set of experiments. 
}

\section{Concluding Remarks}
\label{sect:conclude}
In this paper, we tackle the cardinality-constrained concave submodular minimization problem \eqref{prob:cc_concave_sub_min} with two distinct weights. We propose three classes of strong valid linear inequalities, namely the lifted-EPIs, the lower-SIs and the  higher-SIs, for the convex hull of the epigraph for the objective function with a cardinality constraint. These inequalities are computationally effective when incorporated in a branch-and-cut framework as demonstrated by our experiments on a cardinality-constrained mean-risk optimization problem. We further show that the proposed inequalities, together with a single additional inequality and trivial inequalities, fully describe $\conv{\mathcal{P}^2_k}$ when the cardinality upper bound $k$ is set to two. Moreover, the proposed inequalities give rise to valid inequalities for the multi-weighted instances and can be applied in mixed-binary conic optimization. Next we include a few final remarks about the future exploration directions and the associated challenges. \\

The characterization of $\conv{\mathcal{P}^2_2}$ in Section \ref{sect:conv} assumes that Assumption \ref{assumption} holds for $i_0=0$. The convex hull $\conv{\mathcal{P}^2_2}$ becomes more challenging to linearly describe when we lift this assumption, which we illustrate with the example below. 
\begin{example}
Suppose $f(a^\top x) = 64-(a^\top x-8)^2$, $k=2$ and $a = [2, 2, 5,5,5,5,5]$. Note that this function is normalized. In this example, $f(2+5) - f(2) > f(2\cdot 5)/2$, so Assumption  \ref{assumption} is violated. The inequality \[ w \geq -11 + 20x_1+39x_2+35x_3+35x_4+35x_5+35x_6+35x_7+35x_8 \] is facet-defining for $\conv{\mathcal{P}^2_k}$ because it is an extreme ray in its polar. Although $f(0)=0$, there still exists a non-zero constant term in this facet. Thus this inequality does not fall into any of the three classes of homogeneous inequalities we propose. It seems non-trivial to find an explicit specification for the constant term, as well as the remaining coefficients in relation to this constant. 
\end{example}

A natural next step from this paper is to examine $\conv{\mathcal{P}^m_k}$ where $m\geq 3$ or $k\geq 3$. When $k\geq 3$, $\conv{\mathcal{P}^2_k}$ has other types of facets in addition to the three classes of inequalities we propose. Below is an example of such facets. 
\begin{example}
Suppose $f(a^\top x) = 64-(a^\top x -8)^2$, $k=3$ and $a = [6,6,6,6,8,8,8]$. The inequality \[w\geq -\frac{20}{3}x_1- \frac{44}{3}x_2 - \frac{44}{3}x_3 - \frac{44}{3}x_4 - \frac{176}{3}x_5 - \frac{200}{3}x_6 - \frac{200}{3}x_7\] is an extreme ray in the polar of $\conv{\mathcal{P}^m_k}$, and thus a facet. However, this inequality does not belong to any  of the proposed classes of inequalities. 
\end{example}
Despite the challenge of fully characterizing $\conv{\mathcal{P}^m_k}$ for general $m$ and $k$, we may still obtain valid and even facet-defining inequalities for it, by further lifting the proposed inequalities.

\section*{Acknowledgements}
{We thank the editor and the reviewers for the helpful comments that improved this paper. In particular, we thank the reviewer for providing the example in Remark 2.} This research is supported, in part, by NSF grant 2007814 and ONR grant N00014-22-1-2602. This research is also supported in part through the computational resources and staff contributions provided for the Quest high performance computing facility at Northwestern University, which is jointly supported by the Office of the Provost, the Office for Research, and Northwestern University Information Technology.

\bibliography{concave_submodular}{}
\bibliographystyle{apalike}

\end{document}